\documentclass{amsart}

\usepackage{amssymb}
\usepackage{amsmath}
\usepackage{amsthm}
\usepackage{amscd}
\usepackage{bbm}
\usepackage{changepage}
\usepackage{enumerate}
\usepackage{fancybox}
\usepackage{fancyhdr}
\usepackage{float}
\usepackage{hyperref}
\usepackage{marvosym}
\usepackage{mathrsfs}
\usepackage[pdftex]{graphicx}
\usepackage{setspace}
\usepackage{subfig}
\usepackage{theoremref}
\usepackage{verbatim}
\usepackage{xy}

\def\Xint#1{\mathchoice
   {\XXint\displaystyle\textstyle{#1}}%
   {\XXint\textstyle\scriptstyle{#1}}%
   {\XXint\scriptstyle\scriptscriptstyle{#1}}%
   {\XXint\scriptscriptstyle\scriptscriptstyle{#1}}%
   \!\int}
\def\XXint#1#2#3{{\setbox0=\hbox{$#1{#2#3}{\int}$}
     \vcenter{\hbox{$#2#3$}}\kern-.5\wd0}}

\def\dashint{\Xint-}

\hypersetup{
    colorlinks,
    citecolor=green,
    linkcolor=red
  }

\makeatletter
\def\blfootnote{\xdef\@thefnmark{}\@footnotetext}
\makeatother

\newcommand{\xx}{\times}
\newcommand{\C}{\mathbb{C}}
\newcommand{\R}{\mathbb{R}}
\newcommand{\N}{\mathbb{N}} 
\newcommand{\br}[1]{\left( #1 \right)}
\newcommand{\brs}[1]{\left[ #1 \right]}
\newcommand{\norm}[1]{\left\Vert #1 \right\Vert}
\newcommand{\abs}[1]{\left\vert #1 \right\vert}
\newcommand{\lb}[0]{\left\lbrace}
\newcommand{\rb}[0]{\right\rbrace}
\newcommand\BoldSquare{%
  \setlength\fboxrule{1.1pt}\setlength\fboxsep{0pt}\fbox{\phantom{\rule{5pt}{5pt}}}}

\newtheorem{thm}{Theorem}[section]
\newtheorem{prop}{Proposition}[section]
\newtheorem{lem}{Lemma}[section]
\newtheorem{cor}{Corollary}[section]

\theoremstyle{definition}

\newtheorem{deff}{Definition}[section]
\newtheorem{rmk}{Remark}[section]
\newtheorem{example}{Example}[section]
\newtheorem*{notation}{Notation}

\title[Kato Square Root with Potential]{The Kato Square Root Problem for Divergence Form Operators with Potential}
\author{Julian Bailey}

\date{}

\begin{document}

\begin{abstract}
The Kato square root problem for
divergence form elliptic operators with potential $V : \R^{n}
\rightarrow \C$ is the
equivalence statement $\norm{\br{L + V}^{\frac{1}{2}} u}_{2} \simeq \norm{\nabla u}_{2} +
\norm{V^{\frac{1}{2}} u}_{2}$, where 
$L + V := - \mathrm{div} (A \nabla) + V$ and the
perturbation $A$ is an $L^{\infty}$
complex matrix-valued function satisfying an accretivity
condition. This relation is proved for any potential with range
contained in some positive sector and satisfying
$\norm{\abs{V}^{\frac{\alpha}{2}} u}_{2} + \norm{(-\Delta)^{\frac{\alpha}{2}}u}_{2} \lesssim
\norm{\br{\abs{V} - \Delta}^{\frac{\alpha}{2}}u}_{2}$ for all $u \in
D(\abs{V} - \Delta)$ and some $\alpha \in (1,2]$. The class of
potentials that will satisfy such a condition is known to contain the reverse H\"{o}lder
class $RH_{2}$ and
$L^{\frac{n}{2}}\br{\R^{n}}$ in dimension $n > 4$. To prove the Kato estimate with potential, a non-homogeneous version of the framework introduced by A. Axelsson,
S. Keith and A. McIntosh for proving quadratic estimates is
developed. In addition to applying this non-homogeneous framework to the
scalar Kato problem with zero-order potential, it will also
be applied to the Kato
problem for systems of equations with
zero-order potential.
\end{abstract}

\maketitle

\section{Introduction}
\label{sec:intro}

\blfootnote{\textit{Key words and phrases.}
  Kato problem; non-homogeneous; Schr\"{o}dinger operator; divergence form operator;
  potential; quadratic estimates. \\
  Mathematics Subject Classification. Primary 42B37 $\cdot$ Secondary
  35J10 \\
This research was partially supported by the Australian
Research Council through the Discovery Project DP160100941.}

For Hilbert spaces $\mathcal{H}$ and $\mathcal{K}$, let $\mathcal{L}\br{\mathcal{H}, \mathcal{K}}$ denote the
space of bounded linear operators from $\mathcal{H}$ to $\mathcal{K}$
and set $\mathcal{L} \br{\mathcal{H}} := \mathcal{L} \br{\mathcal{H}, \mathcal{H}}$. Fix
$n \in \N^{*} = \N \setminus \lb 0 \rb$ and let $A \in L^{\infty} \br{\R^{n}; \mathcal{L}
  \br{\C^{n}}}$. Consider the sesquilinear form $\mathfrak{l}^{A} :
H^{1}\br{\R^{n}} \xx H^{1} \br{\R^{n}} \rightarrow \C$ defined by
$$
\mathfrak{l}^{A} \brs{u,v} := \int_{\R^{n}} \langle  A(x) \nabla u (x) , \nabla
v(x) \rangle_{\C^{n}}  \, dx
$$
for $u, \, v \in H^{1}\br{\R^{n}}$. Suppose that $\mathfrak{l}^{A}$
satisfies the G{\aa}rding inequality
\begin{equation}
  \label{eqtn:Garding0}
\mathrm{Re} \br{\mathfrak{l}^{A}\brs{u,u}} \geq \kappa^{A} \norm{\nabla u}^{2}_{2}
\end{equation}
for all $u \in H^{1}\br{\R^{n}}$, for some $\kappa^{A} > 0$. A well-known
representation theorem from classical form theory (c.f. {\cite[Thm.~VI.2.1]{kato1980perturbation}})
asserts the existence of an associated operator $L : D
\br{L} \subset L^{2} \br{\R^{n}} \rightarrow L^{2} \br{\R^{n}}$ 
for which
$$
\mathfrak{l}^{A} \brs{u,v} = \langle L u, v \rangle_{2}
$$
for all $v \in H^{1}\br{\R^{n}}$ and $u$ in the domain of
$L$,
$$
D \br{L} = \lb u \in H^{1} \br{\R^{n}} : \exists \, w \in
L^{2} \br{\R^{n}} \ s.t. \ \mathfrak{l}^{A} \brs{u,v} = \langle w , v
\rangle_{2} \ \forall \ v \in H^{1} \br{\R^{n}} \rb.
$$
The operator $L$  is a densely defined maximal accretive operator that
is denoted
$$
L = - \mathrm{div}( A \nabla).
$$
 Due to the accretivity of the operator, it is possible to define a square root operator
$\sqrt{L}$, with core $D \br{L}$, that satisfies 
$\sqrt{L} \cdot \sqrt{L} =
L$. The Kato square root problem asks: what is the domain of definition
of the square root operator? This problem, first posed by Tosio Kato over 50
years ago, was conjectured to have the following solution.

\begin{thm}[Kato Square Root] 
 \label{thm:KatoOriginal} 
 The domain of $\sqrt{L}$ is
 $D(\sqrt{L}) = H^{1}\br{\R^{n}}$
and for any $u \in H^{1}\br{\R^{n}}$
\begin{equation}
  \label{eqtn:OriginalKato}
\norm{\sqrt{L} u}_{2} \simeq \norm{\nabla u}_{2}.
\end{equation}
 \end{thm}

 This long-standing problem withstood
solution until 2002 where it was proved using local T(b) methods by Steve Hofmann, Michael
Lacey and Alan McIntosh in \cite{hofmann2002solution} under the additional assumption of Gaussian
heat kernel bounds and in full generality by
Pascal Auscher, Steve Hofmann, Michael Lacey, Alan McIntosh and
Phillipe Tchamitchian in  \cite{auscher2002solution}.
 We will be interested in an alternate method of proof that was
built from similar principles and appeared a few years later.

Let $\Pi := \Gamma + \Gamma^{*}$ be a Dirac-type operator on a Hilbert
space $\mathcal{H}$ and $\Pi_{B}
:= \Gamma + B_{1} \Gamma^{*} B_{2}$ be a perturbation of $\Pi$ by
bounded operators $B_{1}$ and $B_{2}$. Typically, $\Pi$ is considered
to be a first-order system acting on $\mathcal{H} := L^{2} \br{\R^{n};
\C^{N}}$ for some $n, \, N \in \N^{*}$ and the perturbations $B_{1}$ and
$B_{2}$ are multiplication by matrix-valued functions $B_{1}, \, B_{2}
\in L^{\infty} \br{\R^{n}; \mathcal{L} \br{\C^{N}}}$. In their seminal paper
\cite{axelsson2006quadratic}, A. Axelsson, S. Keith and A. McIntosh
developed a general framework for proving that the perturbed operator
$\Pi_{B}$ possessed a bounded holomorphic functional calculus. This
ultimately amounted to obtaining square function estimates of
the form
\begin{equation}
  \label{eqtn:Intro1}
\int^{\infty}_{0} \norm{Q^{B}_{t} u}^{2} \frac{dt}{t} \simeq  \norm{u}^{2},
\end{equation}
where $Q^{B}_{t} := t \Pi_{B} \br{I + t^{2} \Pi_{B}^{2}}^{-1}$ and $u$ is
contained in the range $\overline{R \br{\Pi_{B}}}$. They
proved that this estimate would follow entirely from a set of simple
conditions imposed upon the operators $\Gamma$, $B_{1}$ and $B_{2}$,
labelled (H1) - (H8). Then, by checking this list of simple
conditions, the Axelsson-Keith-McIntosh framework, or AKM framework
by way of abbreviation, could be used to conclude that the
particular selection of operators
\begin{equation}
  \label{eqtn:OperatorsClassic}
\Gamma := \br{\begin{array}{c c} 
 0 & 0 \\ \nabla & 0
 \end{array}}, \quad B_{1} = I, \quad B_{2} = \br{\begin{array}{c c} 
 I & 0 \\ 0 & A
 \end{array}},
\end{equation}
defined on $L^{2} \br{\R^{n}} \oplus L^{2} \br{\R^{n};\C^{n}}$, would
satisfy \eqref{eqtn:Intro1} and therefore $\Pi_{B}$ would possess a bounded
holomorphic functional calculus. The Kato square root estimate then
followed almost trivially from this.

\vspace*{0.1in}

Many classical problems from harmonic analysis will have a direct
counterpart in the Schr\"{o}dinger operator setting. Adhering with
this theme, one can consider the Kato square root problem with
potential. Let $V : \R^{n} \rightarrow \C$ be a measurable function
that is contained in $L^{1}_{loc}\br{\R^{n}}$. 
$V$ can be
viewed as a densely defined closed multiplication operator on
$L^{2}\br{\R^{n}}$ with domain
$$
D \br{V} = \lb u \in L^{2}\br{\R^{n}} : V \cdot u \in L^{2} \br{\R^{n}} \rb.
$$
The density of $D(V)$ follows from the condition $V \in L^{1}_{loc}(\R^{n})$. 
Define the subspace
\begin{equation}
  \label{eqtn:H1V}
H^{1,V} \br{\R^{n}} := H^{1}\br{\R^{n}} \cap D
\br{V^{\frac{1}{2}}} := \lb u \in
H^{1} \br{\R^{n}} : V^{\frac{1}{2}} \cdot u \in L^{2} \br{\R^{n}} \rb.
\end{equation}
Here
the complex square root $V^{\frac{1}{2}}$ is defined via the principal
branch cut along the negative real axis. It is easy to see that
$H^{1,V}(\R^{n})$ is dense in $L^{2}\br{\R^{n}}$ since
$C^{\infty}_{0}(\R^{n}) \subset H^{1,V}(\R^{n})$.

Let $A \in L^{\infty}\br{\R^{n};\mathcal{L} \br{\C^{n}}}$ be as before with
\eqref{eqtn:Garding0} satisfied for some $\kappa^{A} > 0$. Consider the sesquilinear form $\mathfrak{l}_{V}^{A} : H^{1,V}
\br{\R^{n}} \xx H^{1,V} \br{\R^{n}} \rightarrow \C$
defined through
$$
\mathfrak{l}_{V}^{A} \brs{u,v} := \mathfrak{l}^{A} \brs{u,v} + \int_{\R^{n}}  V(x)
u(x) \overline{v}(x)  \, dx
$$
for $u, \, v \in H^{1,V} \br{\R^{n}}$. Suppose that there exists
some $\kappa_{V}^{A} > 0$ for which
$\mathfrak{l}_{V}^{A}$ satisfies the associated G{\aa}rding inequality
\begin{equation}
  \label{eqtn:Garding}
\mathrm{Re} \br{\mathfrak{l}_{V}^{A} \brs{u,u}} \geq \kappa_{V}^{A}
\br{\norm{ V^{\frac{1}{2}} u}_{2}^{2} + \norm{\nabla u}_{2}^{2}}
\end{equation}
for all $u \in H^{1,V} \br{\R^{n}}$. 

\begin{rmk}
If the range of $V$ is contained in some sector
$$
S_{\mu^{+}} := \lb z \in \C \cup \lb \infty \rb :
\abs{\mathrm{arg}\br{z}} \leq \mu \ or \ z = 0, \, \infty \rb
$$
for some $\mu \in [0,\frac{\pi}{2})$, then \eqref{eqtn:Garding} will follow
automatically from \eqref{eqtn:Garding0}.
\end{rmk}

Once again, the accretivity of
$\mathfrak{l}_{V}^{A}$  implies the existence of a maximal accretive operator
associated with this form denoted by
$$
L + V = - \mathrm{div} (A \nabla) + V
$$
and defined on
$$
D(L + V) = \lb u \in H^{1,V}(\R^{n}) : \exists \, w \in L^{2}(\R^{n}) \
s.t. \ \mathfrak{l}^{A}_{V}[u,v] = \langle w, v \rangle_{2} \ \forall \ v
\in H^{1,V}(\R^{n}) \rb.
$$

For
  $\alpha \in [1,2]$, define $\mathcal{W}_{\alpha}$ to be the
  class of all $V \in L^{1}_{loc}(\R^{n})$ for which 
  $$
\brs{V}_{\alpha} := \sup_{u \in D(\abs{V} - \Delta)} \frac{\norm{\abs{V}^{\frac{\alpha}{2}}
    u}_{2} + \norm{\br{- \Delta}^{\frac{\alpha}{2}}u}_{2}}{\norm{\br{\abs{V} - \Delta}^{\frac{\alpha}{2}} u}_{2}} < \infty.
$$
As will be proved in Section \ref{subsec:W}, the collection of
potential classes $\lb \mathcal{W}_{\alpha} \rb_{\alpha \in [1,2]}$ is
decreasing. The largest class $\mathcal{W}_{1}$ consists of all
locally integrable potentials with no additional restrictions and the
smallest class $\mathcal{W}_{2}$ contains $RH_{2}$ in any dimension
and $L^{\frac{n}{2}}(\R^{n})$ in dimension $n > 4$.
In this paper, our aim is to prove the
potential dependent Kato estimate as presented in the following
theorem.

\begin{thm}[Kato with Potential]
  \label{thm:KatoPotential}
Let $V \in \mathcal{W}_{\alpha}$ for some $\alpha \in (1,2]$ and $A \in L^{\infty} \br{\R^{n}; \mathcal{L}
  \br{\C^{n}}}$. Suppose that the G{\aa}rding inequalities
\eqref{eqtn:Garding0} and \eqref{eqtn:Garding} are both satisfied
with constants $\kappa^{A} > 0$ and $\kappa_{V}^{A} > 0$
respectively. There exists a constant $C_{V} > 0$ such that
 \begin{equation} 
 \label{eqtn:KatoPotential} \tag{KP}
C_{V}^{-1} \br{\norm{ V^{\frac{1}{2}} u}_{2} + \norm{\nabla u}_{2}} \leq
 \norm{\sqrt{L + V} u}_{2} \leq C_{V}
 \br{\norm{V^{\frac{1}{2}} u}_{2} + \norm{\nabla u}_{2}}
\end{equation}
for all $u \in H^{1,V}\br{\R^{n}}$. Moreover, the constant
$C_{V}$ depends on $V$ and $\alpha$ through
$$
C_{V} = \tilde{C}_{V} \br{\alpha - 1}^{-1} (1 + \brs{V}_{\alpha}^{2}),
$$
where $\tilde{C}_{V}$ only depends on $V$ through $\kappa_{V}^{A}$ and
is independent of $\alpha$.
\end{thm}

This theorem is actually a statement concerning the domain of the
square root operator $\sqrt{L + V}$. Indeed,
\eqref{eqtn:KatoPotential} implies the equality
$$
D(\sqrt{L + V}) = H^{1,V}(\R^{n}).
$$
\begin{rmk}
  \label{rmk:AInf}
  The above theorem tells us, in particular, that the Kato estimate
 with potential, \eqref{eqtn:KatoPotential}, is valid for any
 potential $V$ with range contained in $S_{\mu+}$ for some $\mu \in
 [0,\frac{\pi}{2})$ and such that $\abs{V}$ is either contained in $RH_{2}$ for any
 dimension or in $L^{\frac{n}{2}}(\R^{n})$ in dimension $n > 4$.
  In
 {\cite[pg.~3]{auscher2007maximal}}, it is suggested but not
 proved
 that $\brs{V}_{\alpha} < \infty$ will occur
 provided that $\abs{V} \in RH_{q}$ with $q > \alpha$. If this is
 indeed the case then Theorem \ref{thm:KatoPotential} will imply that
 the Kato estimate is true for any potential
 $V$ with range contained in $S_{\mu+}$ for some $\mu \in
 [0,\frac{\pi}{2})$ and
 with $\abs{V}$ contained in $RH_{1} = A_{\infty}$.
\end{rmk}
 
In direct analogy to the potential free case, the Kato problem with
potential will be solved by constructing appropriate potential
dependent Dirac-type operators and demonstrating that they retain
a bounded holomorphic functional calculus under perturbation. In
particular, this strategy will be applied to the Dirac-type operator
\begin{equation} 
 \label{eqtn:OperatorsPotential} 
 \Pi_{\abs{V}^{\frac{1}{2}}} := \Gamma_{\abs{V}^{\frac{1}{2}}} + \Gamma_{\abs{V}^{\frac{1}{2}}}^{*} := \br{\begin{array}{c c c} 
                                                 0 & 0 & 0 \\
                                                 \abs{V}^{\frac{1}{2}} & 0 &
                                                                       0
                                                 \\
                                                 \nabla & 0 & 0
 \end{array}} + \br{\begin{array}{c c c} 
                      0 & \abs{V}^{\frac{1}{2}} & - \mathrm{div} \\
                      0 & 0 & 0 \\
                      0 & 0 & 0
 \end{array}}
 \end{equation}
defined on $L^{2}(\R^{n}) \oplus L^{2}\br{\R^{n}} \oplus L^{2}
\br{\R^{n} ; \C^{n}}$, under the perturbation
\begin{equation}
  \label{eqtn:OperatorsPotential2}
  B_{1} = I, \quad B_{2} := \br{\begin{array}{c c c} 
         I & 0 & 0 \\
         0 & e^{i \cdot \mathrm{arg} V} & 0 \\
         0 & 0 & A
 \end{array}}.
\end{equation}
It should be observed that the operator
$\Gamma_{\abs{V}^{\frac{1}{2}}}$ is not first-order homogeneous due to
the presence of the zero-order potential term. It will therefore not
necessarily satisfy the two conditions of the original AKM framework that
are intended to capture the first-order homogeneity property, the cancellation and coercivity conditions, (H7)
and (H8). As such, the original AKM framework cannot be directly applied. The key
difficulty in proving our result is then to alter
the original framework in order to allow for such operators. In
particular, a non-homogeneous version of the Axelsson-Keith-McIntosh
framework will be developed to handle operators of the form
\begin{equation}
  \label{eqtn:NonHomog}
\Gamma_{J} := \br{\begin{array}{c c c} 
                    0 & 0 & 0 \\
                    J & 0 & 0 \\
                    D & 0 & 0
 \end{array}},
\end{equation}
where $D$ is a homogeneous first-order differential operator and $J$
is a possibly non-homogeneous differential operator of order less than
or equal to one. The technical challenge presented
by the inclusion of the non-homogeneous part $J$ will be overcome by separating our
square function norm into components and demonstrating that the
non-homogeneous term will allow for the first two components to be bounded
while the third component can be bound using an argument 
similar to the classical argument of \cite{axelsson2006quadratic}.

Since the operator $\Gamma_{J}$ is of a more general form than
$\Gamma_{\abs{V}^{\frac{1}{2}}}$, the non-homogeneous AKM framework that we develop will
have applications not confined to zero-order scalar
potentials. Indeed, the non-homogeneous framework will also be used to
prove Kato estimates for systems of equations with
zero-order potential and for scalar equations with first-order
potentials.

The structure of this paper is as follows. Section
\ref{sec:Holomorphic} is quite classical in nature. It provides a brief survey of
the natural functional calculus for bisectorial operators. It also
describes how
to prove the boundedness of holomorphic functional calculus from
square function estimates for a bisectorial operator. These classical proofs are repeated
here to track the dependence of the estimates on relevant
constants. This will be required to determine the potential dependence
of the constants in
Theorem \ref{thm:KatoPotential}. Section \ref{sec:NonHomog}
describes the non-homogeneous AKM framework and
states the main results associated with it. Section \ref{sec:Square} contains most of the technical
machinery and is dedicated to a proof of our main result. Section
\ref{sec:Applications} will apply the non-homogeneous AKM framework to
the scalar Kato problem with potential, the Kato problem for systems
with zero-order potential and the scalar Kato problem with first-order
potential. It is here that a proof of Theorem \ref{thm:KatoPotential}
will be completed. 
Finally, in Section \ref{sec:Literature}, we will provide a meta-discussion
on the proof techniques used and compare our work with what has been
previously accomplished on non-homogeneous Kato-type
estimates.

\subsection{Acknowledgements} This paper is part of my PhD thesis
undertaken at the Australian National University. I am very thankful to my 
supervisor Pierre Portal for his numerous suggestions and corrections
and for the encouragement that made this article possible. I am
indebted to El Maati Ouhabaz for the discussions surrounding this
project and for years ago suggesting to Pierre that Kato estimates for
Schr\"{o}dinger-type operators would be of particular interest to the community,
thus initiating this line of research.
I am also grateful
to my second supervisor Adam Sikora for his sage
advice, in particular with regards to Proposition \ref{prop:Ln2}. A
part of this paper was written while visiting him at Macquarie 
University. 

While this paper was in its final stages of preparation,
it was found that Andrew Morris and Andrew Turner from the University
of Birmingham were also working on the Kato problem with potential on $\R^{n}$. After meeting them and discussing their 
research, it appears that the two approaches differ in their
assumptions, their results
and, more substantially, their proofs.

I would like to thank Moritz Egert for spotting an error in a previous
version of this article and for his comments on Proposition \ref{lem:Independence}.
Finally, I would like to thank the anonymous referees of
previous versions of this article for providing such detailed and thoughtful
critiques. Reflecting on the comments provided in their reports led me to several significant improvements.

\section{Preliminaries}
\label{sec:Holomorphic}

Let's outline the construction of the natural functional calculus
associated with a bisectorial operator. The treatment of
functional calculi found here follows closely to
\cite{haase2006functional} with significant detail
omitted. Appropriate changes are made to account for the fact that we
consider bisectorial operators instead of sectorial operators. Other
thorough treatments of functional calculus for sectorial operators can
be found in \cite{albrecht1996operator} or \cite{hytonen2017analysis}.

 For $\mu \in [0,\pi)$ define the open and closed sectors
$$
S^{o}_{\mu+} := \left\lbrace \begin{array}{c c} 
 \lb z \in \C \setminus \lb 0 \rb : \abs{\mathrm{arg} \br{
                               z}} < \mu \rb & \mu \in (0,\pi) \\
                                               (0,\infty) & \mu = 0
 \end{array} \right.
$$
and
$$
S_{\mu+} := \left\lbrace \begin{array}{c c} 
 \lb z \in \C \cup \lb \infty \rb: \abs{\mathrm{arg} \br{
                               z}} \leq \mu \ or \ z = 0, \ \infty \rb & \mu \in (0,\pi) \\
                                               \left[ 0 , \infty \right] & \mu = 0.
 \end{array} \right.
$$
Then, for $\mu \in \left[0,\frac{\pi}{2}\right)$, define the open and closed bisectors
$$
S^{o}_{\mu} := \br{S^{o}_{\mu+}} \cup \br{-S^{o}_{\mu+}}
$$
and
$$
S_{\mu} := \br{S_{\mu+}} \cup \br{- S_{\mu +}}
$$
respectively.  Throughout this section we consider bisectorial operators defined on a
Hilbert space $\mathcal{H}$ with norm and inner product denoted by
$\norm{\cdot}$ and $\langle \cdot, \cdot \rangle$ respectively. 

\begin{deff}[Bisectorial Operator]
  \label{def:Bisectorial}
   A linear operator $T : D \br{T} \subseteq \mathcal{H} \rightarrow \mathcal{H}$ is said
   to be $\omega$-bisectorial for $\omega \in \left[0,\frac{\pi}{2}\right)$ if
   the spectrum $\sigma \br{T}$ is contained in the bisector
   $S_{\omega}$ and if for any $\mu \in \br{\omega, \frac{\pi}{2}}$, there exists $C_{\mu} > 0$ such that the resolvent
   bound
   \begin{equation}
    \label{eqtn:ResolventEst}
\abs{\zeta} \norm{\br{\zeta I - T}^{-1}} \leq C_{\mu}
\end{equation}
holds for all $\zeta \in \C \setminus S_{\mu}$. $T$ is said to be
bisectorial if it is $\omega$-bisectorial for some $\omega \in \left[0, \frac{\pi}{2}\right)$.
\end{deff}

Sectorial operators are defined identically except with the sector
$S_{\mu+}$ performing the role of the bisector $S_{\mu}$. An important fact concerning bisectorial
operators is the following decomposition result.

\begin{prop}[{\cite[Thm.~3.8]{cowling1996banach}}]
 \label{prop:Decomposition} 
Let $T : D \br{T} \subset \mathcal{H} \rightarrow \mathcal{H}$ be a bisectorial
operator. Then $T$ is necessarily densely defined and
 the Hilbert space $\mathcal{H}$ admits the following decomposition
 $$
\mathcal{H} = N \br{T} \oplus \overline{R\br{T}}.
 $$
 \end{prop}

Let $T$ be an $\omega$-bisectorial operator for $\omega \in \left[0,
  \frac{\pi}{2} \right)$ and $\mu \in \br{\omega, \frac{\pi}{2}}$. Let $\mathcal{M} \br{S^{o}_{\mu}}$ denote
the algebra of all meromorphic functions on the open bisector
$S^{o}_{\mu}$ and define the following subalgebras,
$$
H\br{S^{o}_{\mu}} := \lb f \in \mathcal{M} \br{S^{o}_{\mu}} : f \
holomorphic \ on \ S^{o}_{\mu} \rb,
$$
$$
H^{\infty} \br{S^{o}_{\mu}} := \lb f \in H \br{S^{o}_{\mu}} :
\norm{f}_{\infty} := \sup_{z  \in S^{o}_{\mu}} \abs{f(z)} < \infty \rb
$$
and
$$
H^{\infty}_{0} \br{S^{o}_{\mu}} :=  \lb f \in H^{\infty} \br{S^{o}_{\mu}} :
\exists  \ C, \, \alpha > 0 \ s.t. \ \abs{f(z)} \leq C \cdot
\frac{\abs{z}^{\alpha}}{1 + \abs{z}^{2 \alpha}} \ \forall \ z \in S^{o}_{\mu}\rb.
$$
For any $f \in H^{\infty}_{0} \br{S^{o}_{\mu}}$, one can associate an
operator $f(T)$ as follows. For $u \in \mathcal{H}$, define
$$
f\br{T}u := \frac{1}{2 \pi i} \oint_{\gamma} f \br{z} \br{z I -
  T}^{-1} u \, d z,
$$
where the curve
$$
\gamma := \lb \pm r e^{\pm i \nu} : 0 \leq r < \infty \rb
$$
for some $\nu \in (\omega, \mu)$ is traversed anticlockwise. Using the
resolvent bounds of our operator in combination with the size
estimates for functions in $H^{\infty}_{0}(S^{o}_{\mu})$, it can
easily be shown that the operator $f(T)$ is well-defined. Moreover, it
can also be proved that this association constitutes an algebra homomorphism.

\begin{thm}[{\cite[Lem.~2.3.1]{haase2006functional}}]
 \label{thm:PsiFunctional} 
 The map
$
\Phi_{0}^{T} : H^{\infty}_{0} \br{S^{o}_{\mu}} \rightarrow \mathcal{L} \br{\mathcal{H}}
$
defined through
$$
\Phi_{0}^{T}(f) := f\br{T}
$$
is a well-defined algebra homomorphism. Moreover, it is independent
of the value of $\nu$.
\end{thm}

 Since the functions in $H^{\infty}_{0} \br{S^{o}_{\mu}}$ approach
 zero at the origin we should expect that the null space of the newly
 formed operator will be larger than the null space of the original
 operator. This is indeed the case as stated in the below proposition.

\begin{prop}[{\cite[Thm.~2.3.3]{haase2006functional}}]
 \label{prop:NullSpaceInclusion} 
 For a bisectorial operator $T : D \br{T} \subseteq \mathcal{H}
 \rightarrow \mathcal{H}$, the null-space inclusion
 $$
N \br{T} \subseteq N \br{f \br{T}}
$$
holds for all $f \in H^{\infty}_{0} \br{S^{o}_{\mu}}$.
 \end{prop}

 Define the subalgebra of functions
$$
\mathcal{E} \br{S^{o}_{\mu}} := H^{\infty}_{0} \br{S^{o}_{\mu}} \oplus \langle
\br{z + i}^{-1} \rangle  \oplus
\langle 1 \rangle.
$$
$\Phi_{0}^{T}$ has an extension
$$
\Phi_{p}^{T} : \mathcal{E}\br{S^{o}_{\mu}} \rightarrow \mathcal{L} \br{H}
$$
defined through 
$$
\Phi_{p}^{T}(g) := g(T) := f(T) + c \cdot 
\br{T + i}^{-1} + d \cdot I
$$
for $g = f + c \cdot \br{z + i}^{-1} + d  \in \mathcal{E}
\br{S^{o}_{\mu}}$, where $f \in H^{\infty}_{0} \br{S^{o}_{\mu}}$ and
$c, \, d \in \C$.

\begin{thm}[{\cite[Thm.~2.3.3]{haase2006functional}}]
 \label{thm:PrimaryFunctional} 
 The map $\Phi_{p}^{T}$ is an algebra homomorphism called the primary
 functional calculus associated with $T$.
 \end{thm}

This map can be extended once more through the  process of
regularization. A function $f \in
\mathcal{M}\br{S^{o}_{\mu}}$ is said to be regularizable with respect
to the primary functional calculus $\Phi_{p}^{T} : \mathcal{E}
\br{S^{o}_{\mu}} \rightarrow \mathcal{L} \br{\mathcal{H}}$ if there
exists $e \in \mathcal{E} \br{S^{o}_{\mu}}$ such that $e \br{T}$ is
injective and $e \cdot f \in \mathcal{E} \br{S^{o}_{\mu}}$. The
notation $\mathcal{E}\br{S^{o}_{\mu}}_{r}$ will be used to denote the
algebra of regularizable functions. Let $\mathcal{C} \br{\mathcal{H}}$
denote the set of closed operators from $\mathcal{H}$ to itself. Then define the extension
$$
\Phi^{T} : \mathcal{E} \br{S^{o}_{\mu}}_{r} \rightarrow \mathcal{C} \br{H}
$$
through
$$
\Phi^{T} \br{f} := f(T) := \Phi_{p}^{T}(e)^{-1} \cdot \Phi_{p}^{T}(e \cdot f)
$$
for $f \in  \mathcal{E} \br{S^{o}_{\mu}}_{r}$ and $e \in
\mathcal{E} \br{S^{o}_{\mu}}$ a
regularizing function for $f$. This definition is independent of the
chosen regularizer $e$ for $f$ and therefore $\Phi^{T}$ is well-defined. We have the following important theorem
that establishes the desired properties of a functional calculus for
this extension. Thus the map $\Phi^{T}$ is known as the natural
functional calculus for the operator $T$.

\begin{thm}[{\cite[Thm.~1.3.2]{haase2006functional}}]
 \label{thm:FundFunctCalc} 
Let $T$ be an $\omega$-bisectorial operator on a Hilbert space
$\mathcal{H}$ for some $\omega \in \left[0,\frac{\pi}{2}\right)$. Let
$\mu \in \br{\omega, \frac{\pi}{2}}$. The following assertions hold.
\begin{enumerate}
  \item $\mathbf{1}\br{T} = I$ and $\br{z}\br{T} = T$, where
    $\mathbf{1} : S^{o}_{\mu} \rightarrow \C$ is the constant function defined by $\mathbf{1}(z) :=
    1$ for $z \in S^{o}_{\mu}$.
 \item  Let $f, \, g \in \mathcal{E} \br{S^{o}_{\mu}}_{r}$. Then
   $$
f(T) + g(T) \subset \br{f + g}(T), \qquad f(T) g(T) \subset
\br{f \cdot g}(T)
$$
and $D \br{f(T) g(T)} = D
\br{\br{f \cdot g}(T)} \cap D \br{g(T)}$. One will
have equality in these relations if $g(T) \in \mathcal{L}(H)$.
\end{enumerate}
\end{thm}

The ensuing definition plays a vital role in the
solution method to the Kato square root problem using the AKM framework.

\begin{deff} 
 \label{def:BoundedHCalc} 
Let $0 \leq \omega < \mu < \frac{\pi}{2}$. An $\omega$-bisectorial operator
$T : D(T) \subset \mathcal{H} \rightarrow \mathcal{H}$ is said to have
a bounded $H^{\infty}\br{S^{o}_{\mu}}$-functional calculus if there
exists $c > 0$ such that 
\begin{equation}
  \label{eqtn:BoundedHolomorphic}
\norm{f \br{T}} \leq c \cdot \norm{f}_{\infty}
\end{equation}
for all $f \in
H^{\infty}_{0} \br{S^{o}_{\mu}}$. $T$ is said to have a bounded holomorphic functional calculus if it
has a bounded $H^{\infty}\br{S^{o}_{\mu}}$-functional calculus for some $\mu$.
\end{deff}

\begin{rmk}
  \label{rmk:Injective}
Note that a more intuitive definition for a bounded $H^{\infty}
\br{S^{o}_{\mu}}$-functional calculus would be to require that
\eqref{eqtn:BoundedHolomorphic} hold for all $f \in H^{\infty}
\br{S^{o}_{\mu}}$. Unfortunately at this stage it is impossible to
ascertain whether
 $H^{\infty} \br{S^{o}_{\mu}} \subset  \mathcal{E}
 \br{S^{o}_{\mu}}_{r}$. When this inclusion does not hold, the
 operator $f(T)$ will not be
 well-defined for all $f \in H^{\infty} \br{S^{o}_{\mu}}$. If $T$ so happens
 to be injective, then each $f \in H^{\infty} \br{S^{o}_{\mu}}$ is in fact
 regularizable by $z \br{1 + z^{2}}^{-1}$ and the estimate
 \eqref{eqtn:BoundedHolomorphic} makes sense for all $f \in H^{\infty}
 \br{S^{o}_{\mu}}$. Fortunately, in this situation the two definitions
 coincide. That is, \eqref{eqtn:BoundedHolomorphic} will be true
 for all $f \in H^{\infty}_{0} \br{S^{o}_{\mu}}$ if and only if it is
 true for all $f \in H^{\infty} \br{S^{o}_{\mu}}$ when $T$ is injective. 
\end{rmk}

\begin{deff}[Square Function Norms]
 \label{def:SquareFunctionNorms} 
 Let $\psi \in H^{\infty}_{0}\br{S^{o}_{\mu}}$.  For $t > 0$, define $\psi_{t} :
 S^{o}_{\mu} \rightarrow \C$ to
 be the function $\psi_{t}(z) := \psi(tz)$ for $z \in S^{o}_{\mu}$. The square function
 norm associated with $\psi$ and $T$ is defined through
 $$
\norm{u}_{\psi,T} := \br{\int^{\infty}_{0} \norm{\psi_{t}(T)u}^{2} \frac{dt}{t}}^{\frac{1}{2}}
$$
for $u \in \mathcal{H}$. Let $q : S^{o}_{\mu} \rightarrow \C$ be the
function given by
$$
q(z) := \frac{z}{1 + z^{2}}, \quad z \in S^{o}_{\mu}.
$$
$\norm{\cdot}_{q,T}$ is called the canonical square function norm
for the operator $T$.
\end{deff}

For injective $T$, true to its name, the square function norms
$\norm{\cdot}_{\psi,T}$, for $\psi \in H^{\infty}_{0}\br{S^{o}_{\mu}}$  not identically
  equal to zero on either $S^{o}_{\mu+}$ or $\br{-S^{o}_{\mu+}}$,
are indeed norms on $\mathcal{H}$. For non-injective $T$, however, they
are at most seminorms on $\mathcal{H}$ and will only be norms
following a restriction to the subspace $\overline{R(T)}$.

\begin{deff}[Square Function Estimates]
  \label{def:SectSqFuncEst}
  A bisectorial operator $T$ on a Hilbert space $\mathcal{H}$ is said
  to satisfy square function estimates if there exists a constant $C_{SF} > 0$ such that
\begin{equation}
  \label{eqtn:def:SectSqFuncEst}
  C^{-\frac{1}{2}}_{SF} \cdot \norm{u} \leq \norm{u}_{q,T} \leq C_{SF}^{\frac{1}{2}} \cdot \norm{u}
\end{equation}
for all $u \in \overline{R \br{T}}$.
\end{deff}

The above definition is the same as saying that the canonical square
function norm $\norm{\cdot}_{q,T}$ is norm equivalent to $\norm{\cdot}_{\mathcal{H}}$ when restricted to
the Hilbert subspace $\overline{R \br{T}}$.

\begin{rmk}
  \label{rmk:Arbitrary}
  The use of the canonical norm $\norm{\cdot}_{q,T}$ in the above
  definition of square function estimates is somewhat
  arbitrary. Indeed, it can be swapped with $\norm{\cdot}_{\psi,T}$
  for any $\psi \in H^{\infty}_{0} \br{S^{o}_{\mu}}$ not identically
  equal to zero on either $S^{o}_{\mu+}$ or $\br{-S^{o}_{\mu+}}$. This
  follows from the equivalence of these two norms as stated in
  \cite[Thm.~7.3.1]{haase2006functional}. The implicit constant in the
  norm equivalence will depend on the function $\psi$ under consideration.
  \end{rmk}

\begin{prop}[Resolution of the Identity]
 \label{prop:Resolution} 
 For $\psi \in H^{\infty}_{0}\br{S^{o}_{\mu}}$ and any $u \in \mathcal{H}$,
 \begin{equation} 
 \label{eqtn:prop:Resolution} 
 c_{\psi} \br{I - \mathbb{P}_{N(T)}}u = \int^{\infty}_{0}
 \psi_{t}(T)^{2} u \, \frac{dt}{t},
\end{equation}
where $\mathbb{P}_{N(T)}$ denotes the projection operator onto the
subspace $N \br{T}$ and
$$
c_{\psi} := \int^{\infty}_{0} \psi(t)^{2} \,  \frac{dt}{t}.
$$
\end{prop}

\begin{proof}  
 Equality follows from Proposition \ref{prop:NullSpaceInclusion} for
 $u \in N \br{T}$. For $u \in \overline{R \br{T}}$ this is given by
 Theorem $5.2.6$ of \cite{haase2006functional} in the sectorial
 case. The bisectorial case can be proved similarly.
\end{proof}

\begin{cor}
  \label{cor:SA}
Suppose that $T$ is self-adjoint and $\psi \in H^{\infty}_{0}\br{S^{o}_{\mu}}$. Then for any $u \in \mathcal{H}$,
  $$
\int^{\infty}_{0} \norm{\psi_{t}(T) u}^{2} \frac{dt}{t} \leq c_{\psi} \norm{u}^{2}
$$
where $c_{\psi}$ is as defined in the previous proposition.
Equality will hold if $u \in \overline{R \br{T}}$.
\end{cor}

\begin{proof}
  As $T$ is self-adjoint, if follows from the definition of $\psi_{t}(T)$
  that it must also be self-adjoint. On expanding the square function norm,
  \begin{align*}\begin{split}  
 \int^{\infty}_{0} \norm{\psi_{t}(T)u}^{2}\frac{dt}{t} &=
 \int^{\infty}_{0} \langle \psi_{t}(T) u, \psi_{t}(T) u \rangle \frac{dt}{t}
 \\
 &= \left\langle u , \int^{\infty}_{0} \psi_{t}(T)^{2} u \frac{dt}{t}
 \right\rangle.
\end{split}\end{align*}
The previous proposition then gives
\begin{align*}\begin{split}  
 \int^{\infty}_{0} \norm{\psi_{t}(T)u}^{2} \frac{dt}{t} &= \left\langle
   u, c_{\psi} \br{I - \mathbb{P}_{N(T)}} u \right\rangle \\
 &\leq c_{\psi} \norm{u}^{2}.
\end{split}\end{align*}
Equality will clearly hold in the above if $u \in \overline{R(T)}$.
 \end{proof}

\begin{thm}
  \label{thm:BoundedHolomorphic}
Let $T$ be an $\omega$-bisectorial operator on $\mathcal{H}$ for
$\omega \in \left[0,\frac{\pi}{2}\right)$.  Suppose that $T$ satisfies square function
  estimates with constant $C_{SF} > 0$. Then $T$ must have a bounded 
$H^{\infty}\br{S^{o}_{\mu}}$-holomorphic functional calculus for any
$\mu \in \br{\omega, \frac{\pi}{2}}$. In particular, there exists a constant $c > 0$,
independent of $T$, such that
$$
\norm{f(T)} \leq \br{c \cdot C_{SF} \cdot C_{\mu}} \cdot \norm{f}_{\infty}
$$
for all $f \in  H^{\infty}_{0} \br{S^{o}_{\mu}}$, where $C_{\mu} > 0$
is the constant from the resolvent estimate \eqref{eqtn:ResolventEst}.
\end{thm}

\begin{proof}
  Let $f \in H^{\infty}_{0} \br{S^{o}_{\mu}}$.  For $u \in N \br{T}$,
  the bound
  \begin{equation}
    \label{eqtn:BoundedHolomorphic0}
    \norm{f \br{T} u} \leq \br{c \cdot C_{SF} \cdot C_{\mu} \cdot \norm{f}_{\infty}} \cdot \norm{u}
  \end{equation}
  follows trivially from Proposition
  \ref{prop:NullSpaceInclusion} for any $c > 0$. Fix $u \in \overline{R
    \br{T}}$. On applying the lower square function estimate to $f(T) u
  \in \overline{R \br{T}}$,
  \begin{align*}\begin{split}  
 \norm{f(T) u}^{2} &\leq C_{SF} \int_{0}^{\infty}
 \norm{q_{s}(T) f\br{T} u}^{2} \frac{ds}{s}  \\
 &= 2 C_{SF} \int^{\infty}_{0} \norm{q_{s}(T) f\br{T}
   \int_{0}^{\infty} \br{q_{t}(T)}^{2} u \frac{dt}{t} }^{2} \frac{ds}{s}
  \\
 &\leq 2 C_{SF} \int^{\infty}_{0} \br{\int^{\infty}_{0}
   \norm{q_{s}(T) f\br{T} q_{t}(T)} \norm{q_{t}(T) u}
  \frac{dt}{t}  }^{2} \frac{ds}{s},
\end{split}\end{align*}
where in the second line we used the resolution of the identity
Proposition \ref{prop:Resolution}.
The Cauchy-Schwarz inequality then gives
\begin{align}\begin{split}
    \label{eqtn:CauchySchwarz}
 \norm{f \br{T} u}^{2} &\leq 2 C_{SF} \int^{\infty}_{0}
 \br{\int^{\infty}_{0} \norm{q_{s}(T) f \br{T} q_{t}(T)} \frac{dt}{t}}
\br{  \int^{\infty}_{0} \norm{q_{s}(T) f \br{T}
   q_{t}(T)} \norm{q_{t}(T) u}^{2} \frac{dt}{t}} \frac{ds}{s}.
\end{split}\end{align}
From the homomorphism property for the $H^{\infty}_{0}\br{S^{o}_{\mu}}$-functional calculus,
$$
q_{s}(T) f
\br{T} q_{t}(T) = \br{q_{s} \cdot f \cdot q_{t}}
\br{T}.
$$
 Since our operator $T$ satisfies resolvent bounds with constant
 $C_{\mu} > 0$,
\begin{align*}\begin{split}  
 \norm{q_{s}(T) f(T) q_{t}(T)} &= \norm{\br{q_{s}
     \cdot f \cdot q_{t}}\br{T}} \\
 &= \frac{1}{2 \pi} \norm{\int_{\gamma} \br{q_{s} \cdot f \cdot q_{t}}(z)
   \br{T - z I}^{-1} dz} \\
 &\leq \frac{C_{\mu}}{2 \pi} \cdot \norm{f}_{\infty} \cdot \int_{\gamma}
 \abs{q_{s}(z)} \abs{q_{t}(z)} \frac{\abs{dz}}{\abs{z}}.
\end{split}\end{align*}
On noting that $q \in H^{\infty}_{0} \br{S^{o}_{\mu}}$,
$$
 \norm{q_{s}(T) f \br{T} q_{t}(T)} \leq
c \cdot C_{\mu} \cdot \norm{f}_{\infty} \cdot \int_{\gamma} \frac{\abs{s
     z}^{\alpha}}{1 + \abs{s z}^{2 \alpha}} \frac{\abs{t
     z}^{\alpha}}{1 + \abs{t z}^{2 \alpha}} \frac{\abs{dz}}{\abs{z}}
 $$
 for some $\alpha > 0$ and constant $c > 0$ independent of $T$. Thus we obtain the estimate
 $$
\norm{q_{s}(T) f \br{T} q_{t}(T)} \leq
c \cdot C_{\mu} \cdot \norm{f}_{\infty} \cdot \left\lbrace \begin{array}{c c} 
 \br{\frac{t}{s}}^{\alpha} \br{1 + \log \br{\frac{s}{t}}} & for \ 0 <
                                                            t \leq s <
                                                            \infty \\
                                                     & \\
                                               \br{\frac{s}{t}}^{\alpha}
                                               \br{1 + \log
                                               \br{\frac{t}{s}}} & for
                                                                   \
                                                                   0 <
                                                                   s
                                                                   <
                                                                   t < \infty,
 \end{array} \right.
 $$
where the value of the $T$ independent constant $c$ is allowed to change. This then implies that
$$
\sup_{s > 0}\int^{\infty}_{0} \norm{q_{s}(T)f(T)
  q_{t}(T)} \frac{dt}{t}, \quad \sup_{t > 0}
\int^{\infty}_{0}\norm{q_{s}(T) f(T) q_{t}(T)} \frac{ds}{s}
\leq c \cdot C_{\mu} \cdot \norm{f}_{\infty}.
$$
On applying this estimate to \eqref{eqtn:CauchySchwarz}, 
\begin{align*}\begin{split}  
 \norm{f \br{T} u}^{2} &\leq c \cdot C_{\mu} \cdot C_{SF} \cdot \norm{f}_{\infty} \int^{\infty}_{0} \int^{\infty}_{0} \norm{q_{s}(T) f
  \br{T} q_{t}(T)} \norm{q_{t}(T) u}^{2} \frac{dt}{t} \frac{ds}{s}  \\
&= c \cdot C_{\mu} \cdot C_{SF} \cdot \norm{f}_{\infty} \int^{\infty}_{0}
\norm{q_{t}(T) u}^{2}  \int^{\infty}_{0} \norm{q_{s}(T) f
  \br{T} q_{t}(T)} \frac{ds}{s} \frac{dt}{t} \\
&\leq c^{2} \cdot C_{\mu}^{2} \cdot C_{SF} \cdot \norm{f}_{\infty}^{2} \int_{0}^{\infty}
\norm{q_{t}(T) u}^{2} \frac{dt}{t}  \\
&\lesssim c^{2} \cdot C_{\mu}^{2} \cdot C_{SF}^{2} \cdot  \norm{f}_{\infty}^{2} \norm{u}^{2}.
 \end{split}\end{align*}
 \end{proof}

Finally, the following Kato-type estimate follows from a well-known
classical argument.

\begin{cor}
  \label{cor:Kato}
  Suppose that the bisectorial operator $T$ satisfies square function
  estimates with constant $C_{SF} > 0$ and the constant in the resolvent
  estimate \eqref{eqtn:ResolventEst} is $C_{\mu} > 0$. Then there
  exists a constant $c > 0$, independent of $T$, such that
  \begin{equation} 
 \label{cor:eqtn:Kato} 
 \br{c \cdot C_{SF} \cdot C_{\mu}}^{-1} \cdot \norm{T u} \leq \norm{\sqrt{T^{2}}u} \leq
 \br{c \cdot C_{SF} \cdot C_{\mu}} \cdot \norm{T u}
\end{equation}
 for any $u \in D\br{T}$.
\end{cor}

\begin{proof}  
Consider the restriction $S := T \vert_{\overline{R\br{T}}}$. $S$ is
an injective bisectorial operator that satisfies square function
estimates with constant $C_{SF} > 0$. Since $S$ is injective it follows that $f(S)$ is
well-defined for any $f \in H^{\infty} \br{S^{o}_{\mu}}$ by Remark
\ref{rmk:Injective}. This allows us to define the operators $f_{1}(S)$
and $f_{2}(S)$, where the functions $f_{1}$ and $f_{2}$ are defined by
$$
f_{1}(z) := \frac{\sqrt{z^{2}}}{z} \quad and \quad f_{2}(z) :=
\frac{z}{\sqrt{z^{2}}} \quad for \quad z \in S^{o}_{\mu}.
$$
The previous theorem allows us to deduce that both of these operators
are norm bounded by $c \cdot C_{SF} \cdot C_{\mu}$ for some $T$
independent constant $c > 0$. Applying the multiplicative part of Theorem
\ref{thm:FundFunctCalc} to the functions $f = f_{1}$ and $g(z) = z$
implies that
\begin{equation}
  \label{eqtn:cor:Kato1}
\frac{\sqrt{S^{2}}}{S} \cdot S = \sqrt{S^{2}}
\end{equation}
on
\begin{equation}
  \label{eqtn:cor:Kato2}
D(S) = D \br{\frac{\sqrt{S^{2}}}{S} \cdot S} = D \br{\sqrt{S^{2}}}
\cap D \br{S}.
\end{equation}
Similarly, applying the multiplicative part of Theorem
\ref{thm:FundFunctCalc} to $f = f_{2}$ and $g(z) = \sqrt{z^{2}}$ gives
\begin{equation}
  \label{eqtn:cor:Kato3}
  \frac{S}{\sqrt{S^{2}}} \cdot \sqrt{S^{2}} = S
\end{equation}
on
\begin{equation} 
 \label{eqtn:cor:Kato4} 
 D \br{\sqrt{S^{2}}} = D \br{\frac{S}{\sqrt{S^{2}}} \cdot
   \sqrt{S^{2}}} = D \br{S} \cap D \br{\sqrt{S^{2}}}.
\end{equation}
Equations \eqref{eqtn:cor:Kato2} and \eqref{eqtn:cor:Kato4} together
imply that the domains $D \br{\sqrt{S^{2}}}$ and $D \br{S}$ coincide
and therefore both \eqref{eqtn:cor:Kato1} and \eqref{eqtn:cor:Kato3}
will remain valid on all of $D \br{S}$.

Let $u \in D \br{T}$. Proposition \ref{prop:Decomposition} states that $u$ has the decomposition $u = u_{1} \oplus
u_{2} \in N \br{T} \oplus \overline{R \br{T}}$. Then
\begin{align*}\begin{split}  
    \norm{T u} &= \norm{S u_{2}} \\
    &= \norm{\frac{S}{\sqrt{S^{2}}} \cdot \sqrt{S^{2}} u_{2}} \\
    &\leq c \cdot C_{SF} \cdot C_{\mu} \cdot \norm{\sqrt{S^{2}} u_{2}} \\
    &= c \cdot C_{SF} \cdot C_{\mu} \cdot \norm{\sqrt{T^{2}}u},
  \end{split}\end{align*}
where in the last line we used the fact that the functional
calculus commutes with the restriction map as given in Proposition
2.6.5 of \cite{haase2006functional}. Also,
\begin{align*}\begin{split}  
    \norm{\sqrt{T^{2}}u} &= \norm{\sqrt{S^{2}}u_{2}} \\
    &= \norm{\frac{\sqrt{S^{2}}}{S} \cdot S u_{2}} \\
    &\leq c \cdot C_{SF} \cdot C_{\mu} \cdot \norm{S u_{2}} \\
    &= c \cdot C_{SF} \cdot C_{\mu} \cdot \norm{T u}.
  \end{split}\end{align*}
\end{proof}

\section{Non-Homogeneous Axelsson-Keith-McIntosh}
\label{sec:NonHomog}

In this section we describe how the Axelsson-Keith-McIntosh framework
can be altered to account for non-homogeneous operators of the form \eqref{eqtn:NonHomog}. Our  main
results for this framework will also be stated.

\subsection{AKM without Cancellation and Coercivity}

The operators that we wish to consider, $\Gamma_{J}$,
will satisfy the first six conditions of
\cite{axelsson2006quadratic}. However, they will not necessarily
satisfy the cancellation condition (H7) and the coercivity condition
(H8). It will therefore be fruitful to see what happens to the
original AKM framework when the cancellation and coercivity conditions
are removed.

Similar to the original result, we begin by assuming that we have operators that satisfy the
hypotheses (H1) - (H3) from \cite{axelsson2006quadratic}. Recall these
conditions for operators $\Gamma$, $B_{1}$ and $B_{2}$ on a Hilbert
space $\mathcal{H}$ with norm $\norm{\cdot}$ and inner product
$\langle \cdot, \cdot \rangle$.

\vspace*{0.1in}

\begin{enumerate}
\item[(H1)] $\Gamma : D(\Gamma) \rightarrow \mathcal{H}$ is a closed,
  densely defined, nilpotent operator.
  \\
\item[(H2)] $B_{1}$ and $B_{2}$ satisfy the accretivity conditions
  $$
\mathrm{Re} \langle B_{1} u, u \rangle \geq \kappa_{1} \norm{u}^{2}
\qquad and \qquad
\mathrm{Re} \langle B_{2} v, v \rangle \geq \kappa_{2} \norm{v}^{2}
$$
for all $u \in R(\Gamma^{*})$ and $v \in R \br{\Gamma}$ for some
$\kappa_{1}$, $\kappa_{2} > 0$.
\\
\item[(H3)] The operators $\Gamma$ and $\Gamma^{*}$ satisfy
  $$
\Gamma^{*} B_{2} B_{1} \Gamma^{*} = 0 \qquad and \qquad
\Gamma B_{1} B_{2} \Gamma
= 0.
$$ 
\end{enumerate}

In \cite{axelsson2006quadratic} Section $4$, the authors assume that they
have operators that satisfy the
hypotheses (H1) - (H3) and they derive several important operator
theoretic consequences from only these hypotheses. As our operators
$\Gamma$, $B_{1}$ and $B_{2}$ also satisfy (H1) - (H3), it follows
that any result proved in {\cite[Sec.~4]{axelsson2006quadratic}}
must also be true for our operators and can be used with
impunity.  In the
interest of making this article as self-contained as possible, we will now
restate any such result that is to be used in this paper.

\begin{prop}[The Hodge Decomposition, {\cite[Prop.~2.2]{axelsson2006quadratic}}]
 \label{prop:Hodge} 
 Suppose that the operators $\lb \Gamma, B_{1}, B_{2} \rb$ satisfy (H1)
- (H3). Define the perturbation dependent operators 
$$
\Gamma^{*}_{B} := B_{1} \Gamma^{*} B_{2}, \quad \Gamma_{B} :=
B_{2}^{*} \Gamma B_{1}^{*} \quad and \quad \Pi_{B} := \Gamma + \Gamma_{B}^{*}.
$$
The Hilbert space $\mathcal{H}$ has the following Hodge decomposition
 into closed subspaces:
\begin{equation} 
 \label{eqtn:Hodge} 
 \mathcal{H} = N \br{\Pi_{B}} \oplus \overline{R \br{\Gamma^{*}_{B}}}
 \oplus \overline{R \br{\Gamma}}.
\end{equation}
Moreover, we have $N \br{\Pi_{B}} = N \br{\Gamma^{*}_{B}} \cap N
\br{\Gamma}$ and $\overline{R \br{\Pi_{B}}} = \overline{R
  \br{\Gamma^{*}_{B}}} \oplus \overline{R \br{\Gamma}}$. When $B_{1} =
B_{2} = I$ these decompositions are orthogonal, and in general the
decompositions are topological. Similarly, there is also a
decomposition
$$
\mathcal{H} = N \br{\Pi_{B}^{*}} \oplus \overline{R \br{\Gamma_{B}}}
\oplus \overline{R \br{\Gamma^{*}}}.
$$
 \end{prop}

\begin{prop}[{\cite[Prop.~2.5]{axelsson2006quadratic}}]
  \label{prop:Bisectoriality}
  Suppose that the operators $\lb \Gamma, B_{1}, B_{2} \rb$ satisfy (H1)
- (H3).
  The perturbed Dirac-type operator $\Pi_{B}$ is an $\omega$-bisectorial operator with $\omega
  := \frac{1}{2} \br{\omega_{1} + \omega_{2}}$ where
  $$
\omega_{1} := \sup_{u \in R \br{\Gamma^{*}} \setminus \lb 0 \rb}
\abs{\mathrm{arg} \langle B_{1} u, u \rangle} < \frac{\pi}{2}
$$
and
$$
\omega_{2} := \sup_{u \in R \br{\Gamma} \setminus \lb 0 \rb}
\abs{\mathrm{arg} \langle B_{2}u, u \rangle} < \frac{\pi}{2}.
$$
 \end{prop}

 The bisectoriality of $\Pi_{B}$ ensures that the following operators
 will be well-defined.

\begin{deff} 
 \label{def:Operators}  
Suppose that the operators $\lb \Gamma, B_{1}, B_{2} \rb$ satisfy (H1)
- (H3). For $t \in \R \setminus \lb 0 \rb$, define the perturbation dependent operators
$$
R^{B}_{t} := \br{I + i t \Pi_{B}}^{-1}, \quad P^{B}_{t} := \br{I +
  t^{2} \br{\Pi_{B}}^{2}}^{-1},
$$
$$
Q_{t}^{B} := t \Pi_{B} P_{t}^{B} \quad and \quad \Theta_{t}^{B} := t \Gamma^{*}_{B} P_{t}^{B}.
$$
When there is no perturbation, i.e. when $B_{1} =
B_{2} = I$, the $B$ will dropped from the superscript or subscript. For example,
instead of $\Theta^{I}_{t}$ or $\Pi_{I}$ the notation $\Theta_{t}$ and
$\Pi$ will be employed.
\end{deff}

\begin{rmk}
An easy consequence of Proposition \ref{prop:Bisectoriality} is that
the operators $R^{B}_{t}$, $P^{B}_{t}$ and $Q^{B}_{t}$ are all uniformly
$\mathcal{H}$-bounded in $t$. Furthermore, on taking the Hodge decomposition
Proposition \ref{prop:Hodge} into account, it is clear that the
operators $\Theta^{B}_{t}$ will also be uniformly $\mathcal{H}$-bounded in
$t$.
\end{rmk}

The next result tells us how the operators $\Pi_{B}$ and $P^{B}_{t}$ interact
with $\Gamma$ and $\Gamma^{*}_{B}$.

\begin{lem}[{\cite[Rmk.~4.5]{axelsson2006quadratic}}]
  \label{lem:Commutation}
  Suppose that the operators $\lb \Gamma, B_{1}, B_{2} \rb$ satisfy (H1)
- (H3). The following relations are true.
 $$
 \Pi_{B} \Gamma u = \Gamma^{*}_{B} \Pi_{B} u \quad for \ all \ u \in D
 \br{\Gamma^{*}_{B} \Pi_{B}},
 $$
 $$
\Pi_{B} \Gamma^{*}_{B} u = \Gamma \Pi_{B} u \quad for \ all \ u \in D
 \br{\Gamma \Pi_{B}},
 $$
 $$
\Gamma P_{t}^{B} u = P^{B}_{t} \Gamma u \quad for \ all \ u \in D
\br{\Gamma}, \quad and
$$
$$
\Gamma^{*}_{B} P^{B}_{t} u = P^{B}_{t} \Gamma^{*}_{B}u  \quad for \ all
\ u \in D \br{\Gamma^{*}_{B}}.
$$
 \end{lem}

The subsequent lemma provides a square function estimate for the
unperturbed Dirac-type operator $\Pi$. When considering square
function estimates for the perturbed operator, there will be several
instances where the perturbed case can be reduced with the assistance
of this unperturbed estimate. Its proof follows directly from the
self-adjointness of the operator $\Pi$ and Corollary \ref{cor:SA}.

\begin{lem}[{\cite[Lem.~4.6]{axelsson2006quadratic}}]
 \label{lem:Unperturbed} 
 Suppose that the operators $\lb \Gamma, B_{1}, B_{2} \rb$ satisfy (H1)
- (H3). The quadratic estimate
\begin{equation} 
 \label{eqtn:Unperturbed} 
 \int^{\infty}_{0} \norm{Q_{t}u}^{2} \frac{dt}{t} \leq \frac{1}{2} \norm{u}^{2}
 \end{equation}
holds for all $u \in \mathcal{H}$. Equality holds on $\overline{R\br{\Pi}}$.
 \end{lem}

The following result will play a crucial role in
the reduction of the square function estimate \eqref{eqtn:Intro1}.

\begin{prop}[{\cite[Prop.~4.8]{axelsson2006quadratic}}]
 \label{prop:Reduction} 
 Suppose that the operators $\lb \Gamma, B_{1}, B_{2} \rb$ satisfy (H1)
- (H3). Assume that the estimate
 \begin{equation}
   \label{eqtn:prop:Reduction}
   \int^{\infty}_{0} \norm{\Theta^{B}_{t} P_{t} u}^{2} \frac{dt}{t}
   \leq c \cdot \norm{u}^{2}
 \end{equation}
 holds for all $u \in R \br{\Gamma}$ and some constant $c > 0$, together
 with three similar estimates obtained on replacing $\lb \Gamma,
 B_{1}, B_{2} \rb$ by $\lb \Gamma^{*}, B_{2}, B_{1} \rb$, $\lb
 \Gamma^{*}, B_{2}^{*}, B_{1}^{*} \rb$ and $\lb \Gamma, B_{1}^{*},
 B_{2}^{*} \rb$. Then $\Pi_{B}$ satisfies the quadratic estimate
 \begin{equation} 
 \label{eqtn:prop:Reduction1} 
 \br{c \cdot C}^{-1} \cdot \norm{u}^{2} \leq \int^{\infty}_{0} \norm{Q^{B}_{t} u}^{2}
 \frac{dt}{t} \leq c \cdot C \cdot \norm{u}^{2}
\end{equation}
for all $u \in \overline{R \br{\Pi_{B}}}$, for some $C > 0$ entirely
dependent on (H1) - (H3).
\end{prop}

The constant dependence of \eqref{eqtn:prop:Reduction1} is not
explicitly mentioned in Proposition 4.8 of \cite{axelsson2006quadratic}, but it is
relatively easy to trace through their argument and record where
\eqref{eqtn:prop:Reduction} is used. The following corollary is proved
during the course of the proof  of 
{\cite[Prop.~4.8]{axelsson2006quadratic}}. 

\begin{cor}[High Frequency Estimate]
  \label{cor:HighFrequency}
  Suppose that the operators $\lb \Gamma, B_{1}, B_{2} \rb$ satisfy (H1)
- (H3). There exists a constant $c > 0$ such
  that for any $u \in R\br{\Gamma}$,
  $$
\int^{\infty}_{0} \norm{\Theta^{B}_{t} \br{I - P_{t}}u}^{2}
\frac{dt}{t} \leq c \cdot \norm{u}^{2}.
  $$
  \end{cor}

From this point onwards, it will also be assumed that our operators
satisfy the additional hypotheses (H4) - (H6). These hypotheses are
stated below for reference.

\vspace*{0.1in}

\begin{enumerate}
\item[(H4)] The Hilbert space is $\mathcal{H} = L^{2} \br{\R^{n};
    \C^{N}}$ for some $n, \, N \in \N^{*}$.
  \\
  \item[(H5)] The operators $B_{1}$ and $B_{2}$ represent
    multiplication by matrix-valued functions. That is,
    $$
B_{1}(f)(x) = B_{1}(x) \cdot f(x) \qquad and \qquad B_{2}(f)(x) =
B_{2}(x) \cdot f(x)
$$
for all $f \in \mathcal{H}$ and $x \in \R^{n}$, where $B_{1}$, $B_{2}
\in L^{\infty}\br{\R^{n}; \mathcal{L}\br{\C^{N}}}$.
\\
 \item[(H6)] For every bounded Lipschitz function $\eta : \R^{n}
   \rightarrow \C$, we have that $\eta D(\Gamma) \subset D
   \br{\Gamma}$ and $\eta D \br{\Gamma^{*}} \subset D
   \br{\Gamma^{*}}$. Moreover, the commutators $\brs{\Gamma,
     \eta I}$ and $\brs{\Gamma^{*}, \eta I}$ are 
   multiplication operators that satisfy the bound
   $$
\abs{\brs{\Gamma, \eta I}(x)}, \ \abs{\brs{\Gamma^{*}, \eta I}(x)} \leq c \abs{\nabla \eta(x)}
$$
for all $x \in \R^{n}$ and some constant $c > 0$.
\end{enumerate}

\vspace*{0.1in}

In contrast to the original result, our operators will not be assumed to satisfy the cancellation
condition (H7) or the coercivity condition (H8). Without these two
conditions, many of the results from Section $5$ of
\cite{axelsson2006quadratic} will fail. One notable exception to this
is that the bounded operators associated with our perturbed
Dirac-type operator $\Pi_{B}$ will satisfy off-diagonal estimates.

\begin{deff}[Off-Diagonal Bounds]
  \label{def:OffDiagonal}
  Define
  $
  \langle x \rangle := 1 + \abs{x}
  $
  for $x \in \C$ and $ \mathrm{dist}(E,F) := \inf \lb \abs{x - y} : x
  \in E, y \in F \rb$ for $E$, $F \subset \R^{n}$.
  
Let $\lb U_{t} \rb_{t > 0}$ be a family of operators on $\mathcal{H} =
L^{2}\br{\R^{n};\C^{N}}$.
 This collection is said to have off-diagonal bounds of order $M
 > 0$ if there exists $C_{M} > 0$  such that
 \begin{equation}
   \label{def:eqtn:OffDiagonal}
\norm{U_{t}u}_{L^{2}\br{E}} \leq C_{M}\langle \mathrm{dist}(E,F) / t
  \rangle^{-M} \norm{u}
  \end{equation}
  whenever $E, \, F \subset \R^{n}$  are Borel sets and $u \in
  \mathcal{H}$ satisfies $\mathrm{supp} \, u \subset F$. 
\end{deff}

 \begin{prop}[{\cite[Prop.~5.2]{axelsson2006quadratic}}]
   \label{prop:OffDiagonal}
   Suppose that the operators $\lb \Gamma, B_{1}, B_{2} \rb$ satisfy
   (H1) - (H6).
 Let $U_{t}$ be given by either $R^{B}_{t}$,
$R^{B}_{-t}$, $P^{B}_{t}$, $Q^{B}_{t}$ or $\Theta^{B}_{t}$ for every $t >
0$. The collection of operators $\lb U_{t} \rb_{t
  > 0}$  has off-diagonal bounds of every order $M > 0$. Moreover, the constant $C_{M}$ in the estimate
\eqref{def:eqtn:OffDiagonal} depends only on $M$ and the
hypotheses (H1) - (H6).
\end{prop}

Introduce the following dyadic decomposition of $\R^{n}$. Let $\Delta
= \cup_{j = - \infty}^{\infty} \Delta_{2^{j}}$ where $\Delta_{t} :=
\lb 2^{j} \left(k + (0,1]^{n}\right) : k \in \mathbb{Z}^{n} \rb$ if $2^{j-1} < t
  \leq 2^{j}$ for $j \in \mathbb{Z}$. Define the averaging operator $A_{t} : \mathcal{H}
  \rightarrow \mathcal{H}$ through
$$
A_{t}u(x) := \dashint_{Q(x,t)} u(y) \, dy := \frac{1}{\abs{Q(x,t)}} \int_{Q(x,t)} u(y) \, dy
$$
for $x \in \R^{n}$, $t > 0$ and $u \in \mathcal{H}$, where $Q(x,t)$ is
the unique dyadic cube in $\Delta_{t}$ that contains the point $x$.

For an operator family $\lb U_{t} \rb_{t > 0}$ that satisfies
off-diagonal bounds of every order, there exists an extension $U_{t}:
L^{\infty} \br{\R^{n}; \C^{N}} \rightarrow L^{2}_{loc}
\br{\R^{n};\C^{N}}$ for each $t > 0$. This is constructed by defining 
$$
U_{t} u(x) := \lim_{r \rightarrow \infty} \sum_{\substack{R \in
    \Delta_{t} \\ \mathrm{dist} \br{Q,R}} < r} U_{t}
\br{\mathbbm{1}_{R} u}(x),
$$
for $x \in Q \in \Delta_{t}$ and $u \in L^{\infty} \br{\R^{n};
  \C^{N}}$. The convergence of the above limit is guaranteed
by the off-diagonal bounds of $\lb U_{t} \rb_{t > 0}$. Further detail
on this construction can be found in \cite{axelsson2006quadratic},
\cite{egert2016kato}, \cite{morris2012kato} or \cite{frey2018conical}. The above
extension then allows us to introduce the principal part of the
operator $U_{t}$.

\begin{deff} 
 \label{def:PrincipalPart} 
 Let $\lb U_{t} \rb_{t > 0}$ be operators on $\mathcal{H}$ that satisfy
 off-diagonal bounds of every order. For $t > 0$, the principal part of $U_{t}$ is
 the operator $\zeta_{t} : \R^{n} \rightarrow \mathcal{L}\br{\C^{N}}$
 defined through
 $$
\brs{\zeta_{t}(x)}\br{w} := \br{U_{t} w}(x)
 $$
for each $x \in \R^{n}$ and $w \in \C^{N}$.
 \end{deff}

The following generalisation of Corollary $5.3$ of
\cite{axelsson2006quadratic} will also be true with an identical proof.

\begin{prop}
  \label{prop:OffDiagonal}
Let $\lb U_{t} \rb_{t > 0}$ be operators on $\mathcal{H}$ that satisfy 
off-diagonal bounds of every order. Let $\zeta_{t} : \R^{n} \rightarrow
\mathcal{L}\br{\C^{N}}$ denote the principal part of the operator
$U_{t}$. Then there exists $c > 0$ such that
$$
\dashint_{Q} \abs{\zeta_{t}(y)}^{2} dy \leq c
$$
for all $Q \in \Delta_{t}$ and $t > 0$. Moreover, the operators $\zeta_{t} A_{t}$
are uniformly $L^{2}$-bounded in $t$.
\end{prop}

Finally, the ensuing partial result will also be valid. Its proof follows in
an identical manner to the first part of the proof of Proposition 5.5
of \cite{axelsson2006quadratic}.

 \begin{prop} 
 \label{prop:PrincPart} 
 Let $\lb U_{t} \rb_{t > 0}$ be operators on $\mathcal{H}$ that satisfy
 off-diagonal bounds of every order. Let $\zeta_{t} : \R^{n} \rightarrow
 \mathcal{L} \br{\C^{N}}$ denote the principal part of $U_{t}$. Then
 there exists $c > 0$ such that
 \begin{equation}
   \label{eqtn:PrincPart}
   \norm{\br{U_{t} - \zeta_{t} A_{t}} v} \leq c \cdot   \norm{t \nabla v}.
 \end{equation}
 for any $v \in H^{1} \br{\R^{n};\C^{N}} \subset \mathcal{H}$ and $t > 0$.
 \end{prop}

 \subsection{Additional Structure}
 \label{subsec:Additional}

At this point, further structure will be imposed upon  our
operators in order to generalise the non-homogeneous operator $\Gamma_{\abs{V}^{\frac{1}{2}}}$ defined in
\eqref{eqtn:OperatorsPotential}. This additional structure will later be exploited in order to obtain square
function estimates of the form \eqref{eqtn:prop:Reduction}.

\vspace*{0.1in}

Let $\C^{N} = V_{1} \oplus V_{2} \oplus V_{3}$
where $V_{1}, \, V_{2}$ and $V_{3}$ are finite-dimensional complex Hilbert spaces. Let $\mathbb{P}_{i} : \C^{N} \rightarrow \C^{N}$ be
the projection operator onto the space $V_{i}$ for $i = 1, \, 2$ and $3$. Our Hilbert space will have the following
orthogonal decomposition
$$
\mathcal{H} := L^{2} \br{\R^{n}; \C^{N}} =  L^{2}\br{\R^{n};V_{1}} \oplus L^{2}\br{\R^{n};V_{2}} \oplus L^{2}(\R^{n};V_{3}).
$$
The notation $\mathbb{P}_{i}$ will also be used to denote the natural
projection operator from $\mathcal{H}$ onto
$L^{2}\br{\R^{n};V_{i}}$. For a vector $v \in \mathcal{H}$, $v_{i} \in
L^{2}\br{\R^{n};V_{i}}$ will denote the $i$th component for $i = 1, \,
2$ or $3$.

Let $\Gamma_{J}$ be an operator on $\mathcal{H}$ of the form
$$
\Gamma_{J} := \br{\begin{array}{c c} 
 0 & 0 \\ D_{J} & 0
 \end{array}} := \br{\begin{array}{c c c} 
                    0 & 0 & 0 \\
                    J & 0 & 0 \\
                    D & 0 & 0
 \end{array}},
$$
where $J$ and $D$ are closed densely defined operators
$$
J : L^{2}\br{\R^{n};V_{1}} \rightarrow L^{2}\br{\R^{n};V_{2}} \ and
$$
$$
D : L^{2}\br{\R^{n};V_{1}} \rightarrow L^{2}\br{\R^{n};V_{3}},
$$
and $D_{J} : L^{2}\br{\R^{n};V_{1}} \rightarrow L^{2}\br{\R^{n};V_{2}}
\oplus L^{2}\br{\R^{n};V_{3}}$ is the operator $D_{J} = \br{\begin{array}{c} 
 J \\ D
 \end{array}}$. Define the operators
$$
\Gamma_{0} := \br{\begin{array}{c c c} 
            0 & 0 & 0 \\
            0 & 0 & 0 \\
            D & 0 & 0
 \end{array}}, \qquad M_{J} := \br{\begin{array}{c c c} 
               0 & 0 & 0 \\
               J & 0 & 0 \\
               0 & 0 & 0
 \end{array}},
$$

$$
\Pi_{0} := \Gamma_{0} + \Gamma_{0}^{*},  \quad S_{J} := M_{J} +
M_{J}^{*} \quad and \quad \Pi_{J} := \Gamma_{J} + \Gamma_{J}^{*}.
$$

\vspace*{0.1in}

Let $B_{1}, \, B_{2} \in L^{\infty} \br{\R^{n}; \mathcal{L}\br{\C^{N}}}$ be matrix-valued multiplication operators. The
following key assumption will be imposed on our operators throughout
the entirety of this article.

\vspace*{0.2in}

\fbox{\begin{minipage}{33em}
\textbf{Key Assumption.} \textit{The family of operators
$\lb \Gamma_{0}, B_{1}, B_{2} \rb$ satisfies the conditions (H1) -
(H8) of \cite{axelsson2006quadratic} while $\lb \Gamma_{J}, B_{1},
B_{2} \rb$ satisfies (H1) - (H6).}
\end{minipage}}

\vspace*{0.2in}

For reference, the cancellation condition (H7)
and the coercivity condition (H8) are shown below for the operator
$\Gamma_{0}$.

\vspace*{0.1in}

\begin{enumerate}

  \item[(H7)] For any $u \in D \br{\Gamma_{0}}$ and $v \in D
    \br{\Gamma^{*}_{0}}$, both compactly supported,
    $$
\int_{\R^{n}} \Gamma_{0} u = 0 \quad and \quad \int_{\R^{n}} \Gamma^{*}_{0} v
= 0.
$$

\vspace*{0.1in}

\item[(H8)]  There  exists $c > 0$ such that
  $$
\norm{\nabla u} \leq c \cdot \norm{\Pi_{0} u}
$$
for all $u \in R \br{\Pi_{0}} \cap D \br{\Pi_{0}}$.
  \end{enumerate}

  \vspace*{0.1in}

  \begin{example}
    Typical examples of operators that satisfy the previous key
    assumption are when both $D$ and $J$ are partial differential
    operators of order less than or equal to one. If 
the perturbations $B_{1}$ and $B_{2}$ satisfy suitable accretivity
conditions then the families of operators $\lb \Gamma_{0}, B_{1}, B_{2} \rb$ and $\lb
\Gamma_{J}, B_{1}, B_{2} \rb$ will both satisfy (H1) - (H6). If, in
addition, $D$ is first-order homogeneous and there exists $c > 0$ for which
$$
\norm{\nabla u} \leq c \cdot \norm{D u}
$$
for all $u \in R \br{D^{*}} \cap D \br{D}$ and
$$
\norm{\nabla v} \leq c \cdot \norm{D^{*}v}
$$
for all $v \in R \br{D}\cap D \br{D^{*}}$ then $\lb \Gamma_{0}, B_{1},
B_{2}\rb$ will also satisfy (H7) and (H8). A particular example of
such a situation is given by the operator
$\Gamma_{\abs{V}^{\frac{1}{2}}}$ together with perturbations $B_{1}$
and $B_{2}$ as defined in \eqref{eqtn:OperatorsPotential} and
\eqref{eqtn:OperatorsPotential2} with \eqref{eqtn:Garding0} and
\eqref{eqtn:Garding} satisfied.
    \end{example}

\begin{rmk}
Since the operator $\Gamma_{0}$, together with the
perturbations $B_{1}$ and $B_{2}$, satisfy  all
eight conditions (H1) - (H8) of \cite{axelsson2006quadratic}, it follows that any
result from that paper must be valid for these operators.
\end{rmk}

\begin{deff} 
  \label{def:Operators2}
For $t \in \R \setminus \lb 0 \rb$, define the perturbation dependent operators
$$
\Gamma_{J,B} := B_{2}^{*} \Gamma_{J} B_{1}^{*}, \quad \Gamma_{J,B}^{*} := B_{1} \Gamma^{*}_{J} B_{2}, \quad \Pi_{J,B} :=
\Gamma_{J} + \Gamma_{J,B}^{*},
$$
$$
R^{J,B}_{t} := \br{I + i t \Pi_{J,B}}^{-1}, \quad P^{J,B}_{t} := \br{I +
  t^{2} \br{\Pi_{J,B}}^{2}}^{-1},
$$
$$
Q_{t}^{J,B} := t \Pi_{J,B} P_{t}^{J,B} \quad and \quad \Theta_{t}^{J,B} := t \Gamma^{*}_{J,B} P_{t}^{J,B}.
$$
When there is no perturbation, i.e. when $B_{1} =
B_{2} = I$, the $B$ will dropped from the superscript or subscript. For example,
instead of $\Theta^{J,I}_{t}$ the notation $\Theta^{J}_{t}$ will be
employed.
\end{deff}

 We now introduce coercivity conditions to serve as a
replacement for (H8) for the operators $\lb \Gamma_{J}, B_{1}, B_{2} \rb$.  These conditions will not be automatically imposed
upon our operators but, rather, will be taken as hypotheses for
our main results.

\vspace*{0.1in}

\begin{enumerate}
\item[(H8D$\alpha$)] Let $\alpha \in (1,2]$. The domain inclusion
  $$
D\br{(D_{J}^{*}D_{J})^{\frac{\alpha}{2}}} \subset D \br{(D^{*}D)^{\frac{\alpha}{2}}}
$$
holds and there exists a
constant $C > 0$ such that
$$
\norm{\br{D^{*} D}^{\frac{\alpha}{2}} u} \leq C \cdot
\norm{\br{D_{J}^{*}D_{J}}^{\frac{\alpha}{2}}u}
$$
for all $u \in D \br{(D_{J}^{*}D_{J})^{\frac{\alpha}{2}}}$.

\vspace*{0.1in}

\item[(H8J)] $B_{2}$ is of the form
  \begin{equation}
    \label{eqtn:B2Form}
B_{2} = \br{\begin{array}{c c} 
 I & 0 \\ 0 & \hat{A}
 \end{array}} := \br{\begin{array}{c c c} 
              I & 0 & 0 \\
              0 & A_{22} & A_{23} \\
              0 & A_{32} & A_{33}
 \end{array}},
\end{equation}
where $A_{ij} \in L^{\infty} \br{\R^{n}; \mathcal{L} \br{V_{j},
    V_{i}}}$ for $i, \, j = 2$ or $3$. The inclusion
  $$
D \br{D_{J}^{*} D_{J}} \subset D \br{D_{J}^{*} \hat{A} \br{\begin{array}{c} 
 J \\ 0
 \end{array}}}
 $$
is satisfied. Furthermore, there exists a constant $C > 0$ such that for all $u
  \in D \br{D_{J}^{*} D_{J}}$,
  $$
\norm{ D_{J}^{*} \hat{A} \br{\begin{array}{c} 
 J \\ 0
 \end{array}}u} \leq C \cdot
\norm{D_{J}^{*} D_{J} u}.
$$
\end{enumerate}

\vspace*{0.1in}

\begin{rmk}
  \label{rmk:H8J}
  The situation of most interest to us is when $A_{32} = 0$ and
  $$
\norm{J^{*} A_{22} J u} = \norm{J^{*}J u}
$$
for all $u \in D \br{J^{*} A_{22} J} = D \br{J^{*} J}$. In this case,
the domain inclusion of (H8J) becomes $D(D^{*}_{J} D_{J}) \subset
D(J^{*}J)$ and the Riesz transform condition becomes the perturbation free condition
$$
\norm{J^{*}J u} \leq C \cdot \norm{D_{J}^{*} D_{J} u}
$$
for all $u \in D \br{D_{J}^{*} D_{J}}$. Furthermore, when this occurs, the Riesz
transform condition of (H8J) will
be equivalent to the condition
$$
\norm{S_{J} u} \leq C \cdot \norm{\Pi_{J} u}
$$
or equivalently
$$
\norm{\Pi_{0} u} \leq C \cdot \norm{\Pi_{J} u}
$$
for all $u \in D \br{\Pi_{J}} \cap R \br{\Pi_{J}}$. 
  \end{rmk}

  \vspace*{0.1in}

The Kato square root estimate is a first-order Riesz transform
condition. To some extent, it then seems intuitively unnatural to use a
second-order Riesz transform condition as in (H8J) as our
hypotheses. Indeed, when the conditions of the above remark are satisfied and $J$ is a
  positive operator, it will be
  sufficient to consider a lower-order version of (H8J) as given below.

  \vspace*{0.1in}

\begin{enumerate}
\item[(H8J$\alpha$)] Let $\alpha \in (1,2]$. The perturbation $B_{2}$
  is of the form \eqref{eqtn:B2Form} with $A_{32} = 0$. $J$ is a
  positive operator and
  $$
\norm{J A_{22} J u} = \norm{J^{2} u}
$$
for all $u \in D \br{J A_{22} J} = D \br{J^{2}}$. The domain inclusion
$$
D \br{(D_{J}^{*} D_{J})^{\frac{\alpha}{2}}} \subset D \br{J^{\alpha}}
$$
holds. Furthermore, there exists a
constant $C > 0$ such that
$$
\norm{J^{\alpha} u} \leq C \cdot  \norm{\br{D_{J}^{*} D_{J}}^{\frac{\alpha}{2}}u}
$$
for all $u \in D \br{(D_{J}^{*} D_{J})^{\frac{\alpha}{2}}}$.
\end{enumerate}

\vspace*{0.1in}

\begin{notation}
  For $\alpha \in (1,2]$, let $b^{D}_{\alpha}, \, b^{J}$ and
  $b^{J}_{\alpha}$ denote the smallest constant for which
  (H8D$\alpha$), (H8J) or (H8J$\alpha$) are satisfied respectively. If the criteria
  for one of these conditions is not met then the corresponding
  constant will be set to infinity. For example, if (H8J) is not
  satisfied then $b^{J} = \infty$. Also define
  $$
c_{\alpha}^{J} := \br{1 + \br{b^{D}_{\alpha}}^{2} + \br{\min \lb
    b^{J}, b^{J}_{\alpha} \rb}^{2}} \cdot (\alpha - 1)^{-1}.
  $$
For the remainder of this article we
introduce the notation $A \lesssim B$ and $A \simeq B$ to denote that there exists a constant $C >
0$, independent of (H8D$\alpha$), (H8J) and (H8J$\alpha$), for which $A \leq C \cdot B$ and $C^{-1}
\cdot B \leq A \leq C \cdot B$ respectively.  $C$ is still
allowed to depend on (H1) - (H8) for $\lb \Gamma_{0}, B_{1}, B_{2}
\rb$ and (H1) - (H6) for $\lb \Gamma_{J}, B_{1}, B_{2} \rb$. Note that
this implies that $C$ will only depend on $J$ through the constants in
(H2) for $\lb \Gamma_{J}, B_{1}, B_{2} \rb$ and it will be independent
of $\alpha$.
\end{notation}

We are now in a position
where the main result of the non-homogeneous AKM framework can finally
be stated.

\begin{thm} 
 \label{thm:MainSqEst} 
Let $\lb \Gamma_{J}, B_{1}, B_{2} \rb$ be as defined
above. Consider the following square function estimate:
 \begin{equation} 
 \label{eqtn:MainSqEst} 
 \int^{\infty}_{0} \norm{\Theta^{J,B}_{t} \mathbb{P}_{i} P_{t}^{J}u}^{2} \frac{dt}{t}
 \lesssim C \cdot \norm{u}^{2}
\end{equation}
for all $u \in R \br{\Gamma_{J}}$, for some $C > 0$,  for $i = 1, \, 2$ or
$3$.

\vspace{0.1in}

\begin{enumerate}[(i)]
  \item \label{Main1} The estimate is trivially satisfied for $i = 1$ for any
    $C \geq 0$.

\vspace*{0.1in}

\item \label{Main2} If (H8J) is satisfied then \eqref{eqtn:MainSqEst} is satisfied
  for $i = 2$ with $C = \br{b^{J}}^{2}$.

\vspace*{0.1in}
  
  \item \label{Main3} If (H8J$\alpha$) is satisfied for some $\alpha \in (1,2]$
    then \eqref{eqtn:MainSqEst} is satisfied
for $i = 2$ with $C = (1+\br{b_{\alpha}^{J}}^{2}) (\alpha - 1)^{-1}$.

\vspace*{0.1in}

\item \label{Main4} If (H8D$\alpha$) is satisfied for some $\alpha \in (1,2]$ and
  either (H8J) or (H8J$\alpha$) is also satisfied then
\eqref{eqtn:MainSqEst} holds for $i = 3$ with constant $C = c_{\alpha}^{J}$.

\end{enumerate}
\end{thm}

\begin{proof}  
  The proof of (iii) and (iv) will be postponed until Section
  \ref{sec:Square}. For (i), simply note that since $\Gamma_{J}$
  commutes with the operator $P_{t}^{J}$ by Lemma \ref{lem:Commutation} we must have $\mathbb{P}_{1}
  P_{t}^{J} u = 0$ for any $u \in R \br{\Gamma_{J}}$.
  
It remains to  consider (ii).   Suppose that (H8J) holds. First it will be proved that for $u \in R
  \br{\Gamma_{J}}$ we have $\mathbb{P}_{2} P_{t}^{J} u \in D
  \br{\Gamma_{J,B}^{*}}$.
 Since $u \in R \br{\Gamma_{J}}$, $u = \Gamma_{J} v$ for some
$v \in D \br{\Gamma_{J}}$. As
$$
P_{t}^{J} u = P_{t}^{J} \Gamma_{J} v = \Gamma_{J} P_{t}^{J} v
$$
by Lemma \ref{lem:Commutation} and $P_{t}^{J} u \in D \br{\Pi_{J}}$,
it follows that $\br{P_{t}^{J} v}_{1}\in D \br{D_{J}^{*} D_{J}}$,
which by (H8J) is contained in $D \br{D_{J}^{*} \hat{A} \br{\begin{array}{c} 
 J \\ 0
 \end{array}}}$. This implies that
$$
\br{\begin{array}{c} 
 J (P_{t}^{J} v)_{1} \\ 0
 \end{array}} = \br{\begin{array}{c} 
  (P_{t}^{J}u)_{2} \\ 0
 \end{array}} \in D (D_{J}^{*} \hat{A})
$$
and therefore $\mathbb{P}_{2} P_{t}^{J} u \in D (\Gamma^{*}_{J,B})$.

Since $\mathbb{P}_{2} P_{t}^{J} u \in D \br{\Gamma_{J,B}^{*}}$, it
follows from Lemma \ref{lem:Commutation} that
$$
\Theta^{J,B}_{t} \mathbb{P}_{2} P_{t}^{J}u = P^{J,B}_{t}
t \Gamma^{*}_{J,B} \mathbb{P}_{2} P_{t}^{J}u.
$$
The estimate in (H8J) gives 
$$
\norm{\Gamma_{J,B}^{*} \mathbb{P}_{2} \tilde{v}} \leq b^{J} \cdot
\norm{\Gamma^{*}_{J} \tilde{v}}
$$
for any $\tilde{v} \in R \br{\Gamma_{J}} \cap D (\Pi_{J})$. Since $P_{t}^{J}$ and $\Gamma_{J}$
commute by Lemma \ref{lem:Commutation}, it follows that
$$
\norm{\Gamma_{J,B}^{*} \mathbb{P}_{2} P_{t}^{J} u} \leq b^{J}
\cdot  \norm{\Gamma^{*}_{J} P_{t}^{J} u}
$$ 
for $u \in R\br{\Gamma_{J}}$.  On applying the uniform
$L^{2}$-boundedness of the $P_{t}^{J,B}$ operators,
\begin{align*}\begin{split}
    \int^{\infty}_{0}\norm{\Theta^{J,B}_{t} \mathbb{P}_{2} P_{t}^{J}
      u}^{2}\frac{dt}{t} &=
 \int^{\infty}_{0} \norm{P_{t}^{J,B} t \Gamma^{*}_{J,B} \mathbb{P}_{2} P_{t}^{J} u}^{2}
 \frac{dt}{t} \\ &\lesssim \int^{\infty}_{0} \norm{t \Gamma_{J,B}^{*} \mathbb{P}_{2} P_{t}^{J}
   u}^{2} \frac{dt}{t} \\
 &\leq \br{b^{J}}^{2} \int^{\infty}_{0} \norm{t \Gamma_{J}^{*} P_{t}^{J}u}^{2}  \frac{dt}{t}.
\end{split}\end{align*}
 On successively applying
Proposition \ref{prop:Hodge} and Lemma \ref{lem:Unperturbed} we obtain
\begin{align*}\begin{split}  
 \int^{\infty}_{0} \norm{\Theta^{J,B}_{t} \mathbb{P}_{2} P_{t}^{J} u}^{2}
 \frac{dt}{t} &\lesssim \br{b^{J}}^{2}  \int^{\infty}_{0} \norm{t \Pi_{J} P_{t}^{J}
   u}^{2} \frac{dt}{t} \\
 &= \br{b^{J}}^{2} \int^{\infty}_{0} \norm{Q_{t}^{J} u}^{2} \frac{dt}{t} \\
 &= \frac{1}{2} \br{b^{J}}^{2} \norm{u}^{2}.
\end{split}\end{align*}
This shows that \eqref{eqtn:MainSqEst} is valid for $i = 2$ with constant $C = \br{b^{J}}^{2}$. 
 \end{proof}

Let's consider an estimate that serves as a dual to \eqref{eqtn:MainSqEst}.

\begin{prop} 
 \label{prop:PropSwap} 
For $t > 0$, define the operator
 $$
\underline{P}^{J,B}_{t} := \br{I + t^{2} \br{\Gamma_{J}^{*} + B_{2}
    \Gamma_{J} B_{1}}^{2}}^{-1}.
$$
The square function estimate
\begin{equation}
  \label{eqtn:PropSwap}
 \int^{\infty}_{0} \norm{ t B_{2} \Gamma_{J} B_{1} \underline{P}^{J,B}_{t}
   P_{t}^{J} u}^{2} \frac{dt}{t} \lesssim \norm{u}^{2}
\end{equation}
will hold for all $u \in \mathcal{H}$ when $B_{1} = I$.
 \end{prop}

 \begin{proof}  
Since $\lb \Gamma_{J}, B_{1}, B_{2} \rb$ satisfies (H1) - (H6) it
follows that $\lb \Gamma_{J}^{*}, B_{2}, B_{1} \rb$ must also satisfy
(H1) - (H6). Proposition \ref{prop:Bisectoriality} then implies that the operators
$\underline{P}_{t}^{J,B}$ are well-defined  and uniformly
$L^{2}$-bounded. Since $B_{1} = I$, it follows that $P_{t}^{J} u \in
D(B_{2} \Gamma_{J} B_{1})$ for any $u \in \mathcal{H}$ and therefore,
by Lemma \ref{lem:Commutation},
$$
B_{2} \Gamma_{J} B_{1} \underline{P}_{t}^{J,B} P_{t}^{J} u =
\underline{P}_{t}^{J,B} B_{2} \Gamma_{J} B_{1} P_{t}^{J} u =
\underline{P}_{t}^{J} B_{2} \Gamma_{J} P_{t}^{J} u.
$$
This together with the uniform $L^{2}$-boundedness of the
$\underline{P}_{t}^{J,B}$ operators implies that
\begin{align*}\begin{split}  
 \int^{\infty}_{0} \norm{t B_{2} \Gamma_{J} B_{1}
   \underline{P}_{t}^{J,B} P_{t}^{J} u}^{2} \, \frac{dt}{t} &=   \int^{\infty}_{0} \norm{\underline{P}_{t}^{J,B} t B_{2} \Gamma_{J} 
   P_{t}^{J}u}^{2} \frac{dt}{t} \\
 &\lesssim \int^{\infty}_{0} \norm{t \Gamma_{J} P_{t}^{J} u}^{2} \frac{dt}{t}
 \\
 &\leq \int^{\infty}_{0} \norm{t \Pi_{J} P_{t}^{J} u}^{2}
 \frac{dt}{t}\\
 &= \int^{\infty}_{0} \norm{Q_{t}^{J}u}^{2} \frac{dt}{t} \\
 &= \frac{1}{2} \norm{u}^{2},
\end{split}\end{align*}
where the inequality $\norm{\Gamma_{J} v} \leq \norm{\Pi_{J} v}$ for $v
\in D \br{\Pi_{J}}$ follows immediately from the three-by-three matrix form
of the operators and Lemma \ref{lem:Unperturbed} was applied to obtain
the last line.
\end{proof}

From our main result, Theorem \ref{thm:MainSqEst}, and the previous proposition, the upper and lower square function
estimates for $Q^{J,B}_{t}$ can be proved. 

\begin{thm} 
 \label{thm:MainSqEst1} 
 Suppose that $B_{1} = I$. Suppose further that
(H8D$\alpha$) is satisfied for some $\alpha \in (1,2]$ and either
(H8J) or (H8J$\alpha$) is satisfied. Then
 \begin{equation} 
 \label{eqtn:ThmMainSqEst1} 
\br{c_{\alpha}^{J}}^{-1} \cdot \norm{u}^{2} \lesssim  \int^{\infty}_{0}
\norm{Q^{J,B}_{t}u}^{2} \frac{dt}{t} \lesssim c_{\alpha}^{J} \cdot  \norm{u}^{2}
\end{equation}
for all $u \in \overline{R \br{\Pi_{J}}}$.
\end{thm}

\begin{proof}  
 Proposition \ref{prop:Reduction} states that in order to prove the square function estimate
 \eqref{eqtn:ThmMainSqEst1}, it is sufficient for the estimate
 \eqref{eqtn:MainSqEst} to be valid for all $i = 1, \, 2$ and $3$ for the permutations of operators
 $\lb \Gamma_{J}, B_{1}, B_{2} \rb$, $\lb \Gamma_{J}, B_{1}^{*},
 B_{2}^{*} \rb$, $\lb \Gamma_{J}^{*},B_{2}, B_{1}  \rb$ and $\lb
 \Gamma_{J}^{*}, B_{2}^{*}, B_{1}^{*} \rb$. The permutations $\lb
 \Gamma_{J}, B_{1}, B_{2} \rb$ and $\lb \Gamma_{J}, B_{1}^{*},
 B_{2}^{*} \rb$ both come under the umbrella of Theorem
 \ref{thm:MainSqEst} and the permutations $\lb \Gamma_{J}^{*}, B_{2},
 B_{1} \rb$ and $\lb \Gamma_{J}^{*}, B_{2}^{*}, B_{1}^{*} \rb$ are
 handled by Proposition \ref{prop:PropSwap}.
\end{proof}

From the upper and lower estimate of the previous theorem, Theorem
\ref{thm:BoundedHolomorphic} then implies that $\Pi_{J,B}$ has a bounded holomorphic functional calculus.

\begin{thm} 
 \label{thm:PiBLCalc} 
 Suppose that $B_{1} = I$. Suppose further that
(H8D$\alpha$) is satisfied for some $\alpha \in (1,2]$ and either
(H8J) or (H8J$\alpha$) is satisfied.  Then $\Pi_{J,B}$ has a bounded
$H^{\infty}\br{S^{o}_{\mu}}$-holomorphic functional calculus for any
$\mu \in \br{\omega_{J}, \frac{\pi}{2}}$, where
$$
\omega_{J} := \frac{1}{2}\br{\sup_{u \in R \br{\Gamma_{J}^{*}}
    \setminus \lb 0 \rb} \abs{\mathrm{arg} \langle B_{1} u, u \rangle}
+ \sup_{u \in R \br{\Gamma_{J}} \setminus \lb 0 \rb} \abs{\mathrm{arg} \langle
  B_{2} u, u
\rangle}}.
$$
 In particular,
$$
\norm{f(\Pi_{J,B})} \lesssim c_{\alpha}^{J} \cdot \sup_{\zeta \in
  S^{o}_{\mu}} \abs{f(\zeta)}
$$
for any $f \in H^{\infty}_{0} \br{S^{o}_{\mu}}$. 

\end{thm}

\begin{cor}
  \label{cor:LBJCalc}
 Suppose that $B_{1} = I$. Suppose further that
(H8D$\alpha$) is satisfied for some $\alpha \in (1,2]$ and either
(H8J) or (H8J$\alpha$) is satisfied. The operator
$$
L_{B}^{J} := D_{J}^{*} \hat{A} D_{J}
$$
is a $2 \omega_{J}$-sectorial operator with  a bounded
$H^{\infty}\br{S^{o}_{2\mu+}}$-functional calculus for any $\mu \in
\br{\omega_{J}, \frac{\pi}{2}}$ and
\begin{equation}
  \label{eqtn:LBJCalc}
\norm{\sqrt{L_{B}^{J}} u} \simeq c_{\alpha}^{J} \cdot \br{\norm{J u} +
  \norm{D u}}
\end{equation}
for all $u \in D \br{J} \cap D(D)$.
  \end{cor}

\begin{proof}  
The bounded $H^{\infty}\br{S^{o}_{2 \mu+}}$-functional calculus of
$L_{B}^{J}$ follows from the bounded
$H^{\infty}\br{S^{o}_{\mu}}$-functional calculus of $\Pi_{J,B}$ and
that $\Pi_{J,B}^{2}$ is of the form
$$
\Pi_{J,B}^{2} = \br{\begin{array}{c c c} 
                  L_{B}^{J} & 0 & 0 \\
                  0 & * & * \\
                  0 & * & *
 \end{array}}.
$$
The estimate \eqref{eqtn:LBJCalc} follows from Corollary
\ref{cor:Kato} applied to the operator $\Pi_{J,B}$ and an element
$(u,0,0) \in \mathcal{H}$ with $u \in D \br{J} \cap D(D)$.
\end{proof}

\section{Square Function Estimates}
\label{sec:Square}

In this section, a proof of our main result, Theorem
\ref{thm:MainSqEst}, will be provided. The first part of the proof
consists in showing that the operators $P_{t}^{J}$ can effectively be
diagonalised when estimating square function norms from above. This diagonalisation
will be applied to bound the second component of our square
function norm when (H8J$\alpha$) is satisfied thus proving the third
part of Theorem \ref{thm:MainSqEst}.
To prove the fourth and most challenging part of Theorem \ref{thm:MainSqEst} we will use this diagonalisation and an argument similar to
the original result \cite{axelsson2006quadratic}. That is, a $T(1)$-type
reduction will be applied to reduce the third component of the square
function norm to a Carleson measure norm which will subsequently be
proved to be bounded.

\subsection{Diagonalisation of the $P_{t}^{J}$ Operators}

Define the bounded operator $\mathcal{P}_{t}^{J} : \mathcal{H}
\rightarrow \mathcal{H}$ through
$$
\mathcal{P}_{t}^{J} := \br{\begin{array}{c c c} 
                             \br{I + t^{2} D_{J}^{*} D_{J}}^{-1} & 0 & 0 \\
                             0 & \br{I + t^{2} J J^{*}}^{-1} & 0 \\
                             0 & 0 & \br{I + t^{2} D D^{*}}^{-1}
 \end{array}},
$$
for $t > 0$. Observe that since the operators $D_{J}^{*} D_{J}$,
$J J^{*}$ and $D D^{*}$ are all self-adjoint, it follows from Corollary
\ref{cor:SA} that square function estimates hold for each of these
operators with constant independent of $J$ and $D$. Therefore each of
these operators possesses a bounded holomorphic functional calculus
with constant independent of $J$ and $D$. It can be deduced from this
that the operators $\mathcal{P}_{t}^{J}$ are uniformly $L^{2}$-bounded
with constant independent of $J$ and $D$.

Let us prove that the operator $P_{t}^{J}$ can be
effectively diagonalised when evaluating square function
estimates. Specifically, the following theorem will be proved.

\begin{thm} 
 \label{thm:Diagonalisation} 
 Suppose that (H8D$\alpha$) is satisfied for some $\alpha \in (1,2]$
 then
 $$
\int^{\infty}_{0} \norm{\mathbb{P}_{3}(\mathcal{P}_{t}^{J} -
  P_{t}^{J})u}^{2} \frac{dt}{t} \lesssim (1 + (b^{D}_{\alpha})^{2})
\cdot \norm{u}^{2}
$$
for all $u \in R(\Gamma_{J})$. Suppose, in addition, that
(H8J$\alpha$) is also satisfied. Then
 \begin{equation}
   \label{eqtn:DiagonalEst}
\int^{\infty}_{0} \norm{\br{\mathcal{P}_{t}^{J} -
    P_{t}^{J} }u}^{2} \frac{dt}{t}\lesssim c_{\alpha}^{J}
\cdot \norm{u}^{2}
\end{equation}
for all $u \in R \br{\Gamma_{J}}$.
\end{thm}

 Such a
diagonalisation will aid us tremendously in the bounding of our main
square function estimate \eqref{eqtn:MainSqEst} for the second and
third component.  This theorem will be proved by inspecting each component separately.

\begin{rmk}
  \label{rmk:TrivialFirst}
It is easy to see that the diagonalisation estimate
\eqref{eqtn:DiagonalEst} is trivially satisfied on the first component
for any $u \in \mathcal{H}$ since
$\mathbb{P}_{1} \mathcal{P}_{t}^{J} = \mathbb{P}_{1} P_{t}^{J}$. 
\end{rmk}

  \begin{prop} 
 \label{prop:LowerOrder1} 
 For any $u \in \mathcal{H}$,
 \begin{equation}
   \label{eqtn:LowerOrder1}
\int^{\infty}_{0} \norm{\mathcal{P}_{t}^{J}
  \br{P_{t}^{J} - I} u}^{2} \frac{dt}{t} \lesssim \norm{u}^{2}.
 \end{equation}
\end{prop}

\begin{proof}
    The estimate is trivially satisfied for any $u \in N(\Pi_{J})$ since
\begin{align*}\begin{split}  
    \br{P_{t}^{J} - I} u &= \br{\br{I + t^{2} \Pi^{2}_{J}}^{-1} - I} u \\
    &= \br{I + t^{2} \Pi^{2}_{J}}^{-1} \br{I - \br{I + t^{2} \Pi^{2}_{J}}} u \\
    &= 0
  \end{split}\end{align*}
for any $t > 0$. Suppose that
$u \in \overline{R \br{\Pi_{J}}}$. On applying the resolution of the
identity, Proposition \ref{prop:Resolution},
\begin{align*}\begin{split}  
\int^{\infty}_{0} \norm{\mathcal{P}_{t}^{J}\br{P_{t}^{J} - I} u}^{2}
\frac{dt}{t} &= \int^{\infty}_{0} \norm{\mathcal{P}_{t}^{J} \br{P_{t}^{J} - I}
  2 \int^{\infty}_{0} \br{Q_{s}^{J}}^{2} u \frac{ds}{s}}^{2}
\frac{dt}{t} \\
&\lesssim \int^{\infty}_{0} \br{\int^{\infty}_{0}
  \norm{\mathcal{P}_{t}^{J}
    \br{P_{t}^{J} - I} \br{Q_{s}^{J}}^{2} u} \frac{ds}{s}}^{2}  \frac{dt}{t}.
\end{split}\end{align*}
The Cauchy-Schwarz inequality leads to
\begin{align}\begin{split}  
    \label{eqtn:Middle0}
     \int^{\infty}_{0} &\norm{\mathcal{P}_{t}^{J} \br{P_{t}^{J} - I} u}^{2}
 \frac{dt}{t} \lesssim \\ & \quad \int^{\infty}_{0} \br{\int^{\infty}_{0}
   \norm{\mathcal{P}_{t}^{J} \br{P_{t}^{J} - I} Q_{s}^{J}}  \frac{ds}{s}} \cdot
 \br{\int^{\infty}_{0} \norm{ \mathcal{P}_{t}^{J} \br{P_{t}^{J} - I} Q_{s}^{J}} 
   \norm{Q_{s}^{J}u}^{2}  \frac{ds}{s}} \frac{dt}{t}.
 \end{split}\end{align}
Let's estimate the term $\norm{\mathcal{P}_{t}^{J} \br{P_{t}^{J} - I} Q_{s}^{J}} $. First
assume that $t \leq s$. On noting that $\br{P_{t}^{J} - I} Q_{s}^{J} =
\frac{t}{s} Q_{t}^{J} \br{P_{s}^{J} - I}$ we obtain
\begin{equation} 
 \label{eqtn:Middle01} 
 \norm{\mathcal{P}_{t}^{J} \br{P_{t}^{J} - I} Q_{s}^{J}} \lesssim \norm{\br{P_{t}^{J} - I}
   Q_{s}^{J}} \lesssim \frac{t}{s} \norm{Q_{t}^{J} \br{P_{s}^{J} - I}} \lesssim \frac{t}{s}.
 \end{equation}
Next, suppose that $t > s$. Then the equality $P_{t}^{J} Q_{s}^{J} =
\frac{s}{t} Q_{t}^{J} P_{s}^{J}$ gives
$$
\norm{\mathcal{P}_{t}^{J} \br{P_{t}^{J} - I}Q_{s}^{J}} \lesssim \norm{P_{t}^{J} Q_{s}^{J}} +
\norm{\mathcal{P}_{t}^{J} Q_{s}^{J}} \lesssim \frac{s}{t}
+ \norm{\mathcal{P}_{t}^{J} Q_{s}^{J}}.
$$
The term $\mathcal{P}_{t}^{J} Q_{s}^{J}$ will be considered
component-wise. For the first component, recall that $\mathbb{P}_{1}
\mathcal{P}_{t}^{J} = \mathbb{P}_{1} P_{t}^{J}$ and observe that
\begin{align*}\begin{split}  
 \norm{\mathbb{P}_{1} \mathcal{P}_{t}^{J} Q_{s}^{J}} &=
 \norm{\mathbb{P}_{1} P_{t}^{J} s \Pi_{J} P_{s}^{J}} \\
 &= \frac{s}{t} \norm{\mathbb{P}_{1} P_{t}^{J} t \Pi_{J} P_{s}^{J}} \\
 &= \frac{s}{t} \norm{\mathbb{P}_{1} Q_{t}^{J} P_{s}^{J}} \\
 &\lesssim \frac{s}{t}.
 \end{split}\end{align*}
For the second component, note that
$$
\mathbb{P}_{2} \mathcal{P}_{t}^{J} = \br{I + t^{2} S_{J}^{2}}^{-1} \mathbb{P}_{2}.
$$
Also observe
$$
\mathbb{P}_{2} \Pi_{J} u = \mathbb{P}_{2} \Pi_{J} \mathbb{P}_{1} u =
\mathbb{P}_{2} S_{J} \mathbb{P}_{1} u
$$
for $u \in D(\Pi_{J})$.
This gives
\begin{align*}\begin{split}  
 \norm{\mathbb{P}_{2} \mathcal{P}_{t}^{J} Q_{s}^{J}} &=
 \norm{\br{I + t^{2} S_{J}^{2}}^{-1} \mathbb{P}_{2} s \Pi_{J}
   P_{s}^{J}} \\
 &= \norm{(I + t^{2} S_{J}^{2})^{-1} \mathbb{P}_{2} s S_{J}
   \mathbb{P}_{1} P_{s}^{J}} \\
 &= \frac{s}{t} \norm{\mathbb{P}_{2} t S_{J} (I + t^{2}
   S_{J}^{2})^{-1} \mathbb{P}_{1} P_{s}^{J}} \\
 &\lesssim \frac{s}{t},
\end{split}\end{align*}
where the last line follows from the fact that $S_{J}$ is self-adjoint
and therefore possesses a bounded holomorphic functional calculus with
constant independent of $J$. Lastly, for the third component, we have
$$
\mathbb{P}_{3} \mathcal{P}_{t}^{J} = \mathbb{P}_{3} P_{t}^{0} =
P_{t}^{0} \mathbb{P}_{3}
$$
and
$$
\mathbb{P}_{3} \Pi_{J} u = \mathbb{P}_{3} \Pi_{J} \mathbb{P}_{1} u =
\mathbb{P}_{3} \Pi_{0} \mathbb{P}_{1} u
$$
for $u \in D(\Pi_{J})$. This leads to
\begin{align*}\begin{split}  
 \norm{\mathbb{P}_{3} \mathcal{P}_{t}^{J} Q_{s}^{J}}
 &= \norm{P_{t}^{0} \mathbb{P}_{3} s \Pi_{J} P_{s}^{J}} \\
 &= \norm{P_{t}^{0} \mathbb{P}_{3} s \Pi_{0} \mathbb{P}_{1} P_{s}^{J}}
 \\
 &= \frac{s}{t} \norm{\mathbb{P}_{3} t \Pi_{0} P_{t}^{0}
   \mathbb{P}_{1} P_{s}^{J}} \\
 &\lesssim \frac{s}{t}.
 \end{split}\end{align*}
Putting everything together gives
\begin{equation}
  \label{eqtn:PrincMin}
  \norm{ \mathcal{P}_{t}^{J} \br{P_{t}^{J} - I} Q_{s}^{J}} \lesssim \min \lb
  \frac{t}{s}, \frac{s}{t} \rb.
\end{equation}
This bound can then be applied to \eqref{eqtn:Middle0} to give \eqref{eqtn:LowerOrder1}.
\end{proof}

\begin{prop} 
 \label{prop:LowerOrder2} 
 Suppose that the condition (H8J$\alpha$) is satisfied for some
 $\alpha \in (1,2]$. Then
 $$
\int^{\infty}_{0} \norm{\mathbb{P}_{2} \br{I - \mathcal{P}_{t}^{J}}
  P_{t}^{J} u}^{2} \frac{dt}{t} \lesssim (\alpha - 1)^{-1} \br{b_{\alpha}^{J}}^{2} \cdot \norm{u}^{2}
$$
for any $u \in R \br{\Gamma_{J}}$.
 \end{prop}

 \begin{proof}
 It will first be proved that $\br{P_{t}^{J} u}_{2} \in D
   \br{J^{\alpha - 1}}$. Since $\Gamma_{J}$ commutes with $P_{t}^{J}$ we must have
$$
P_{t}^{J} u = \Gamma_{J} \br{\begin{array}{c} 
 v \\ 0 \\ 0
 \end{array}}
$$
for some $\br{v,0,0} \in D \br{\Gamma_{J}}$. This then gives
$\br{P_{t}^{J} u}_{2} = J v$. Therefore $\br{P_{t}^{J}u}_{2} \in
D(J^{\alpha - 1})$ if and only if $v \in D(J^{\alpha})$. We know that
$P_{t}^{J} u \in D(\Pi_{J})$ which implies that $v \in D(D_{J}^{*}
D_{J})$ and therefore $v \in D \br{(D_{J}^{*}
  D_{J})^{\frac{\alpha}{2}}}$. Our hypothesis
(H8J$\alpha$) then tells us that $v \in D(J^{\alpha})$
which allows us to conclude, using the previous reasoning, that
$\br{P_{t}^{J}u}_{2} \in D(J^{\alpha - 1})$.

Since $(P_{t}^{J}u)_{2} \in D(J^{\alpha - 1})$, it follows that
   \begin{align}\begin{split}
       \label{eqtn:prop:LowerOrder21}
 \mathbb{P}_{2} \br{I - \mathcal{P}_{t}^{J}}P_{t}^{J}u &= \br{0,t^{2} J^{2}
 \br{I + t^{2} J^{2}}^{-1} \br{ P_{t}^{J} u}_{2},0} \\
 &= \br{0, g_{\alpha}^{t}\br{J} t^{\alpha - 1} J^{\alpha - 1} \br{
 P_{t}^{J} u}_{2},0},
\end{split}\end{align}
where $g_{\alpha}^{t} : S^{o}_{\mu} \rightarrow \C$ is the bounded
holomorphic function defined through
\begin{equation}
  \label{eqtn:galphat}
g_{\alpha}^{t}(z) := \frac{t^{2} z^{2}}{\br{I + t^{2} z^{2}} t^{\alpha -
    1} \br{\sqrt{z^{2}}}^{\alpha - 1}}.
\end{equation}
As $J$ is self-adjoint, it follows from Corollary \ref{cor:SA} that $J$ possesses a bounded holomorphic functional calculus with
constant independent of $J$. This, together with
\eqref{eqtn:prop:LowerOrder21}, gives
$$
\norm{\mathbb{P}_{2} \br{I - \mathcal{P}_{t}^{J}} P_{t}^{J} u}
\lesssim \norm{t^{\alpha - 1} J^{\alpha - 1} \br{P_{t}^{J} u}_{2}}.
$$
On applying (H8J$\alpha$),
\begin{align*}\begin{split}  
 \norm{J^{\alpha - 1} \br{P_{t}^{J} u}_{2}} &= \norm{J^{\alpha}
   v} \\
 &\leq b_{\alpha}^{J} \cdot \norm{\br{D_{J}^{*}D_{J}}^{\frac{\alpha}{2}} v} \\
 &= b_{\alpha}^{J} \cdot \norm{\abs{\Pi_{J}}^{\alpha}
   (v,0,0)} \\
   &\simeq b_{\alpha}^{J} \cdot \norm{\abs{\Pi_{J}}^{\alpha - 1}
     \Pi_{J} (v,0,0)} \\
             &= b_{\alpha}^{J} \cdot \norm{\abs{\Pi_{J}}^{\alpha - 1}
               P_{t}^{J} u},
           \end{split}\end{align*}
         where $\abs{\Pi_{J}} := \sqrt{\Pi_{J}^{2}}$ and in the fourth line we applied the bounded holomorphic
         functional calculus of the operator $\Pi_{J}$. Therefore
         \begin{align*}\begin{split}  
 \int^{\infty}_{0} \norm{\mathbb{P}_{2} \br{I - \mathcal{P}_{t}^{J}}
   P_{t}^{J} u}^{2} \frac{dt}{t} &\lesssim \br{b_{\alpha}^{J}}^{2} \cdot
 \int^{\infty}_{0}\norm{t^{\alpha - 1} \abs{\Pi_{J}}^{\alpha - 1}
   P_{t}^{J} u}^{2} \frac{dt}{t} \\
 &\lesssim (\alpha - 1)^{-1} \br{b_{\alpha}^{J}}^{2} \cdot \norm{u}^{2},
\end{split}\end{align*}
where we used the fact that $\Pi_{J}$ is self-adjoint
and Corollary \ref{cor:SA} in the last line.
\end{proof}

 \begin{prop} 
 \label{prop:ThirdAtPt1} 
 Suppose that (H8D$\alpha$) is satisfied for some $\alpha \in
 (1,2]$. Then
 \begin{equation}
   \label{eqtn:ThirdAtPt1}
\int^{\infty}_{0} \norm{\mathbb{P}_{3} \br{I - \mathcal{P}_{t}^{J}}
  P_{t}^{J} u}^{2} \frac{dt}{t} \lesssim (\alpha - 1)^{-1} \br{b_{\alpha}^{D}}^{2} \cdot \norm{u}^{2}
\end{equation}
for all $u \in R \br{\Gamma_{J}}$.
\end{prop}

\begin{proof}  
First note that the left-hand side of \eqref{eqtn:ThirdAtPt1} can be
re-written as
 \begin{align*}\begin{split}  
 \int^{\infty}_{0} \norm{\mathbb{P}_{3} \br{I - \mathcal{P}_{t}^{J}}
   P_{t}^{J} u}^{2} \frac{dt}{t} &= \int^{\infty}_{0}
 \norm{\mathbb{P}_{3} \br{I - P_{t}^{0}} P_{t}^{J}u}^{2} \frac{dt}{t}
 \\
 &= \int^{\infty}_{0}\norm{\br{I - P_{t}^{0}} \mathbb{P}_{3}
   P_{t}^{J}u}^{2} \frac{dt}{t}.
\end{split}\end{align*}
It will be shown that $\mathbb{P}_{3} P_{t}^{J} u \in D
\br{\abs{\Pi_{0}}^{\alpha - 1}}$. Since $\Gamma_{J}$ commutes with the
operator $P_{t}^{J}$ and $u \in R(\Gamma_{J})$, we must have
$P_{t}^{J} u = \Gamma_{J} P_{t}^{J}(v,0,0)$ for some $(v,0,0) \in
D(\Gamma_{J})$. This implies that
$$
\mathbb{P}_{3} P_{t}^{J} u = \mathbb{P}_{3} \Gamma_{J} P_{t}^{J}
(v,0,0) = \mathbb{P}_{3} \Pi_{0} P_{t}^{J}(v,0,0)
$$
and therefore $\mathbb{P}_{3} P_{t}^{J} u \in D
\br{\abs{\Pi_{0}}^{\alpha - 1}}$ will follow from
$P_{t}^{J}(v,0,0) \in D(\abs{\Pi_{0}}^{\alpha - 1} \mathbb{P}_{3}
\Pi_{0})$ which itself will follow from $P_{t}^{J}(v,0,0) \in D (\abs{\Pi_{0}}^{\alpha - 1} \Pi_{0})$. The bounded
holomorphic functional calculus of the operator $\Pi_{0}$ tells us
that $D(\abs{\Pi_{0}}^{\alpha - 1} \Pi_{0}) = D
(\abs{\Pi_{0}}^{\alpha})$ and it is therefore sufficient to prove that
$P_{t}^{J} (v,0,0) \in D(\abs{\Pi_{0}}^{\alpha})$. Since
$P_{t}^{J}(v,0,0)$ is non-zero only in the first component, this in turn is
equivalent to proving that
$$
(P_{t}^{J}(v,0,0))_{1} \in D\br{(D^{*}D)^{\frac{\alpha}{2}}}.
$$
This however follows directly from our hypothesis
(H8D$\alpha$) and the fact that $(P_{t}^{J}(v,0,0))_{1}
\in D \br{D_{J}^{*}D_{J}} \subset D\br{(D_{J}^{*}D_{J})^{\frac{\alpha}{2}}}$. This completes the
proof of our claim that $\mathbb{P}_{3} P_{t}^{J}u \in D
\br{\abs{\Pi_{0}}^{\alpha - 1}}$ .

Since $\mathbb{P}_{3} P_{t}^{J} u \in D \br{\abs{\Pi_{0}}^{\alpha - 1}}$ we must have
$$
\br{I - P_{t}^{0}} \mathbb{P}_{3} P_{t}^{J} u = g_{\alpha}^{t} \br{\Pi_{0}}
t^{\alpha - 1} \abs{\Pi_{0}}^{\alpha - 1} \mathbb{P}_{3} P_{t}^{J} u,
$$
where $g_{\alpha}^{t}$ is as defined in \eqref{eqtn:galphat}. 
From the bounded holomorphic functional calculus of $\Pi_{0}$ we then obtain
$$
 \int^{\infty}_{0} \norm{\mathbb{P}_{3} \br{I - \mathcal{P}_{t}^{J}}
   P_{t}^{J} u}^{2} \frac{dt}{t} \lesssim \int^{\infty}_{0}
 \norm{\mathbb{P}_{3} t^{\alpha - 1} \abs{\Pi_{0}}^{\alpha - 1}
   P_{t}^{J} u}^{2} \frac{dt}{t}.
$$
On recalling that $P_{t}^{J} u = \Gamma_{J} P_{t}^{J} \br{v, 0, 0}$ for some $\br{v, 0 ,
  0} \in D \br{\Gamma_{J}}$,
\begin{align*}\begin{split}  
 \int^{\infty}_{0}
 \norm{\mathbb{P}_{3} t^{\alpha - 1} \abs{\Pi_{0}}^{\alpha - 1}
   P_{t}^{J} u}^{2} \frac{dt}{t} &= \int^{\infty}_{0}
 \norm{\mathbb{P}_{3} t^{\alpha - 1} \abs{\Pi_{0}}^{\alpha - 1}
   \Gamma_{J} P_{t}^{J} (v,0,0)}^{2} \frac{dt}{t} \\
                        &= \int^{\infty}_{0}\norm{\mathbb{P}_{3}
                          t^{\alpha - 1} \abs{\Pi_{0}}^{\alpha - 1}
                          \Pi_{0} P_{t}^{J} (v,0,0)}^{2} \frac{dt}{t}.
\end{split}\end{align*}
On exploiting the bounded holomorphic functional calculus of the
operator $\Pi_{0}$ once more,
$$
 \int^{\infty}_{0}\norm{\mathbb{P}_{3}
                          t^{\alpha - 1} \abs{\Pi_{0}}^{\alpha - 1}
                          \Pi_{0} P_{t}^{J} (v,0,0)}^{2} \frac{dt}{t} \lesssim \int^{\infty}_{0}
\norm{t^{\alpha - 1} \abs{\Pi_{0}}^{\alpha} P_{t}^{J}
  (v,0,0)}^{2} \frac{dt}{t}.
$$
Observe that since $P_{t}^{J}(v,0,0)$ is non-zero only in the first entry,
\begin{align*}\begin{split}  
 \norm{\abs{\Pi_{0}}^{\alpha} P_{t}^{J} (v,0,0)} &= \norm{\br{D^{*}D}^{\frac{\alpha}{2}} \br{P_{t}^{J}(v,0,0)}_{1}} \\
                                                                  &\leq
                                                                  b_{\alpha}^{D} \cdot
                                                                  \norm{\br{D_{J}^{*}
                                                                    D_{J}}^{\frac{\alpha}{2}} \br{P_{t}^{J}(v,0,0)}_{1}}
                                                                  \\
                                                                  &=
                                                                  b_{\alpha}^{D}
                                                                  \cdot\norm{\abs{\Pi_{J}}^{\alpha} P_{t}^{J} (v,0,0)}.
\end{split}\end{align*}
On then applying the bounded holomorphic functional calculus of the
operator $\Pi_{J}$,
\begin{align*}\begin{split}  
 \br{b_{\alpha}^{D}}^{2} \cdot \int^{\infty}_{0}
  \norm{t^{\alpha - 1} \abs{\Pi_{J}}^{\alpha} P_{t}^{J} (v,0,0)}^{2} \frac{dt}{t} &\lesssim \br{b_{\alpha}^{D}}^{2}
\int^{\infty}_{0} \norm{t^{\alpha - 1} \abs{\Pi_{J}}^{\alpha - 1}
  \Pi_{J} P_{t}^{J} (v,0,0)}^{2} \frac{dt}{t} \\
                    &= \br{b_{\alpha}^{D}}^{2}  \cdot \int^{\infty}_{0} \norm{t^{\alpha - 1}
                      \abs{\Pi_{J}}^{\alpha - 1} P_{t}^{J} u}^{2}
                    \frac{dt}{t} \\
                    &\lesssim (\alpha - 1)^{-1} \br{b_{\alpha}^{D}}^{2} \cdot \norm{u}^{2},
                  \end{split}\end{align*}
                where we used the fact that $\Pi_{J}$ is self-adjoint
                and Corollary \ref{cor:SA} in the final line.
              \end{proof}

Combining Propositions \ref{prop:LowerOrder1}, \ref{prop:LowerOrder2} and \ref{prop:ThirdAtPt1}
together then gives Theorem \ref{thm:Diagonalisation}. With this
diagonalisation in hand we can now return to our proof of Theorem
\ref{thm:MainSqEst}. In particular the second component of our square
function norm will now be bounded.

\vspace*{0.1in}

\textsc{Proof of Theorem} \ref{thm:MainSqEst}.\ref{Main3}.  On splitting the second component of our square function norm from above,
   \begin{align*}\begin{split}  
 \int^{\infty}_{0} \norm{\Theta^{J,B}_{t} \mathbb{P}_{2} P_{t}^{J}
   u}^{2} \frac{dt}{t} \lesssim \int^{\infty}_{0}
 \norm{\Theta^{J,B}_{t} \mathbb{P}_{2} \br{\mathcal{P}_{t}^{J} -
     P_{t}^{J}}u}^{2} \frac{dt}{t} + \int^{\infty}_{0}
 \norm{\Theta^{J,B}_{t} \mathbb{P}_{2} \mathcal{P}_{t}^{J} u}^{2} \frac{dt}{t}.
\end{split}\end{align*}
The uniform $L^{2}$-boundedness of the operators $\Theta^{J,B}_{t}$
together with Propositions \ref{prop:LowerOrder1} and
\ref{prop:LowerOrder2} give
$$
\int^{\infty}_{0}\norm{\Theta^{J,B}_{t} \mathbb{P}_{2}
  \br{\mathcal{P}_{t}^{J} - P_{t}^{J}} u}^{2} \frac{dt}{t} \lesssim
\br{1 + \br{b_{\alpha}^{J}}^{2}} \br{\alpha - 1}^{-1} \cdot \norm{u}^{2}.
$$
It remains to bound the second term in the splitting. In order to do
so, it will first be shown that $\mathbb{P}_{2}
\mathcal{P}_{t}^{J} u \in D (\Gamma^{*}_{J,B})$. From the definition
of the adjoint, $\mathbb{P}_{2} \mathcal{P}_{t}^{J} u \in
D(\Gamma^{*}_{J,B})$ if and only if there exists some $u' \in
\mathcal{H}$ such that
$$
\langle \Gamma_{J,B} w, \mathbb{P}_{2} \mathcal{P}_{t}^{J} u \rangle =
\langle w, u' \rangle
$$
for all $w \in D(\Gamma_{J,B})$, where $\Gamma_{J,B} = B_{2}^{*}
\Gamma_{J} B_{1}^{*}$. For $w \in D(\Gamma_{J,B})$,
\begin{align}\begin{split}
    \label{eqtn:AdjointSecondComp}
 \langle \Gamma_{J,B} w, \mathbb{P}_{2} \mathcal{P}_{t}^{J} u \rangle
 &= \langle B_{2}^{*} \Gamma_{J} B_{1}^{*} w, \mathbb{P}_{2}
 \mathcal{P}_{t}^{J} u \rangle \\
 &= \langle \Gamma_{J} B_{1}^{*} w, B_{2} \mathbb{P}_{2}
 \mathcal{P}_{t}^{J} u \rangle \\
 &= \langle M_{J} B_{1}^{*} w, B_{2} \mathbb{P}_{2}
 \mathcal{P}_{t}^{J} u \rangle,
\end{split}\end{align}
where in the last line we used the fact that $A_{32} = 0$ by
(H8J$\alpha$) and therefore $B_{2} \mathbb{P}_{2} =
\mathbb{P}_{2} B_{2} \mathbb{P}_{2}$.
This proves that $\mathbb{P}_{2}\mathcal{P}_{t}^{J}u \in
D(\Gamma^{*}_{J,B})$ will follow from $B_{2} \mathbb{P}_{2}
\mathcal{P}_{t}^{J} u \in D \br{M_{J}^{*}}$ which, in turn, will
follow from $(\mathcal{P}_{t}^{J}u)_{2} \in D(J^{*}A_{22}) = D(J A_{22})$.

Note that $u \in R \br{\Gamma_{J}}$ implies
that $u = \Gamma_{J}(v,0,0)$ for some $(v,0,0)\in D
\br{\Gamma_{J}}$. Then
\begin{align*}\begin{split}  
    \br{\mathcal{P}_{t}^{J}u}_{2} &= \br{I + t^{2} J^{2}}^{-1} J v \\
    &= J \br{I + t^{2} J^{2}}^{-1} v.
 \end{split}\end{align*}
(H8J$\alpha$) states that $D(J^{2}) = D(J A_{22}J)$ with $\norm{J^{2}
  \tilde{u}} = \norm{J A_{22} J \tilde{u}}$ for $\tilde{u} \in D
\br{J^{2}}$. Since $\br{I + t^{2} J^{2}}^{-1}v \in D \br{J^{2}}$ we
must have $\br{I + t^{2} J^{2}}^{-1}v \in D \br{J A_{22}J}$ and
therefore $\br{\mathcal{P}_{t}^{J} u}_{2} \in D \br{J A_{22}}$. This
allows us to conclude that $\mathbb{P}_{2} \mathcal{P}_{t}^{J}u \in D
\br{\Gamma_{J,B}^{*}}$. Moreover, from
\eqref{eqtn:AdjointSecondComp} we know that $\Gamma_{J,B}^{*}
\mathbb{P}_{2} \mathcal{P}_{t}^{J}u = B_{1} M_{J}^{*} B_{2}
\mathbb{P}_{2} \mathcal{P}_{t}^{J}u$ and therefore
\begin{align*}\begin{split}  
 \norm{\Gamma^{*}_{J,B} \mathbb{P}_{2} \mathcal{P}_{t}^{J}u} &=
 \norm{B_{1} M_{J}^{*} B_{2} \mathbb{P}_{2} \mathcal{P}_{t}^{J}u} \\
 &\lesssim \norm{M_{J}^{*} B_{2} \mathbb{P}_{2} \mathcal{P}_{t}^{J} u} \\
 &= \norm{J A_{22} J \br{I + t^{2}J^{2}}^{-1}v} \\
 &= \norm{J^{2} \br{I + t^{2} J^{2}}^{-1} v} \\
 &= \norm{J \br{I + t^{2} J^{2}}^{-1} u_{2}}.
\end{split}\end{align*}
This together with Lemma \ref{lem:Commutation} gives
\begin{align*}\begin{split}  
 \int^{\infty}_{0} \norm{\Theta^{J,B}_{t} \mathbb{P}_{2} \mathcal{P}_{t}^{J} u}^{2}
 \frac{dt}{t} &= \int^{\infty}_{0} \norm{P_{t}^{J,B} t
   \Gamma^{*}_{J,B} \mathbb{P}_{2}
   \mathcal{P}_{t}^{J} u}^{2} \frac{dt}{t} \\
 &\lesssim \int^{\infty}_{0} \norm{t \Gamma_{J,B}^{*}
   \mathbb{P}_{2}   \mathcal{P}_{t}^{J} u}^{2} \frac{dt}{t} \\
 &\lesssim \int^{\infty}_{0} \norm{t J \br{I + t^{2}J^{2}}^{-1} u_{2}} \frac{dt}{t}.
\end{split}\end{align*}
The theorem then follows from the fact that $J$ is self-adjoint and
therefore satisfies square function estimates with constant
independent of $J$ by Corollary \ref{cor:SA}.
\hfill \BoldSquare \vspace{5pt}

\subsection{The Third Component}
\label{subsec:ThirdComponent}

This section is dedicated to bounding the third component of our
square function norm and thus proving the fourth and final part of
Theorem \ref{thm:MainSqEst}. Specifically, it will be proved that when
(H8D$\alpha$) is satisfied for some $\alpha \in (1,2]$ and either
(H8J) or (H8J$\alpha$) is satisfied the estimate
\begin{equation} 
 \label{eqtn:ReducedMain} 
 \int^{\infty}_{0} \norm{\Theta^{J,B}_{t} \mathbb{P}_{3} P_{t}^{J} u}^{2}
 \frac{dt}{t} \lesssim c_{\alpha}^{J} \cdot \norm{u}^{2}
\end{equation}
will hold for any $u \in R \br{\Gamma_{J}}$. A similar argument to that of
\cite{axelsson2006quadratic} will be used, but one will need to keep
track of the effect of the projection $\mathbb{P}_{3}$.

\subsubsection{$T(1)$-Reduction}

Our first step towards a $T(1)$-reduction is to use the splitting
$$
\int^{\infty}_{0} \norm{\Theta^{J,B}_{t} \mathbb{P}_{3} P_{t}^{J}
  u}^{2} \frac{dt}{t} \lesssim \int^{\infty}_{0}
\norm{\Theta^{J,B}_{t} \mathbb{P}_{3} \br{P_{t}^{J} -
    \mathcal{P}_{t}^{J}} u}^{2} \frac{dt}{t} + \int^{\infty}_{0}
\norm{\Theta^{J,B}_{t} \mathbb{P}_{3} \mathcal{P}_{t}^{J} u}^{2} \frac{dt}{t}.
$$
The uniform $L^{2}$-boundedness of the operators $\Theta^{J,B}_{t}$
and Theorem \ref{thm:Diagonalisation} can
be applied to the first term to obtain
$$
\int^{\infty}_{0} \norm{\Theta^{J,B}_{t} \mathbb{P}_{3}
  \br{\mathcal{P}_{t}^{J} - P_{t}^{J}}u}^{2} \frac{dt}{t} \lesssim
c_{\alpha}^{J} \cdot \norm{u}^{2}.
$$
On recalling that $\mathbb{P}_{3} \mathcal{P}_{t}^{J} = \mathbb{P}_{3}
P_{t}^{0}$, this reduces the task of proving our square function
estimate to obtaining the  bound
$$
\int^{\infty}_{0}\norm{\Theta^{J,B}_{t} \mathbb{P}_{3} P_{t}^{0}
  u}^{2} \frac{dt}{t} \lesssim c_{\alpha}^{J} \cdot \norm{u}^{2}.
$$
Introduce the notation $\tilde{\Theta}^{J,B}_{t}$ to denote the operators $\tilde{\Theta}^{J,B}_{t} := \Theta^{J,B}_{t}
\mathbb{P}_{3}$.
 Let $\gamma^{J,B}_{t}$ and $\tilde{\gamma}^{J,B}_{t}$ denote the
 principal parts of the operators $\Theta^{J,B}_{t}$ and
 $\tilde{\Theta}^{J,B}_{t}$ respectively. That is, they are the
 multiplication operators defined through
 $$
\gamma^{J,B}_{t}(x)w := \Theta^{J,B}_{t}(w)(x) \quad and \quad \tilde{\gamma}^{J,B}_{t}(x)(w) := \br{\Theta^{J,B}_{t} \mathbb{P}_{3}}(w)(x),
$$
for $w \in \C^{N}$ and $x \in \R^{n}$. Evidently we must have
$\tilde{\gamma}^{J,B}_{t}(x)w = \gamma^{J,B}_{t}(x)
\mathbb{P}_{3}w$.

Our square function norm can be reduced to this
principal part by applying the splitting
\begin{equation}
  \label{eqtn:MainProof1}
 \int^{\infty}_{0} \norm{\tilde{\Theta}^{J,B}_{t} P_{t}^{0} u}^{2}
 \frac{dt}{t} 
 \lesssim \int^{\infty}_{0} \norm{\br{\tilde{\Theta}^{J,B}_{t} -
     \tilde{\gamma}^{J,B}_{t} A_{t}} P_{t}^{0} u}^{2} \frac{dt}{t} + \int^{\infty}_{0}
 \norm{\tilde{\gamma}^{J,B}_{t} A_{t} P_{t}^{0} u}^{2} \frac{dt}{t}.
\end{equation}
Since the operator $\Theta^{J,B}_{t}$ satisfies the conditions
of Proposition \ref{prop:PrincPart}, it follows that
\begin{align*}\begin{split}  
 \int^{\infty}_{0} \norm{\br{\tilde{\Theta}^{J,B}_{t} -
     \tilde{\gamma}^{J,B}_{t} A_{t}} P_{t}^{0} u}^{2} \frac{dt}{t}
 &= \int^{\infty}_{0} \norm{\br{\Theta^{J,B}_{t} - \gamma^{J,B}_{t}
     A_{t}} \mathbb{P}_{3} P_{t}^{0}u}^{2} \frac{dt}{t} \\
 &\lesssim \int^{\infty}_{0} \norm{t \nabla \mathbb{P}_{3}
   P_{t}^{0}u}^{2} \frac{dt}{t} \\
 &\lesssim \int^{\infty}_{0} \norm{t \Pi_{0} P_{t}^{0} u}^{2}
 \frac{dt}{t} \\
 &= \int^{\infty}_{0}\norm{Q_{t}^{0} u}^{2}
 \frac{dt}{t} \\
 &= \frac{1}{2} \norm{u}^{2},
\end{split}\end{align*}
where the estimate $\norm{\nabla \mathbb{P}_{3} P_{t}^{0} u} \lesssim
\norm{\Pi_{0} P_{t}^{0} u}$  follows from (H8) for the operator
$\Gamma_{0}$. It should be noted that in order to use (H8) we had to
use the fact that $u = \Gamma_{J} v$ for some $v \in D
\br{\Gamma_{J}}$ and therefore
$$
\mathbb{P}_{3} P_{t}^{0}u = P_{t}^{0} \mathbb{P}_{3} \Gamma_{J} v =
P_{t}^{0} \mathbb{P}_{3} \Gamma_{0} v = \Gamma_{0} P_{t}^{0} v \in R \br{\Gamma_{0}}.
$$
Our theorem has thus been reduced to a proof of the following
 square function estimate
 $$
\int^{\infty}_{0} \norm{\tilde{\gamma}^{J,B}_{t} A_{t}P_{t}^{0} u}^{2}
\frac{dt}{t} \lesssim c_{\alpha}^{J} \cdot \norm{u}^{2}.
$$
On splitting from above using the triangle inequality,
\begin{equation}
  \label{eqtn:BeforeCarleson}
\int^{\infty}_{0} \norm{\tilde{\gamma}^{J,B}_{t} A_{t} P_{t}^{0}
  u}^{2} \frac{dt}{t} \lesssim \int^{\infty}_{0}
\norm{\tilde{\gamma}^{J,B}_{t} A_{t} \br{P_{t}^{0} - I}u}^{2}
\frac{dt}{t} + \int^{\infty}_{0} \norm{\tilde{\gamma}^{J,B}_{t}
  A_{t} u}^{2} \frac{dt}{t}.
\end{equation}
Proposition \ref{prop:OffDiagonal} states that the uniform estimate
$\norm{\tilde{\gamma}_{t}^{J,B} A_{t}} \lesssim 1$  is true for all $t
> 0$.
Furthermore, notice that $A_{t}^{2} = A_{t}$ and $\mathbb{P}_{3}
A_{t} = A_{t} \mathbb{P}_{3}$ for all $t > 0$. These facts combine
together to produce
\begin{align*}\begin{split}  
 \int^{\infty}_{0} \norm{\tilde{\gamma}^{J,B}_{t} A_{t} \br{P_{t}^{0}
     - I}u}^{2} \frac{dt}{t} &= \int^{\infty}_{0} \norm{\gamma^{J,B}_{t}
   A_{t} \mathbb{P}_{3} A_{t} \br{P_{t}^{0} - I} u}^{2} \frac{dt}{t}
 \\
 &\lesssim \int^{\infty}_{0} \norm{\mathbb{P}_{3} A_{t} \br{P_{t}^{0}
     - I} u}^{2} \frac{dt}{t}.
\end{split}\end{align*}
According to the argument from Proposition 5.7 of
\cite{axelsson2006quadratic}, this final term can be bounded by
$$
\int^{\infty}_{0} \norm{A_{t} \br{P_{t}^{0} - I}u}^{2}
\frac{dt}{t}\lesssim \norm{u}^{2},
$$
since $\lb \Gamma_{0}, B_{1}, B_{2} \rb$ by hypothesis satisfies (H1)
- (H8).
For the second term in \eqref{eqtn:BeforeCarleson}, apply Carleson's
theorem (\cite[pg.~59]{stein1993harmonic}) to obtain
$$
\int^{\infty}_{0} \norm{\tilde{\gamma}^{J,B}_{t} A_{t} u}^{2}
\frac{dt}{t} \lesssim \norm{\mu}_{\mathcal{C}} \cdot \norm{u}^{2},
$$
where $\mu$ is the measure on $\R^{n+1}$ defined through
$$
d \mu(x,t) := \abs{\tilde{\gamma}^{J,B}_{t}(x)}^{2}  \frac{dx \, dt}{t}
$$
for $x \in \R^{n}$ and $t > 0$ and $\norm{\mu}_{\mathcal{C}}$ denotes
its Carleson norm. The proof of our theorem has thus been reduced to showing that the measure
$\mu$ is a Carleson measure with constant smaller than a multiple of $c_{\alpha}^{J}$.

\subsubsection{Carleson Measure Estimate}
\label{subsec:Carleson}

 Our goal now is to prove the following Carleson measure
 estimate,
  \begin{equation} 
 \label{eqtn:CarlesonMain} 
\sup_{Q \in \Delta} \frac{1}{\abs{Q}} \int^{l(Q)}_{0} \int_{Q}
\abs{\tilde{\gamma}^{J,B}_{t}(x)}^{2} 
 \frac{dx \, dt}{t} \lesssim c_{\alpha}^{J} < \infty.
\end{equation}
  Let $\mathcal{L}_{3}$ denote the subspace
\begin{equation} 
 \label{eqtn:L3} 
 \mathcal{L}_{3} := \lb \nu \in \mathcal{L} \br{\C^{N}} \setminus \lb 0 \rb
 : \nu \mathbb{P}_{3} = \nu \rb.
\end{equation}

By construction, we have $\tilde{\gamma}_{t}^{J,B}(x) \in
\mathcal{L}_{3}$ for any $t > 0$ and $x \in \R^{n}$ since
\begin{align*}\begin{split}  
 \tilde{\gamma}_{t}^{J,B}(x) \mathbb{P}_{3}w &= \br{\Theta^{J,B}_{t}
   \mathbb{P}_{3}} \br{\mathbb{P}_{3}w}(x) \\
 &= \br{\Theta^{J,B}_{t} \mathbb{P}_{3}}(w)(x) \\
 &= \tilde{\gamma}_{t}^{J,B}(x)(w).
\end{split}\end{align*}
Let $\sigma > 0$ be a constant to be determined at a later time. Let $\mathcal{V}$ be a finite set consisting of $\nu \in \mathcal{L}_{3}$ with
$\abs{\nu} = 1$ such that $\cup_{\nu \in \mathcal{V}} K_{\nu} =
\mathcal{L}_{3} \setminus \lb 0 \rb$, where
$$
K_{\nu} := \lb \nu' \in \mathcal{L}_{3} \setminus \lb 0 \rb :
\abs{\frac{\nu'}{\abs{\nu'}} - \nu} \leq \sigma \rb.
$$
Then, in order to prove our Carleson measure estimate
\eqref{eqtn:CarlesonMain}, it is sufficient to fix $\nu \in
\mathcal{V}$ and prove that
\begin{equation} 
 \label{eqtn:CarlesonMain2} 
 \sup_{Q \in \Delta} \frac{1}{\abs{Q}} \int \int_{\substack{(x,t) \in
     R_{Q} \\ \tilde{\gamma}^{J,B}_{t}(x) \in K_{\nu}}}
 \abs{\tilde{\gamma}^{J,B}_{t}(x)}^{2}  \frac{dx \, dt}{t} \lesssim
 c_{\alpha}^{J} < \infty,
\end{equation}
where $R_{Q} := Q \xx [0,l(Q))$. Recall the John-Nirenberg lemma for
Carleson measures as applied in \cite{axelsson2006quadratic} and \cite{auscher2002solution}.

\begin{lem}[The John-Nirenberg Lemma for Carleson Measures]
  \label{lem:JN}
  Let $\rho$ be a measure on $\R^{n+1}_{+}$ and $\beta > 0$. Suppose
  that for every  $Q \in \Delta$ there exists a collection $\lb Q_{k}
  \rb_{k} \subset \Delta$ of disjoint subcubes of $Q$ such that $E_{Q}
  := Q \setminus \cup_{k} Q_{k}$ satisfies $\abs{E_{Q}} > \beta \abs{Q}$
  and such that
  \begin{equation}
    \label{eqtn:JN1}
\sup_{Q \in \Delta} \frac{\rho(E^{*}_{Q})}{\abs{Q}} \leq C
\end{equation}
for some $C > 0$, where $E^{*}_{Q} := R_{Q} \setminus \cup_{k} R_{Q_{k}}$. Then
\begin{equation}
  \label{eqtn:JN2}
\sup_{Q \in \Delta} \frac{\rho(R_{Q})}{\abs{Q}} \leq \frac{C}{\beta}.
\end{equation}
\end{lem}

\begin{proof}  
Fix $Q \in \Delta$ and let $\lb Q_{k_{1}} \rb_{k_{1}}$ be a collection
of subcubes as in the hypotheses of the lemma. Apply the bound \eqref{eqtn:JN1} to the decomposition
$$
\rho(R_{Q}) = \rho \br{E^{*}_{Q}} + \sum_{k_{1}} \rho \br{R_{Q_{k_{1}}}}
$$
to obtain
$$
\rho(R_{Q}) \leq C \abs{Q} + \sum_{k_{1}} \rho \br{R_{Q_{k_{1}}}}.
$$
For each $k_{1}$, let $\lb Q_{k_{1},k_{2}} \rb_{k_{2}}$ be a collection of subcubes of
$Q_{k_{1}}$ that satisfy the hypotheses of the lemma. Decompose
$\rho(R_{Q_{k_{1}}})$ and once again apply \eqref{eqtn:JN1} to obtain
\begin{align*}\begin{split}  
 \rho(R_{Q}) &\leq C \abs{Q} + \sum_{k_{1}} \br{ \rho(E^{*}_{Q_{k_{1}}}) +
 \sum_{k_{2}} \rho(R_{Q_{k_{1},k_{2}}})} \\
 &\leq C \abs{Q} + \sum_{k_{1}} C \abs{Q_{k_{1}}} + \sum_{k_{1},k_{2}} \rho(R_{Q_{k_{1},k_{2}}}) \\
 &\leq C \abs{Q} + C \abs{Q} \br{1 - \beta} + \sum_{k_{1},k_{2}} \rho(R_{Q_{k_{1},k_{2}}}).
\end{split}\end{align*}
Iterating this process and summing the resulting geometric series
gives \eqref{eqtn:JN2}.
 \end{proof}

 With this tool at our disposal, the proof of our theorem is reduced to the
following proposition.

\begin{prop} 
 \label{prop:Carleson} 
 There exists $\beta > 0$ and $\sigma > 0$ that will satisfy the
 following conditions. For every $\nu \in \mathcal{V}$ and $Q \in
 \Delta$, there is a collection $\lb Q_{k} \rb_{k} \subset \Delta$ of
 disjoint subcubes of $Q$ such that $E_{Q,\nu} = Q \setminus \cup_{k} Q_{k}$
 satisfies $\abs{E_{Q,\nu}} > \beta \abs{Q}$ and such that
 \begin{equation} 
 \label{eqtn:Carleson1} 
 \sup_{Q \in \Delta} \frac{1}{\abs{Q}} \int \int_{\substack{(x,t) \in
     E^{*}_{Q,\nu} \\ \tilde{\gamma}^{J,B}_{t}(x) \in K_{\nu}}}
 \abs{\tilde{\gamma}^{J,B}_{t}(x)}^{2} \frac{dx \,
 dt}{t} \lesssim c_{\alpha}^{J} < \infty,
\end{equation}
where $E^{*}_{Q,\nu} := R_{Q} \setminus \cup_{k} R_{Q_{k}}$. Moreover, $\beta$
and $\sigma$ are entirely independent of the conditions (H8D$\alpha$),
(H8J) and (H8J$\alpha$).
 \end{prop}

 For now, fix $\nu \in \mathcal{V}$ and $Q \in \Delta$. Let $w^{\nu}$,
 $\hat{w}^{\nu} \in \C^{N}$ with $\abs{\hat{w}^{\nu}} =
 \abs{w^{\nu}} = 1$ and $\nu^{*} \br{\hat{w}^{\nu}} = w^{\nu}$. To
 simplify notation, when superfluous, this dependence will be kept
 implicit by defining $w := w^{\nu}$ and $\hat{w} :=
 \hat{w}^{\nu}$. Notice that since $\nu$ satisfies $\nu = \nu
 \mathbb{P}_{3}$, $w$ must satisfy $\mathbb{P}_{3} w = w$.

  For $\epsilon > 0$  the function $f^{w}_{Q,
  \epsilon}$ can be defined in an identical manner to
\cite{axelsson2006quadratic}. Specifically, let $\eta_{Q} : \R^{N}
\rightarrow [0,1]$ be a smooth cutoff function equal to $1$ on $2 Q$,
with support in $4 Q$ and with $\norm{\nabla \eta_{Q}}_{\infty} \leq
\frac{1}{l}$, where $l := l(Q)$. Then define $w_{Q} := \eta_{Q} \cdot w$ and
$$
f^{w}_{Q, \epsilon} := w_{Q} - \epsilon  l i \Gamma_{J} \br{I + \epsilon l
i \Pi_{J,B}}^{-1} w_{Q} = \br{I + \epsilon l i \Gamma^{*}_{J,B}} \br{I +
\epsilon l i \Pi_{J,B}}^{-1} w_{Q}.
$$

\begin{lem} 
 \label{lem:LocalTb1} 
 There exists a constant $C > 0$, independent of (H8D$\alpha$), (H8J)
 and (H8J$\alpha$), that
 satisfies $\norm{f^{w}_{Q, \epsilon}} \leq C \abs{Q}^{\frac{1}{2}}$
 and
 \begin{equation} 
 \label{eqtn:LocalTb1} 
 \abs{\dashint_{Q} \mathbb{P}_{3} f^{w}_{Q, \epsilon} - w} \leq C \cdot \epsilon^{\frac{1}{2}},
\end{equation}
for any $\epsilon > 0$.
\end{lem}

\begin{proof}
  The first claim follows from
  \begin{align*}\begin{split}  
 \norm{f^{w}_{Q,\epsilon}} &\lesssim \norm{w_{Q}} + \norm{\epsilon l i
   \Gamma_{J} \br{I + \epsilon l i \Pi_{J,B}}^{-1} w_{Q}} \\
 &\lesssim \abs{Q}^{\frac{1}{2}} + \norm{\epsilon l i \Pi_{J,B} \br{I
     + \epsilon l i \Pi_{J,B}}^{-1}w_{Q}} \\
 &\lesssim \abs{Q}^{\frac{1}{2}}.
 \end{split}\end{align*}
  On recalling that $w$ is zero in the first two components,
  \begin{align*}\begin{split}  
 \abs{\dashint_{Q} \mathbb{P}_{3} f^{w}_{Q,\epsilon} - w}^{2} &=
 \abs{\dashint_{Q} \mathbb{P}_{3} \epsilon l i \Gamma_{J} \br{I + \epsilon
     l i \Pi_{J,B}}^{-1} w_{Q}}^{2} \\
 &= \abs{\dashint_{Q} \epsilon l i \Gamma_{0} \br{I + \epsilon l i \Pi_{J,B}}^{-1}w_{Q}}^{2}.
 \end{split}\end{align*}
At this point, apply Lemma 5.6 of \cite{axelsson2006quadratic} to the
operator $\Upsilon = \Gamma_{0}$ to obtain
\begin{align*}\begin{split}  
 \abs{\dashint_{Q} \epsilon l i \Gamma_{0} \br{I + \epsilon l i
     \Pi_{J,B}}^{-1}w_{Q}}^{2} &\lesssim \frac{\br{\epsilon l}^{2}}{l}
 \br{\dashint_{Q} \abs{\br{I + \epsilon l i \Pi_{J,B}}^{-1}
     w_{Q}}^{2}}^{\frac{1}{2}} \\ & \qquad  \qquad  \cdot \br{\dashint_{Q} \abs{\Gamma_{0}
     \br{I + \epsilon l i \Pi_{J,B}}^{-1} w_{Q}}^{2}}^{\frac{1}{2}} \\
 &\lesssim \epsilon \br{\dashint_{Q} \abs{\epsilon l i \Gamma_{0} \br{I +
       \epsilon l i \Pi_{J,B}}^{-1} w_{Q}}^{2}}^{\frac{1}{2}} \\
 &\leq \epsilon \br{\dashint_{Q} \abs{\epsilon l i \Gamma_{J} \br{I +
       \epsilon l i \Pi_{J,B}}^{-1} w_{Q}}^{2}}^{\frac{1}{2}} \\
 &\lesssim \epsilon,
\end{split}\end{align*}
where the inequality $\norm{\Gamma_{0} v} \leq \norm{\Gamma_{J} v}$
for $v \in D \br{\Gamma_{J}}$ follows trivially from the matrix form
of $\Gamma_{0}$ and $\Gamma_{J}$.
 \end{proof}

 \begin{lem} 
 \label{lem:LocalTb2} 
There exists a constant $D > 0$, independent of (H8D$\alpha$), (H8J)
and (H8J$\alpha$), such
that
\begin{equation} 
  \label{eqtn:LocalTb2}
  \int \int_{R_{Q}} \abs{\Theta^{J,B}_{t} f^{w}_{Q,\epsilon}(x)}^{2}
  \frac{dx \, dt}{t} \leq D \frac{\abs{Q}}{\epsilon^{2}}.
 \end{equation}
\end{lem}

\begin{proof}  
  First observe that
  \begin{align*}\begin{split}  
\Theta^{J,B}_{t} f^{w}_{Q,\epsilon} &=  P_{t}^{J,B} t \Gamma^{*}_{J,B}
\br{I + \epsilon l i \Gamma^{*}_{J,B}} \br{I + \epsilon l i
  \Pi_{J,B}}^{-1} w_{Q} \\
&= \frac{t}{\epsilon l} P_{t}^{J,B} \epsilon l \Gamma^{*}_{J,B} \br{I
  + \epsilon l i \Pi_{J,B}}^{-1} w_{Q}.
\end{split}\end{align*}
Therefore
\begin{align*}\begin{split}  
 \int^{l}_{0} \int_{Q} \abs{\Theta^{J,B}_{t}
   f^{w}_{Q,\epsilon}(x)}^{2} \frac{dx \, dt}{t} &= \int^{l}_{0}
 \br{\frac{t}{\epsilon l}}^{2} \int_{Q} \abs{P_{t}^{J,B} \epsilon l
   \Gamma^{*}_{J,B} \br{I + \epsilon l i \Pi_{J,B}}^{-1} w_{Q}}^{2}
 dx \, \frac{dt}{t} \\
 &\lesssim \int^{l}_{0} \br{\frac{t}{\epsilon l}}^{2} \norm{\epsilon l
 i \Gamma^{*}_{J,B} \br{I + \epsilon l i \Pi_{J,B}}^{-1} w_{Q}}^{2}
\frac{dt}{t} \\
&\lesssim \frac{\abs{Q}}{\br{\epsilon l}^{2}} \int^{l}_{0} t \, dt \\
&\simeq \frac{\abs{Q}}{\epsilon^{2}}.
 \end{split}\end{align*}
 \end{proof}

 From this point forward, with $C$ as in Lemma \ref{lem:LocalTb1}, set
 $\epsilon := \frac{1}{4 C^{2}}$ and introduce the notation $f^{w}_{Q} := f^{w}_{Q,
   \epsilon}$. With this choice of $\epsilon$ it must be true that
 $$
\abs{\dashint_{Q} \mathbb{P}_{3}f^{w}_{Q} - w} \leq \frac{1}{2}.
$$
That is,
\begin{align*}\begin{split}  
 1 - 2 \mathrm{Re} \left\langle \dashint_{Q} \mathbb{P}_{3} f^{w}_{Q}, w \right\rangle &=
 \abs{w}^{2} - 2 \mathrm{Re} \left\langle \dashint_{Q} \mathbb{P}_{3} f^{w}_{Q}, w \right\rangle
 \\
 &\leq \abs{\dashint_{Q} \mathbb{P}_{3} f^{w}_{Q} - w}^{2} \\
 &\leq \frac{1}{4}.
\end{split}\end{align*}
On rearranging we find that
\begin{equation}
  \label{eqtn:RealPart}
  \mathrm{Re} \left\langle \dashint_{Q} \mathbb{P}_{3}f^{w}_{Q}, w \right\rangle \geq \frac{1}{4}.
\end{equation}

 In this context, Lemma 5.11 of \cite{axelsson2006quadratic} will take on the below form.

 \begin{lem} 
   \label{lem:LocalTb3}
   There exists $\beta$, $c_{1}$, $c_{2} > 0$ and a collection $\lb Q_{k} \rb$ of dyadic cubes of $Q$
   such that $\abs{E_{Q,\nu}} > \beta \abs{Q}$ and such that
   $$
\mathrm{Re} \left\langle w, \dashint_{Q'} \mathbb{P}_{3} f^{w}_{Q} \right\rangle
\geq c_{1} \quad and \quad \dashint_{Q'} \abs{\mathbb{P}_{3}f^{w}_{Q}}
\leq c_{2}
$$
for all dyadic subcubes $Q' \in \Delta$ of $Q$ which satisfy $R_{Q'}
\cap E^{*}_{Q, \nu} \neq \emptyset$. Moreover, $\beta$, $c_{1}$ and
$c_{2}$ are independent of (H8D$\alpha$), (H8J), (H8J$\alpha$), $Q$, $\sigma$ and $\nu$.
\end{lem}

The proof of this statement follows in an identical manner to the
argument in \cite{axelsson2006quadratic}. If we set $\sigma = \frac{c_{1}}{2 c_{2}}$, then the following pointwise estimate can be
deduced.

\begin{lem} 
 \label{lem:LocalTb4} 
 If $(x,t) \in E^{*}_{Q, \nu}$ and $\tilde{\gamma}_{t}^{J,B}(x) \in
 K_{\nu}$ then
 \begin{equation}
   \label{eqtn:LocalTb4}
   \abs{\tilde{\gamma}_{t}^{J,B}(x) \br{A_{t} f^{w}_{Q}(x)}} \geq \frac{1}{2}
   c_{1} \abs{\tilde{\gamma}^{J,B}_{t}(x)}.
   \end{equation}
 \end{lem}

 \begin{proof}
   First observe that
\begin{align*}\begin{split}  
 \abs{\nu \br{A_{t} f^{w}_{Q}(x)}} &\geq \mathrm{Re} \left\langle \hat{w},
 \nu \br{A_{t} f^{w}_{Q}(x)} \right\rangle \\
&= \mathrm{Re} \left\langle w , A_{t} f^{w}_{Q}(x) \right\rangle \\
 &= \mathrm{Re} \left\langle w, A_{t} \mathbb{P}_{3} f^{w}_{Q}(x) \right\rangle \\
 &\geq c_{1}.
\end{split}\end{align*}
Then
\begin{align*}\begin{split}  
 \abs{\frac{\tilde{\gamma}^{J,B}_{t}(x)}{\abs{\tilde{\gamma}^{J,B}_{t}(x)}}
   \br{A_{t} f^{w}_{Q}(x)}} &=  \abs{\frac{\tilde{\gamma}^{J,B}_{t}(x)}{\abs{\tilde{\gamma}^{J,B}_{t}(x)}}
   \br{A_{t} \mathbb{P}_{3} f^{w}_{Q}(x)}} \\
 &\geq \abs{\nu \br{A_{t} f^{w}_{Q}(x)}} -
 \abs{\frac{\tilde{\gamma}^{J,B}_{t}(x)}{\abs{\tilde{\gamma}^{J,B}_{t}(x)}}
   - \nu} \abs{A_{t} \mathbb{P}_{3} f^{w}_{Q}(x)} \\
 &\geq c_{1} - \sigma c_{2} \\
 &= \frac{1}{2} c_{1}.
 \end{split}\end{align*}
\end{proof}

\textsc{Proof of Proposition \ref{prop:Carleson}.} From the pointwise bound of the previous lemma,
$$
 \int \int_{\substack{(x,t) \in E^{*}_{Q,\nu} \\
     \tilde{\gamma}_{t}^{J,B}(x) \in K_{\nu}}}
 \abs{\tilde{\gamma}_{t}^{J,B}(x)}^{2} \frac{dx \, dt}{t} \lesssim \int \int_{R_{Q}}
 \abs{\tilde{\gamma}^{J,B}_{t}(x) A_{t} f^{w}_{Q}(x)}^{2}
\frac{dx \,  dt}{t}.
$$
At this stage we can begin to unravel our square function norm,
\begin{align}\begin{split}
    \label{eqtn:Final0}
 \int \int_{R_{Q}} \abs{\tilde{\gamma}^{J,B}_{t}(x) A_{t}
   f^{w}_{Q}(x)}^{2} \frac{dx \, dt}{t} &\lesssim \int \int_{R_{Q}} \abs{\Theta^{J,B}_{t}
   f^{w}_{Q}(x) - \tilde{\gamma}^{J,B}_{t}(x)A_{t} f^{w}_{Q}(x)}^{2}
 \frac{dx \, dt}{t} \\
 & \qquad \qquad + \int \int_{R_{Q}} \abs{\Theta^{J,B}_{t} f^{w}_{Q}(x)}^{2}
 \frac{dx \, dt}{t}.
\end{split}\end{align}
 Lemma \ref{lem:LocalTb2} states that the final term in the above
estimate will be bounded from above by a multiple of $\abs{Q}$. This reduces the task of
proving the proposition to bounding the first term of the above
splitting.
Recall that $f^{w}_{Q}$ can be expressed in the form
$$
f^{w}_{Q} := w_{Q} - u^{w}_{Q},
$$
where $u^{w}_{Q} \in R \br{\Gamma_{J}}$ is given by
$$
u^{w}_{Q} := \epsilon l i \Gamma_{J} \br{I + \epsilon l i
  \Pi_{J,B}}^{-1} w_{Q}.
$$
An application of the triangle inequality then leads to
\begin{align}\begin{split}
    \label{eqtn:Final01}
 \int \int_{R_{Q}} &\abs{ \Theta^{J,B}_{t} f^{w}_{Q}(x) -
  \tilde{\gamma}^{J,B}_{t}(x) A_{t} f^{w}_{Q}(x)}^{2} \frac{dx \,
  dt}{t} \\ &\qquad \qquad \lesssim \int \int_{R_{Q}} \abs{\Theta^{J,B}_{t} w_{Q}(x) -
  \tilde{\gamma}^{J,B}_{t}(x) A_{t} w_{Q}(x)}^{2}  \frac{dx \,
  dt}{t} \\ & \qquad \qquad \qquad + \int \int_{R_{Q}} \abs{\Theta^{J,B}_{t} u^{w}_{Q}(x) -
  \tilde{\gamma}^{J,B}_{t}(x) A_{t} u^{w}_{Q}(x)}^{2}  \frac{dx \, dt}{t}.
\end{split}\end{align}
On noticing that for every $x \in Q$ and $0 < t < l(Q)$
\begin{align*}\begin{split}  
 \Theta^{J,B}_{t} w_{Q}(x) - \tilde{\gamma}^{J,B}_{t}(x) A_{t}
 w_{Q}(x) &= \Theta^{J,B}_{t} w_{Q}(x) - \Theta^{J,B}_{t} \br{A_{t}
   w_{Q}(x)}(x) \\
 &= \Theta^{J,B}_{t} \br{\br{\eta_{Q} - 1} w}(x),
 \end{split}\end{align*}
it is clear that the first term in \eqref{eqtn:Final01} can be handled
in  an identical manner as in the proof
of Proposition 5.9 from \cite{axelsson2006quadratic}. Specifically, 
since  $\br{\mathrm{supp} \br{\eta_{Q} - 1} w} \cap 2 Q = \emptyset$,
the off-diagonal estimates of the operator $\Theta^{J,B}_{t}$ lead to
$$
\int_{Q} \abs{\Theta^{J,B}_{t} \br{\br{\eta_{Q} - 1} w}(x)}^{2} dx
\lesssim \frac{t \abs{Q}}{l},
$$
which implies that
$$
\int \int_{R_{Q}} \abs{\Theta^{J,B}_{t} w_{Q}(x) -
  \tilde{\gamma}^{J,B}_{t}(x) A_{t} w_{Q}(x)}^{2} \frac{dx \, dt}{t} \lesssim \abs{Q}.
$$
As for the second term in \eqref{eqtn:Final01},
\begin{align}\begin{split}
    \label{eqtn:Final1}
 \int \int_{R_{Q}} &\abs{\Theta^{J,B}_{t} u^{w}_{Q} -
   \tilde{\gamma}^{J,B}_{t}(x) A_{t} u^{w}_{Q}(x)}^{2} \frac{dx \,
 dt}{t} \\ &\qquad \qquad\lesssim \int \int_{R_{Q}} \abs{\Theta^{J,B}_{t} \br{I -
   P_{t}^{J}}u^{w}_{Q}(x)}^{2} \frac{dx \, dt}{t} \\
& \qquad \qquad \qquad + \int \int_{R_{Q}} \abs{\Theta^{J,B}_{t} P_{t}^{J} u^{w}_{Q}(x) -
  \tilde{\gamma}^{J,B}_{t}(x) A_{t} u^{w}_{Q}(x)}^{2} \frac{dx
  \, dt}{t}.
\end{split}\end{align}
Since $u_{Q}^{w} \in R \br{\Gamma_{J}}$, Corollary \ref{cor:HighFrequency} gives
\begin{align*}\begin{split}  
 \int \int_{R_{Q}} \abs{\Theta^{J,B}_{t} \br{I - P_{t}^{J}}
   u^{w}_{Q}}^{2} \frac{dx \, dt}{t} &\lesssim
 \norm{u^{w}_{Q}}^{2} \\
 &\lesssim \abs{Q}.
\end{split}\end{align*}
 For the remaining term in \eqref{eqtn:Final1},
\begin{align}\begin{split}
    \label{eqtn:Final2}
 \int \int_{R_{Q}} &\abs{\Theta^{J,B}_{t} P_{t}^{J} u^{w}_{Q}(x) -
   \tilde{\gamma}^{J,B}_{t}(x) A_{t} u^{w}_{Q}(x)}^{2} \frac{dx \,
 dt}{t} \\ & \qquad \qquad \lesssim \int \int_{R_{Q}} \abs{\Theta^{J,B}_{t} \br{I -
   \mathbb{P}_{3}} P_{t}^{J} u^{w}_{Q}}^{2}\frac{dx \, dt}{t} \\
& \qquad \qquad \qquad + \int \int_{R_{Q}} \abs{\tilde{\Theta}^{J,B}_{t} P_{t}^{J}
  u^{w}_{Q}(x) - \tilde{\gamma}^{J,B}_{t}(x) A_{t} u^{w}_{Q}(x)}^{2} 
\frac{dx \, dt}{t}.
\end{split}\end{align}
Since we have already proved the boundedness of the first and second components,
\begin{align*}\begin{split}  
 \int \int_{R_{Q}} \abs{\Theta^{J,B}_{t} \br{I - \mathbb{P}_{3}}
   P_{t}^{J} u^{w}_{Q}}^{2} \frac{dx \, dt}{t} 
 &\lesssim c_{\alpha}^{J} \cdot \norm{u^{w}_{Q}}^{2} \\
 &\lesssim c_{\alpha}^{J} \cdot \abs{Q}.
\end{split}\end{align*}
For the second term in \eqref{eqtn:Final2},
\begin{align}\begin{split}
    \label{eqtn:Final3}
 \int \int_{R_{Q}} &\abs{\tilde{\Theta}^{J,B}_{t}  P_{t}^{J}
   u^{w}_{Q}(x) - \tilde{\gamma}^{J,B}_{t}(x) A_{t} u^{w}_{Q}(x)}^{2}
 \frac{dx \, dt}{t}
 \\ & \qquad \qquad\lesssim \int \int_{R_{Q}} \abs{\tilde{\Theta}^{J,B}_{t} P_{t}^{J} u^{w}_{Q}(x) -
 \tilde{\gamma}^{J,B}_{t}(x) A_{t} P_{t}^{J} u^{w}_{Q}(x)}^{2}
\frac{dx \, dt}{t} \\
& \qquad \qquad \qquad + \int \int_{R_{Q}} \abs{\gamma^{J,B}_{t}(x) \mathbb{P}_{3} \br{ A_{t}
  P_{t}^{J} -  A_{t}} u^{w}_{Q}(x)}^{2} \frac{dx \, dt}{t}.
\end{split}\end{align}
To bound the first term on the right-hand side of the above estimate
notice that
$$
\tilde{\Theta}^{J,B}_{t} P_{t}^{J} u^{w}_{Q}(x) -
 \tilde{\gamma}^{J,B}_{t}(x) A_{t} P_{t}^{J} u^{w}_{Q}(x) =
 \br{\Theta^{J,B}_{t} - \gamma^{J,B}_{t} A_{t}} \mathbb{P}_{3}
 P_{t}^{J} u^{w}_{Q}(x).
 $$
 Theorem \ref{thm:Diagonalisation} then allows us to diagonalise our
 $P_{t}^{J}$ operators in the first term of
 \eqref{eqtn:Final3} to get
 $$
\int^{l(Q)}_{0} \norm{\br{\Theta^{J,B}_{t} - \gamma^{J,B}_{t}
    A_{t}} \mathbb{P}_{3} P_{t}^{J} u^{w}_{Q}}^{2}_{L^{2}\br{Q}} \frac{dt}{t} \lesssim
c_{\alpha}^{J} \abs{Q} + \int^{\infty}_{0} \norm{\br{\Theta^{J,B}_{t}
    - \gamma^{J,B}_{t} A_{t}} \mathbb{P}_{3} \mathcal{P}_{t}^{J}
  u^{w}_{Q}}^{2} \frac{dt}{t}.
 $$
From Proposition
\ref{prop:PrincPart} we know that
\begin{align*}\begin{split}  
 \int^{\infty}_{0} \norm{\br{\Theta^{J,B}_{t} - \gamma_{t}^{J,B}
     A_{t}} \mathbb{P}_{3} \mathcal{P}_{t}^{J} u^{w}_{Q}}^{2}
 \frac{dt}{t} &\lesssim \int^{\infty}_{0} \norm{t \nabla
   \mathbb{P}_{3} \mathcal{P}_{t}^{J} u^{w}_{Q}}^{2} \frac{dt}{t} \\
 &= \int^{\infty}_{0}\norm{t \nabla \mathbb{P}_{3} P_{t}^{0}
   u^{w}_{Q}}^{2} \frac{dt}{t} \\
 &\lesssim \int^{\infty}_{0} \norm{t \Pi_{0} P_{t}^{0} u^{w}_{Q}}^{2}
 \frac{dt}{t} \\
 &= \int^{\infty}_{0} \norm{Q_{t}^{0}u^{w}_{Q}}^{2} \frac{dt}{t} \\
 &\lesssim \abs{Q},
\end{split}\end{align*}
where in the third line we applied (H8) for the operators $\lb
\Gamma_{0},B_{1}, B_{2} \rb$. It should be noted that in order to use
(H8), we had to use the fact that
\begin{align*}\begin{split}  
 \mathbb{P}_{3} P_{t}^{0} u^{w}_{Q} &= P_{t}^{0} \mathbb{P}_{3}
 u^{w}_{Q} = P_{t}^{0} \mathbb{P}_{3} \epsilon l i \Gamma_{J} (I +
 \epsilon l i \Pi_{J,B})^{-1} w_{Q} \\
 &= P_{t}^{0} \mathbb{P}_{3} \epsilon l i \Gamma_{0} (I + \epsilon l i
 \Pi_{J,B})^{-1} w_{Q} = \epsilon l i \Gamma_{0} P_{t}^{0} (I +
 \epsilon l i \Pi_{J,B})^{-1} w_{Q} \in R (\Gamma_{0}).
 \end{split}\end{align*}
It remains to bound the second term in \eqref{eqtn:Final3},
$$
\int \int_{R_{Q}} \abs{\gamma^{J,B}_{t}(x) \mathbb{P}_{3} A_{t}
  \br{P_{t}^{J} - I} u^{w}_{Q}(x)}^{2} \frac{dx \, dt}{t} = \int
\int_{R_{Q}} \abs{\gamma^{J,B}_{t} A_{t} \mathbb{P}_{3} A_{t}
  \br{P_{t}^{J} - I} u^{w}_{Q}(x)}^{2} \frac{dx \, dt}{t}
$$
On noting the uniform $L^{2}$-boundedness of the $\gamma^{J,B}_{t}
A_{t}$ operators and applying the triangle inequality,
\begin{align*}\begin{split}  
 \int
\int_{R_{Q}} \abs{\gamma^{J,B}_{t} A_{t} \mathbb{P}_{3} A_{t}
  \br{P_{t}^{J} - I} u^{w}_{Q}(x)}^{2} \frac{dx \, dt}{t} &\lesssim
\int_{0}^{\infty} \int_{\R^{n}} \abs{\mathbb{P}_{3} A_{t} \br{P_{t}^{J} - I}
  u^{w}_{Q}(x)}^{2} \frac{dx \, dt}{t} \\
&\lesssim \int_{0}^{\infty} \int_{\R^{n}} \abs{\mathbb{P}_{3} A_{t} \br{P_{t}^{J} -
    \mathcal{P}_{t}^{J}} u^{w}_{Q}(x)}^{2} \frac{dx \, dt}{t} \\
&\qquad + \int_{0}^{\infty} \int_{\R^{n}} \abs{\mathbb{P}_{3} A_{t} \br{\mathcal{P}_{t}^{J} - I}
  u^{w}_{Q}(x)}^{2} \frac{dx \, dt}{t}.
\end{split}\end{align*}
Applying Theorem \ref{thm:Diagonalisation} and recalling that $\mathbb{P}_{3}
\mathcal{P}_{t}^{J} = \mathbb{P}_{3} P_{t}^{0}$,
\begin{align*}\begin{split}  
\int \int_{R_{Q}} \abs{\gamma^{J,B}_{t} A_{t} \mathbb{P}_{3} A_{t}
  \br{P_{t}^{J} - I} u^{w}_{Q}(x)}^{2} \frac{dx \, dt}{t} &\lesssim
c_{\alpha}^{J} \cdot \norm{u}^{2} +  \int_{0}^{\infty} \int_{\R^{n}} \abs{\mathbb{P}_{3} A_{t} \br{P_{t}^{0} - I}
  u^{w}_{Q}(x)}^{2} \frac{dx \, dt}{t}.
\end{split}\end{align*}
From the proof of Proposition 5.7 of \cite{axelsson2006quadratic} we
know that
$$
\int_{0}^{\infty} \int_{\R^{n}} \abs{A_{t} \br{P_{t}^{0} - I}u^{w}_{Q}(x)}^{2}
\frac{dx \, dt}{t} \lesssim \abs{Q},
$$
allowing us to finally conclude our proof.
\hfill \BoldSquare \vspace{5pt}

\section{Applications}
\label{sec:Applications}

Our non-homogeneous framework will now be applied to three
different contexts. We begin with the case that serves as the primary
motivation for this article,
the scalar Kato square root problem with zero-order
potential.

\subsection{Scalar Kato with Zero-Order Potential}
 \label{subsec:Scalar}

Theorem \ref{thm:KatoPotential},
the promised result of the introductory section, will now be
proved. Fix $V \in \mathcal{W}_{\alpha}$ for some $\alpha \in (1,2]$.
 Brand the definition of the operators $\Gamma_{J}$, $B_{1}$ and $B_{2}$ to be as
follows. Define our Hilbert space to be
$$
\mathcal{H} := L^{2}\br{\R^{n}} \oplus L^{2}\br{\R^{n}}
\oplus L^{2}(\R^{n};\C^{n}),
$$
for some $n \in \N^{*}$. Set $J = \abs{V}^{\frac{1}{2}}$ and $D = \nabla$. 
 Our operator $\Gamma_{J}$ is
then given by
$$
\Gamma_{J} = \Gamma_{\abs{V}^{\frac{1}{2}}} = \br{\begin{array}{c c} 
 0 & 0 \\ \nabla_{\abs{V}^{\frac{1}{2}}} & 0
 \end{array}} =  \br{\begin{array}{c c c} 
            0 & 0 & 0 \\
            \abs{V}^{\frac{1}{2}} & 0 & 0 \\
            \nabla & 0 & 0
 \end{array}},
 $$
 defined on the dense domain $H^{1,V}\br{\R^{n}} \oplus L^{2} \br{\R^{n}} \oplus
 L^{2}\br{\R^{n};\C^{n}}$, where $H^{1,V} \br{\R^{n}}$ is as defined
 in the introductory section.
The density of $H^{1,V}\br{\R^{n}}$ in $L^{2} \br{\R^{n}}$ follows
from $V \in L^{1}_{loc}(\R^{n})$.

Let $A \in L^{\infty} \br{\R^{n}; \mathcal{L}\br{\C^{n}}}$ be a
matrix-valued multiplication operator and suppose that the G{\aa}rding
inequalities \eqref{eqtn:Garding0} and \eqref{eqtn:Garding} are 
satisfied with constants $\kappa^{A} > 0$ and $\kappa_{V}^{A} > 0$ respectively.
Define our perturbations $B_{1}$ and $B_{2}$ through
 $$
B_{1} = I \qquad and
\qquad B_{2} := \br{\begin{array}{c c} 
  I & 0 \\ 0 & \hat{A}
 \end{array}} :=  \br{\begin{array}{c c c} 
                                              I & 0 & 0 \\
                                              0 & e^{i \cdot \mathrm{arg} V} & 0 \\
                                              0 & 0 & A
 \end{array}}.
$$
Our perturbed Dirac-type operator then becomes
 $$
\Pi_{B,\abs{V}^{\frac{1}{2}}} := \Gamma_{\abs{V}^{\frac{1}{2}}} +
\Gamma^{*}_{\abs{V}^{\frac{1}{2}}}  \br{\begin{array}{c c} 
 I & 0 \\ 0 & \hat{A}
 \end{array}} = \br{\begin{array}{c c} 
                      0 & \nabla_{\abs{V}^{\frac{1}{2}}}^{*} \hat{A} \\
                      \nabla_{\abs{V}^{\frac{1}{2}}} & 0
 \end{array}}.
$$
The square of our perturbed Dirac-type operator is then given by
$$
\Pi_{B,\abs{V}^{\frac{1}{2}}}^{2} = \br{\begin{array}{c c} 
 \nabla^{*}_{\abs{V}^{\frac{1}{2}}} \hat{A}
                                          \nabla_{\abs{V}^{\frac{1}{2}}}
                                          & 0 \\
                                          0 &
                                              \nabla_{\abs{V}^{\frac{1}{2}}}
                                              \nabla^{*}_{\abs{V}^{\frac{1}{2}}} \hat{A}
 \end{array}} = \br{\begin{array}{c c} 
                      - \mathrm{div} (A \nabla) + V & 0 \\
                      0 & \nabla_{\abs{V}^{\frac{1}{2}}}
                          \nabla^{*}_{\abs{V}^{\frac{1}{2}}} \hat{A}
 \end{array}}.
$$
It is clear from the form of our operator $\Gamma_{0}$ and the fact
that $A$ satisfies \eqref{eqtn:Garding0} that the operators $\lb
\Gamma_{0}, B_{1}, B_{2} \rb$ satisfy (H1) - (H8). Similarly, since
$A$ and $V$ satisfy \eqref{eqtn:Garding}, it follows that $\lb
\Gamma_{J}, B_{1}, B_{2} \rb$ satisfy the properties (H1) - (H6).

\begin{lem}
  \label{lem:WAlphaImplies}
  For $V \in \mathcal{W}_{\alpha}$, the conditions (H8D$\alpha$) and
(H8J$\alpha$) are both satisfied and
 $c_{\alpha}^{J} \lesssim \br{1 + \brs{V}_{\alpha}^{2}}\br{\alpha - 1}^{-1}$.
\end{lem}

\begin{proof}  
  The condition $V \in \mathcal{W}_{\alpha}$ tells us that
  \begin{equation}
    \label{eqtn:WAlphaImplies1}
    \norm{\abs{V}^{\frac{\alpha}{2}}u}_{2} +
    \norm{(-\Delta)^{\frac{\alpha}{2}}u}_{2} \leq \brs{V}_{\alpha}
    \norm{\br{\abs{V} - \Delta}^{\frac{\alpha}{2}}u}_{2}
  \end{equation}
  for all $u \in D(\abs{V} - \Delta)$. Let's first prove that
  (H8D$\alpha$) is satisfied. In order to prove this
  condition, it is sufficient to show that $D \br{(\abs{V} -
    \Delta)^{\frac{\alpha}{2}}} \subset D
  \br{(-\Delta)^{\frac{\alpha}{2}}}$ and
  \begin{equation}
    \label{eqtn:WAlphaImplies2}
    \norm{(-\Delta)^{\frac{\alpha}{2}}u}_{2} \leq \brs{V}_{\alpha}
    \norm{(\abs{V} - \Delta)^{\frac{\alpha}{2}}u}_{2}
  \end{equation}
  for all $u \in D \br{(\abs{V} - \Delta)^{\frac{\alpha}{2}}}$. Fix $u
  \in D \br{(\abs{V} - \Delta)^{\frac{\alpha}{2}}}$. Since
  $D(\abs{V} - \Delta)$ is a core for $(\abs{V} -
  \Delta)^{\frac{\alpha}{2}}$, there must exist some $\lb u_{n} \rb_{n
    = 1}^{\infty} \subset D(\abs{V} - \Delta)$ with
  \begin{equation}
    \label{eqtn:WAlphaImplies3}
    \norm{u_{n} - u}_{2} + \norm{(\abs{V} - \Delta)^{\frac{\alpha}{2}}(u_{n} - u)}_{2}
    \xrightarrow{n \rightarrow \infty} 0.
  \end{equation}
  We then have for $n, \, m \in \N$,
  $$
\norm{u_{n} - u_{m}}_{2} + \norm{(-\Delta)^{\frac{\alpha}{2}}(u_{n} - u_{m})}_{2} \leq
\norm{u_{n} - u_{m}}_{2} + \brs{V}_{\alpha} \norm{(\abs{V} -
  \Delta)^{\frac{\alpha}{2}} (u_{n} - u_{m})}_{2}.
$$
This proves that $\lb u_{n} \rb_{n = 1}^{\infty}$ is Cauchy in the
graph norm of $(-\Delta)^{\frac{\alpha}{2}}$. The sequence $\lb u_{n}
\rb_{n = 1}^{\infty}$ must therefore 
converge to some $\tilde{u} \in D \br{(-\Delta)^{\frac{\alpha}{2}}}$
in the graph norm of $(-\Delta)^{\frac{\alpha}{2}}$,
$$
\norm{u_{n} - \tilde{u}}_{2} + \norm{(-\Delta)^{\frac{\alpha}{2}} (u_{n} -
  \tilde{u})}_{2} \xrightarrow{n \rightarrow \infty} 0.
$$
This combined with \eqref{eqtn:WAlphaImplies3} shows that $u =
\tilde{u}$ and therefore $u \in
D\br{(-\Delta)^{\frac{\alpha}{2}}}$. Moreover, we have that
\begin{align*}\begin{split}  
 \norm{(-\Delta)^{\frac{\alpha}{2}}u}_{2} &= \lim_{n \rightarrow \infty}
 \norm{(-\Delta)^{\frac{\alpha}{2}}u_{n}}_{2} \\
 &\leq \brs{V}_{\alpha} \lim_{n \rightarrow \infty} \norm{(\abs{V} -
   \Delta)^{\frac{\alpha}{2}}u_{n}}_{2} \\
 &= \brs{V}_{\alpha} \norm{(\abs{V} - \Delta)^{\frac{\alpha}{2}}u}_{2},
\end{split}\end{align*}
completing the proof of (H8D$\alpha$) with
$b^{D}_{\alpha} \lesssim \brs{V}_{\alpha}$. An identical proof can be
used to obtain the condition (H8J$\alpha$) with
constant $b^{J}_{\alpha} \lesssim \brs{V}_{\alpha}$.
 \end{proof}

Combining the above Lemma with Corollary \ref{cor:LBJCalc} completes the proof of Theorem \ref{thm:KatoPotential}.

\subsection{The Class $\mathcal{W}$}
\label{subsec:W}

Define the potential class
$$\label{not:W}
\mathcal{W} := \bigcup_{\alpha \in (1,2]} \mathcal{W}_{\alpha}.
$$
It has so far been proved that the Kato estimate holds for any potential
in the class $\mathcal{W}$.
At this stage, however,
the unperturbed condition $V \in \mathcal{W}$ remains in quite an
abstract form. It will therefore be instructive to unpack this
condition and compare $\mathcal{W}$ with other commonly used classes
of potentials. It is first interesting to note that $\mathcal{W}_{1}$
is the collection of all potentials with no additional restrictions.

\begin{prop} 
 \label{prop:W1} 
 For any locally integrable  $V: \R^{n} \rightarrow \C$ we have $\brs{V}_{1} \leq
 2$. That is, $\mathcal{W}_{1} = L^{1}_{loc}(\R^{n})$.
\end{prop}

\begin{proof}  
  We have
  \begin{align*}\begin{split}  
 \norm{(-\Delta)^{\frac{1}{2}}u}_{2}^{2} &= \langle
 (-\Delta)^{\frac{1}{2}}u, (-\Delta)^{\frac{1}{2}} u \rangle_{2} = \langle
 (-\Delta)u, u \rangle_{2} \\
 &\leq \langle - \Delta u, u \rangle_{2} + \langle \abs{V} u, u \rangle_{2} \\
 &= \langle (\abs{V} - \Delta)^{\frac{1}{2}}u, (\abs{V} -
 \Delta)^{\frac{1}{2}}u \rangle_{2} = \norm{(\abs{V} - \Delta)^{\frac{1}{2}}u}_{2}^{2}
 \end{split}\end{align*}
for all $u \in D(\abs{V} - \Delta)$. The second estimate follows
in a similar manner from the non-negativity of $(-\Delta)$.
 \end{proof}

Using this as an endpoint, it can then be proved using an interpolation style argument that the
potential classes $\lb \mathcal{W}_{\alpha} \rb_{\alpha \in [1,2]}$
form a decreasing collection.
 
\begin{prop} 
 \label{lem:Independence} 
 Suppose that the potential $V : \R^{n} \rightarrow \C$ is in
 $\mathcal{W}_{\alpha}$ for some $\alpha \in (1,2]$. Then $V \in \mathcal{W}_{\beta}$ for any $\beta \in [1,\alpha]$ with
 \begin{equation}
   \label{eqtn:Independence}
\brs{V}_{\beta} \leq  2 \brs{V}_{\alpha}^{\frac{\beta - 1}{\alpha - 1}}.
\end{equation}
Therefore for $\beta \leq \alpha$
$$
\mathcal{W}_{\beta} \supset \mathcal{W}_{\alpha}.
$$
\end{prop}

\begin{proof}
 Assume that $V \in \mathcal{W}_{\alpha}$ for some $\alpha \in (1,2]$. It
will be shown through the Hadamard three-lines theorem that $V \in \mathcal{W}_{\beta}$ for
$\beta \in [1,\alpha]$ with the constant given in
\eqref{eqtn:Independence}.

Let $St := \lb z \in \C : 0 < \mathrm{Re} \, z < 1 \rb$ and set
$\theta := \frac{\beta - 1}{\alpha - 1}$. Fix $u \in D(\abs{V} - \Delta)$ and define $f : \overline{St} \rightarrow \R$ to be the function given by
$$
f(z) := \norm{\abs{V}^{\frac{1}{2} + \br{\frac{\alpha - 1}{2}} z}
\br{\abs{V} - \Delta}^{\br{\frac{\alpha - 1}{2}} \br{\theta - z}}u}_{2}.
$$
$f$ is holomorphic on $St$ and continuous on the closed strip
$\overline{St}$. In order to apply the three-lines theorem, it must
first be proved that $f$ is bounded on
$\overline{St}$. For $z = s + it$ where $0 \leq s \leq 1$ and $t \in
\R$ we have
\begin{align}\begin{split}
    \label{eqtn:Hadamard1}
    f(z)  &= \norm{\abs{V}^{\frac{1}{2} + \br{\frac{\alpha - 1}{2}}
        \br{s+ i t}}
   \br{\abs{V} - \Delta}^{\frac{\beta - 1}{2} - \br{\frac{\alpha -
         1}{2}}(s + i t)}u}_{2} \\
 &= \norm{\abs{V}^{\frac{1}{2} + \br{\frac{\alpha - 1}{2}}s}
   \br{\abs{V} - \Delta}^{-\frac{1}{2} - \br{\frac{\alpha - 1}{2}}s}
   \br{\abs{V} - \Delta}^{-\br{\frac{\alpha - 1}{2}} i t} \br{\abs{V} -
   \Delta}^{\frac{\beta}{2}}u}_{2}.
\end{split}\end{align}
The function
$$
s \mapsto \norm{\abs{V}^{\frac{1}{2} + \br{\frac{\alpha - 1}{2}}s}
  \br{\abs{V} - \Delta}^{-\frac{1}{2} - \br{\frac{\alpha - 1}{2}}s}v}_{2}
$$
is continuous on $[0,1]$ for $v \in L^{2}\br{\R^{n}}$. This, together with \eqref{eqtn:Hadamard1}
then gives
\begin{align*}\begin{split}  
 f(z) &\lesssim \norm{\br{\abs{V} - \Delta}^{-\br{\frac{\alpha -
         1}{2}}i t} \br{\abs{V} - \Delta}^{\frac{\beta}{2}}u}_{2} \\
 &\leq \norm{\br{\abs{V} - \Delta}^{\frac{\beta}{2}}u}_{2},
\end{split}\end{align*}
where we used the fact that $\br{\abs{V} - \Delta}^{i a}$ is a
contraction operator on $L^{2}\br{\R^{n}}$ for $a \in \R$. This
demonstrates that $f$ is bounded on $\overline{St}$.

For $t \in \R$, Proposition \ref{prop:W1} implies that
\begin{align*}\begin{split}  
 f(it) &= \norm{\abs{V}^{\frac{1}{2} + \br{\frac{\alpha - 1}{2}} i t}
   \br{\abs{V} - \Delta}^{\br{\frac{\alpha - 1}{2}} \br{\theta - i
       t}}u}_{2} \\
 &= \norm{\abs{V}^{\frac{1}{2}} \br{\abs{V} -
     \Delta}^{\br{\frac{\alpha - 1}{2}} \br{\theta - i t}} u}_{2} \\
 &\leq \norm{\br{\abs{V} - \Delta}^{\frac{1}{2}} \br{\abs{V} -
     \Delta}^{\br{\frac{\alpha - 1}{2}} \br{\theta - i t}}u}_{2} \\
 &\leq \norm{\br{\abs{V} - \Delta}^{\frac{\beta}{2}}u}_{2}.
\end{split}\end{align*}
We also have
\begin{align*}\begin{split}  
 f(1 + i t) &= \norm{\abs{V}^{\frac{\alpha}{2} + \br{\frac{\alpha -
         1}{2}} i t} \br{\abs{V} - \Delta}^{\br{\frac{\alpha - 1}{2}}
     \br{\theta - 1 - i t}}u}_{2} \\
 &= \norm{\abs{V}^{\frac{\alpha}{2}} \br{\abs{V} -
     \Delta}^{\br{\frac{\alpha - 1}{2}} \br{\theta - 1 - i t}}u}_{2} \\
 &\leq \brs{V}_{\alpha} \norm{\br{\abs{V} - \Delta}^{\frac{\alpha}{2}}
 \br{\abs{V} - \Delta}^{\br{\frac{\alpha - 1}{2}} \br{\theta - 1 - i
     t}}u}_{2} \\
&\leq \brs{V}_{\alpha} \norm{\br{\abs{V} - \Delta}^{\frac{\alpha}{2} +
    \br{\frac{\alpha - 1}{2}} \br{\theta - 1}}u}_{2} \\
&= \brs{V}_{\alpha} \norm{\br{\abs{V} - \Delta}^{\frac{\beta}{2}}u}_{2}.
\end{split}\end{align*}
The Hadamard three-lines theorem then gives the bound
\begin{equation}
  \label{eqtn:Hadamard2}
f(\theta) = \norm{\abs{V}^{\frac{\beta}{2}}u}_{2} \leq
\brs{V}_{\alpha}^{\br{\frac{\beta - 1}{\alpha - 1}}} \norm{\br{\abs{V}
  - \Delta}^{\frac{\beta}{2}}u}_{2}.
\end{equation}
A similar argument can be applied to obtain the bound
\begin{equation}
  \label{eqtn:Hadamard3}
\norm{\br{- \Delta}^{\frac{\beta}{2}}u}_{2} \leq
\brs{V}_{\alpha}^{\br{\frac{\beta - 1}{\alpha - 1}}} \norm{\br{\abs{V}
  - \Delta}^{\frac{\beta}{2}}u}_{2}
\end{equation}
for all $u \in D(\abs{V} - \Delta)$, one must simply remember that
the imaginary powers of the positive self-adjoint operator
$\br{-\Delta}$ are contraction operators on
$L^{2}\br{\R^{n}}$. Combining \eqref{eqtn:Hadamard2} and
\eqref{eqtn:Hadamard3} then gives $\brs{V}_{\beta} \leq 2
\brs{V}_{\alpha}^{\frac{\beta - 1}{\alpha - 1}}$.
\end{proof}

Recall that a non-negative potential $V \in L^{q}_{loc}(\R^{n})$ for index
$1 < q < \infty$ is said to
belong to the reverse H\"{o}lder class $RH_{q}$ if there exists $C > 0$ such that
$$
\br{\frac{1}{\abs{B}} \int_{B} V^{q} \, dx}^{\frac{1}{q}} \leq C
\br{\frac{1}{\abs{B}} \int_{B} V \, dx}
$$
for every ball $B \subset \R^{n}$. For $V \in RH_{q}$, the condition $\brs{V}_{\alpha} < \infty$ for
$\alpha = 2$ was proved in \cite{shen1995lp} for $q \geq \frac{n}{2}$ and $n \geq 3$. This
was later improved to $q
\geq 2$ and arbitrary dimension in \cite{auscher2007maximal}. Note
that in {\cite[Thm.~1.1]{auscher2007maximal}} the estimate $\norm{\Delta u}_{2} +
\norm{\abs{V} u}_{2} \lesssim \norm{(\abs{V} -\Delta)u}_{2}$ was
proved for $u \in C^{\infty}_{0}(\R^{n})$. However, this estimate can
then be extended to all of $D(\abs{V} - \Delta)$ since $\abs{V} \in
L^{2}_{loc}(\R^{n})$ implies that $C^{\infty}_{0}(\R^{n})$ is a core
for $(\abs{V} - \Delta)$ (c.f. \cite{semenov1977schrodinger}).

\begin{thm}[\cite{shen1995lp}, \cite{auscher2007maximal}]
  \label{thm:Shen2}
  Let $V \in L^{1}_{loc}(\R^{n})$ and suppose that $\abs{V} \in
  RH_{q}$ for some $q \geq 2$. Then $V \in \mathcal{W}_{2} \subset \mathcal{W}$.
\end{thm}

Notice that since the absolute value of any polynomial is contained in
$RH_{q}$ for any $q \in (1,\infty)$ we automatically obtain the following corollary.

\begin{cor}
  \label{cor:KatoPoly}
  For any polynomial $P$, we have $P \in \mathcal{W}_{2}$. As a result, the Kato estimate holds for any polynomial with range
  contained in $S_{\mu+}$ for some $\mu \in [0,\frac{\pi}{2})$.
  \end{cor}

The ensuing proposition demonstrates that
the inclusion of the reverse H\"{o}lder potentials in
$\mathcal{W}_{2}$ is strict, at least in dimension $n > 4$.

\begin{prop} 
 \label{prop:Ln2} 
 For $n > 4$,
 $$
L^{\frac{n}{2}} \br{\R^{n}} \subset \mathcal{W}_{2}.
 $$
 \end{prop}

 \begin{proof}  
   Fix $V \in L^{\frac{n}{2}}\br{\R^{n}}$. For $\varepsilon > 0$, the
   resolvent $(\abs{V} + \varepsilon - \Delta)^{-1}$ is well-defined as a bounded
   operator on $L^{2}(\R^{n})$. H\"{o}lder's inequality followed by
    the fact that $(\abs{V} + \varepsilon - \Delta)^{-1}u \leq
    (-\Delta + \varepsilon)^{-1}u$ for any $u \in L^{2}$ and then
    finally the uniform boundedness of $(-\Delta + \varepsilon)^{-1}$
    from $L^{\frac{2 n}{n - 4}}$ to $L^{2}$ (c.f. {\cite[Sec.~6.1.2]{grafakos2009modern}}) produces
   \begin{align*}\begin{split}  
 &\norm{(\abs{V}+ \varepsilon) \br{\abs{V} + \varepsilon -
     \Delta}^{-1} u}_{L^{2}(B(0,N))} \\ & \qquad \qquad \qquad \leq
 \norm{\abs{V} + \varepsilon}_{L^{\frac{n}{2}}(B(0,N))} \cdot \norm{\br{\abs{V} + \varepsilon -
     \Delta}^{-1}u}_{L^{\frac{2n}{n - 4}}(B(0,N))}  \\
 & \qquad \qquad \qquad \leq \norm{\abs{V} + \varepsilon}_{L^{\frac{n}{2}}(B(0,N))} \cdot
 \norm{\br{-\Delta + \varepsilon}^{-1}u}_{L^{\frac{2n}{n - 4}}(B(0,N))} \\
 & \qquad \qquad \qquad \lesssim \norm{\abs{V} + \varepsilon}_{L^{\frac{n}{2}}(B(0,N))} \norm{u}_{2}
\end{split}\end{align*}
for any $N > 0$, where $B(0,N) \subset \R^{n}$ is the open ball of radius $N$ centered
at the origin. Therefore,
\begin{align*}\begin{split}  
 \norm{(\abs{V} + \varepsilon)u}_{L^{2}(B(0,N))} &\lesssim \norm{\abs{V} +
  \varepsilon}_{L^{\frac{n}{2}}(B(0,N))} \cdot \norm{(\abs{V} + \varepsilon -
  \Delta)u}_{2} \\ &\leq \norm{\abs{V} + \varepsilon}_{L^{\frac{n}{2}}(B(0,N))}
\cdot \br{\norm{(\abs{V} -\Delta)u}_{2} + \varepsilon \norm{u}_{2}}
 \end{split}\end{align*}
for all $u \in D(\abs{V} + \varepsilon - \Delta) = D(\abs{V} -
\Delta)$. Letting $\varepsilon \rightarrow 0$ followed by $N
\rightarrow \infty$ gives the estimate
$$
\norm{\abs{V} u}_{2} \lesssim \norm{V}_{\frac{n}{2}} \cdot \norm{(\abs{V} - \Delta )u}_{2}
$$
for all $u \in D(\abs{V} - \Delta)$. The triangle inequality then
yields the bound $\norm{(-\Delta)u}_{2} \lesssim \norm{(\abs{V} -
  \Delta)u}_{2}$ and thus $V \in \mathcal{W}_{2}$.
 \end{proof}

 The above statements demonstrate clearly that the class of potentials
 $\mathcal{W}_{2}$ is quite large. In light of Proposition
 \ref{lem:Independence}, however, it is also evident that $\mathcal{W}_{2}$ is the smallest
 class out of the collection $\lb \mathcal{W}_{\alpha}
 \rb_{\alpha \in (1,2]}$. One can then neatly surmise that
 $\mathcal{W}$, the class of all potentials for which the Kato
 estimate holds, is certain to be quite large itself.

 \subsection{Systems with Zero-Order Potential}
 \label{subsec:Systems}

 Fix $m \in \N^{*}$ and $A \in L^{\infty}\br{\R^{n}; \mathcal{L}
   \br{\C^{n} \otimes \C^{m}}}$.
 Let $V : \R^{n} \rightarrow
 \mathcal{L} \br{\C^{m}}$ be a measurable matrix-valued function
 contained in $L^{1}_{loc}(\R^{n};\mathcal{L}(\C^{m}))$ with $V(x)$
 normal for almost every $x \in \R^{n}$.  $V$ can
 be viewed as a densely defined closed multiplication operator on
 $L^{2} \br{\R^{n};\C^{m}}$ with domain
 $$
D \br{V} = \lb u \in L^{2} \br{\R^{n};\C^{m}} : V \cdot u \in L^{2} \br{\R^{n};\C^{m}} \rb.
$$
Similar to the scalar case, one can define forms $\mathfrak{l}^{A}$
and $\mathfrak{l}_{V}^{A}$ respectively through
$$
\mathfrak{l}^{A}\brs{u,v} := \int_{\R^{n}} \langle A(x) \nabla u(x),
\nabla v(x) \rangle_{\C^{n} \otimes \C^{m}} \, dx
$$
for $u, \, v \in H^{1}\br{\R^{n};\C^{m}}$ and
$$
\mathfrak{l}_{V}^{A} \brs{u',v'} := \mathfrak{l}^{A} \brs{u',v'} +
\int_{\R^{n}} \langle V(x) u'(x), v'(x) \rangle_{\C^{m}} \, dx
$$
for $u'$ and $v'$ contained in
$$
H^{1,V} \br{\R^{n};\C^{m}} := H^{1} \br{\R^{n};\C^{m}} \cap D(\abs{V}^{\frac{1}{2}}),
$$
where $\abs{V(x)} := \sqrt{V(x)^{*}V(x)}$ for $x \in
\R^{n}$. The density of $H^{1,V}(\R^{n};\C^{m})$ in
$L^{2}(\R^{n};\C^{m})$ follows from the fact that
$C^{\infty}_{0}(\R^{n};\C^{m}) \subset H^{1,V}(\R^{n};\C^{m})$.
Assume that the forms $\mathfrak{l}^{A}$ and
$\mathfrak{l}_{V}^{A}$ satisfy the system equivalents of the G{\aa}rding inequalities \eqref{eqtn:Garding0} and
\eqref{eqtn:Garding} with constants $\kappa^{A} > 0$ and
$\kappa_{V}^{A} > 0$
respectively. Then $\mathfrak{l}^{A}$ and $\mathfrak{l}_{V}^{A}$ will
both have a unique associated maximal accretive operator, $L$
and $L + V$.

In the article \cite{auscher2001kato}, the Kato square
root property $\norm{\sqrt{L} u}_{2} \simeq \norm{\nabla u}_{2}$ was proved
for elliptic systems without potential. Using the non-homogeneous
machinery that we have developed, the corresponding property can be
proved for the operator with potential. Define $\brs{V}_{\alpha}$ and
$\mathcal{W}_{\alpha}(\R^{n};\mathcal{L}(\C^{m}))$ for $\alpha \in [1,2]$ to be the system
analogues of the corresponding scalar objects. In the below theorem, our non-homogeneous
framework will be applied to determine the domain of $\sqrt{L + V}$ for normal potentials.

\begin{thm} 
  \label{thm:Systems}
  Suppose that $V \in \mathcal{W}_{\alpha}(\R^{n};
  \mathcal{L}(\C^{m}))$ for some $\alpha \in (1,2]$ and that $V(x)$ is a normal matrix for almost every $x \in
  \R^{n}$.  Suppose that the system equivalents of the G{\aa}rding inequalities \eqref{eqtn:Garding0} and
\eqref{eqtn:Garding} are both satisfied with constants $\kappa^{A}$
and $\kappa_{V}^{A}$ respectively.
  Then there must exist some $C_{V} >
0$ such that
$$
C_{V}^{-1} \br{\norm{\abs{V}^{\frac{1}{2}} u}_{2} +
  \norm{\nabla u}_{2}} \leq \norm{\sqrt{L + V} u}_{2} \leq C_{V} \br{\norm{\abs{V}^{\frac{1}{2}} u}_{2} + \norm{\nabla u}_{2}}
$$
for all $u \in H^{1,V}\br{\R^{n};\C^{m}}$. Moreover, the constant
$C_{V}$ depends on $V$ and $\alpha$ through
$$
C_{V} = \tilde{C}_{V} (\alpha - 1)^{-1} (1 + \brs{V}_{\alpha}^{2}),
$$
where $\tilde{C}_{V}$ only depends on $V$ through $\kappa_{A}^{V}$ and
is independent of $\alpha$.
 \end{thm}

 \begin{proof}
   The polar decomposition theorem asserts the existence of some $U :
   \R^{n} \rightarrow \mathcal{L}(\C^{m})$, with $U(x)$ unitary for all
   $x \in \R^{n}$, such that
   $$
V(x) = U(x) \abs{V(x)}
$$
for all $x \in \R^{n}$. As
$V(x)$ is normal, the matrices $U(x)$ and $\abs{V(x)}$ are well-known
to commute. We therefore have the decomposition
\begin{equation}
  \label{eqtn:PolarDecomp}
V(x) = \abs{V(x)}^{\frac{1}{2}} U(x) \abs{V(x)}^{\frac{1}{2}}
\end{equation}
for almost every $x \in \R^{n}$.
Set
$$
D := \nabla : H^{1}\br{\R^{n};\C^{m}} \subset L^{2} \br{\R^{n} ; \C^{m}} \rightarrow
L^{2}\br{\R^{n};\C^{n} \otimes \C^{m}}
$$
and
$$
J := \abs{V}^{\frac{1}{2}} : D(\abs{V}^{\frac{1}{2}}) \subset
L^{2}\br{\R^{n};\C^{m}} \rightarrow L^{2} \br{\R^{n};\C^{m}},
$$
both defined as operators on $L^{2} \br{\R^{n};\C^{m}}$. Define the
perturbation matrices
$$
B_{1} := I \qquad and \qquad B_{2} = \br{\begin{array}{c c c} 
                                           I & 0 & 0 \\
                                           0 & U & 0 \\
                                           0 & 0 & A
 \end{array}}.
$$
It is not too difficult to see that the operators $\lb \Gamma_{0},
B_{1}, B_{2} \rb$ will satisfy conditions (H1) - (H8) and $\lb
\Gamma_{J}, B_{1}, B_{2} \rb$ will satisfy (H1) - (H6). Indeed, the only
non-trivial condition for both sets of operators is (H2) and this
follows from the respective G{\aa}rding inequalities
\eqref{eqtn:Garding0} and \eqref{eqtn:Garding}. It is also clear from
the fact that $V \in \mathcal{W}_{\alpha}(\R^{n}; \mathcal{L}(\C^{m}))$ that (H8D$\alpha$) and
(H8J$\alpha$) will both be satisfied. This follows from an argument
identical to that of Lemma \ref{lem:WAlphaImplies}.  The Kato estimate then follows
from Corollary \ref{cor:LBJCalc} with constant $\tilde{C}_{V} (\alpha
- 1)^{-1}(1 + \brs{V}_{\alpha}^{2})$. It should be noted that
\eqref{eqtn:PolarDecomp} was needed so that we would have $L^{J}_{B} =
L + V$.
 \end{proof}

 In analogy to the scalar case, it is quite likely that a similar reverse
H\"{o}lder type condition will be sufficient to imply the boundedness
of the operator $\abs{V} \br{\abs{V} - \Delta}^{-1}$ for $m > 1$. However,
as far as the author is aware, this is still an open problem for $m > 1$. What is
apparent is that the condition that the potential belongs to
$L^{\frac{n}{2}}$ will once again be sufficient to imply that it
belongs to $\mathcal{W}_{2}$. The following proposition has an
identical proof to that of Proposition \ref{prop:Ln2}.

\begin{prop}
  \label{prop:SystemsLn2}
  For $n > 4$ and $m \in \N^{*}$,
  $$
L^{\frac{n}{2}}\br{\R^{n} ; \mathcal{L}\br{\C^{m}}} \subset
\mathcal{W}_{2}(\R^{n}; \mathcal{L}(\C^{m})).
  $$
\end{prop}

 \subsection{First Order Potentials}
 \label{subsec:FirstOrder}

 Let $b : \R^{n} \rightarrow \C^{n}$ be contained in
 $L^{1}_{loc}\br{\R^{n};\C^{n}}$ and $A \in L^{\infty}
 \br{\R^{n};\mathcal{L}(\C^{n})}$. Suppose that $A$ satisfies the standard G{\aa}rding inequality
$$
\mathrm{Re} \int_{\R^{n}} \langle A(x) \nabla u(x), \nabla u(x)
\rangle_{\C^{n}} \, dx \geq \kappa^{A} \cdot \norm{\nabla u}_{2}^{2}
$$
for all $u \in H^{1}\br{\R^{n};\C}$, for some $\kappa^{A} > 0$. Consider the accretive sesquilinear form
 $$
\mathfrak{h}^{A}_{b}[u,v] := \langle A \nabla u, \nabla v \rangle_{2}
+ \langle (\nabla + b) u, (\nabla + b)v \rangle_{2}
$$
defined on the dense subspace $H^{1}_{b}(\R^{n}) := \lb u \in
H^{1}(\R^{n}) : \abs{b u} \in L^{2}(\R^{n}) \rb
\subset L^{2}(\R^{n})$. 

\begin{lem}
  The form $\mathfrak{h}^{A}_{b}$ is both continuous and closed.
\end{lem}

\begin{proof}  
  Let's first prove that $\mathfrak{h}^{A}_{b}$ is continuous. It must
  be shown that $\abs{\mathfrak{h}^{A}_{b}[u,v]} \lesssim
  \norm{u}_{\mathfrak{h}^{A}_{b}} \norm{v}_{\mathfrak{h}^{A}_{b}}$ for
  all $u, \, v \in H^{1}_{b}$,
  where $\norm{u}_{\mathfrak{h}^{A}_{b}} := \sqrt{\mathrm{Re} \br{
    \mathfrak{h}^{A}_{b}[u,u]} + \norm{u}_{2}^{2}}$. On consecutively
  applying the boundedness of $A$ and then the accretivity of $A$,
  \begin{align*}\begin{split}  
 \abs{\mathfrak{h}^{A}_{b}[u,v]} &= \abs{\langle A \nabla u, \nabla v
   \rangle_{2} + \langle (\nabla + b)u, (\nabla + b) v \rangle_{2}} \\
 &\lesssim \norm{\nabla u}_{2} \norm{\nabla v}_{2} + \norm{(\nabla +
   b)u}_{2} \norm{(\nabla + b) v}_{2} \\
 &\lesssim \sqrt{\mathrm{Re} \, \langle A \nabla u, \nabla u
   \rangle_{2}} \sqrt{\mathrm{Re} \, \langle A \nabla v, \nabla v
   \rangle_{2}} + \norm{(\nabla + b)u}_{2} \norm{(\nabla + b)v}_{2} \\
 &\lesssim \norm{u}_{\mathfrak{h}^{A}_{b}} \norm{v}_{\mathfrak{h}^{A}_{b}}.
\end{split}\end{align*}
Let's now prove that $\mathfrak{h}^{A}_{b}$ is closed. That is, it
must be proved that $H^{1}_{b}(\R^{n})$ is complete under the norm
$\norm{\cdot}_{\mathfrak{h}^{A}_{b}}$. From the boundedness of $A$,
\begin{align*}\begin{split}  
 \norm{u}_{\mathfrak{h}^{A}_{b}} &= \sqrt{\mathrm{Re}
   \br{\mathfrak{h}^{A}_{b}[u,u]} + \norm{u}_{2}^{2}} \\
 &= \sqrt{\mathrm{Re} \br{\langle A \nabla u, \nabla u \rangle_{2}} +
   \norm{(\nabla + b)u}_{2}^{2} + \norm{u}_{2}^{2}} \\
 &\lesssim \sqrt{\norm{\nabla u}^{2}_{2} + \norm{b u}^{2}_{2} +
   \norm{u}_{2}^{2}} \\
 &\lesssim \norm{u}_{H^{1}_{b}},
\end{split}\end{align*}
where $\norm{u}_{H^{1}_{b}} := \norm{u}_{2} + \norm{\nabla u}_{2} +
\norm{b u}_{2}$. Conversely, the G{\aa}rding inequality for $A$ implies
that
\begin{align*}\begin{split}  
 \norm{u}_{\mathfrak{h}^{A}_{b}} &= \sqrt{\mathrm{Re} \br{\langle A
     \nabla u, \nabla u \rangle_{2}} + \norm{(\nabla + b)u}_{2}^{2} +
   \norm{u}_{2}^{2}} \\
 &\gtrsim \norm{\nabla u}_{2} + \norm{(\nabla + b) u}_{2} + \norm{u}_{2} \\
 &\geq \norm{\nabla u}_{2} + \frac{1}{2}\norm{(\nabla + b)u}_{2} +
 \norm{u}_{2} \\
 &\geq \norm{\nabla u}_{2} + \frac{1}{2}\norm{b u}_{2} - \frac{1}{2}
 \norm{\nabla u}_{2} + \norm{u}_{2} \\
 &\simeq \norm{u}_{H^{1}_{b}}.
\end{split}\end{align*}
This shows that the norm $\norm{\cdot}_{\mathfrak{h}^{A}_{b}}$ is
equivalent to $\norm{\cdot}_{H^{1}_{b}}$ on $H^{1}_{b}$. Since
$H^{1}_{b}$ is known to be complete under the norm $\norm{\cdot}_{H^{1}_{b}}$ it then follows that it must
also be complete under $\norm{\cdot}_{\mathfrak{h}^{A}_{b}}$.
 \end{proof}

The previous lemma implies, in particular, that there exists a maximal
accretive operator $\mathcal{L}^{A}_{b}$ associated with this form
(c.f. {\cite[Sec.~1.2]{ouhabaz2005analysis}}). This operator will be denoted by
$$
\mathcal{L}^{A}_{b} = (\nabla + b)^{*} (\nabla + b) - \mathrm{div}(A \nabla).
$$
 Define the Hilbert space to be
 $$
\mathcal{H} := L^{2} \br{\R^{n}} \oplus L^{2} \br{\R^{n};\C^{n}}
\oplus L^{2} \br{\R^{n};\C^{n}}.
$$
Then set
$$
J := \nabla + b : L^{2}\br{\R^{n}; \C} \rightarrow L^{2} \br{\R^{n};
  \C^{n}} \quad and \quad D := \nabla : L^{2}\br{\R^{n} ; \C}
\rightarrow L^{2}\br{\R^{n};\C^{n}}.
$$
Also, let $B_{1} = I$ as usual and
$$
B_{2} = \br{\begin{array}{c c c} 
              I & 0 & 0 \\
              0 & I & 0 \\
              0 & 0 & A
 \end{array}}.
$$
Then the operator $L_{B}^{J}$ as in Corollary \ref{cor:LBJCalc} is  given
by $L_{B}^{J} = \mathcal{L}^{A}_{b}$. Since $A$ satisfies the standard
G{\aa}rding inequality it follows that $\lb \Gamma_{0}, B_{1},
B_{2} \rb$ and $\lb \Gamma_{J}, B_{1}, B_{2} \rb$ will both satisfy
(H2). This in turn implies that $\lb \Gamma_{0}, B_{1}, B_{2} \rb$
satisfies (H1) - (H8) and $\lb \Gamma_{J}, B_{1}, B_{2} \rb$ satisfies
(H1) - (H6). The non-homogeneous framework, in the form of Corollary \ref{cor:LBJCalc}, applied to these operators
then produces the following theorem.

\begin{thm} 
 \label{thm:FirstOrder2} 
 Suppose that $D \br{\br{\nabla + b}^{*} \br{\nabla + b} -
  \Delta} \subset D (-\Delta) \cap D(\br{\nabla + b}^{*} \br{\nabla + b})$ and there exists some $c_{b} > 0$ such that
 \begin{equation}
   \label{eqtn:FirstOrder2}
\norm{\Delta u}_{2} \leq c_{b} \norm{\brs{\br{\nabla + b}^{*} \br{\nabla +
    b} - \Delta}u}_{2}
\end{equation}
for all $u \in D \br{\br{\nabla + b}^{*} \br{\nabla + b} -
  \Delta}$. Then there exists some constant $c > 0$, independent of
$b$, for which
$$
\br{c \cdot \br{1 + c_{b}}}^{-2} \br{\norm{\br{\nabla + b} u}_{2} + \norm{\nabla u}_{2}} \leq
\norm{\sqrt{\mathcal{L}^{A}_{b}}u}_{2} \leq \br{c \cdot \br{1 + c_{b}}}^{2}
\br{\norm{\br{\nabla + b} u}_{2} + \norm{\nabla u}_{2}}
$$
for all $u \in H^{1}_{b}(\R^{n})$.
\end{thm}

To see that the above theorem is true, simply note that
\eqref{eqtn:FirstOrder2} implies both (H8D$\alpha$) and (H8J$\alpha$) for
$\alpha = 2$ in this context. The independence of
the constant $c$ from $b$ follows from the fact that (H2) is satisfied
by $\lb \Gamma_{J}, B_{1}, B_{2} \rb$ with constant independent of $b$.

\section{Literature Comparison}
\label{sec:Literature}
 
It is important to note that this is not the first time that Kato-type
estimates have been studied for non-homogeneous operators. We will now
take some time to outline how our article differs in techniques and
results from each of these previous forays.

Recently, in \cite{gesztesy2015stability} and
\cite{gesztesy2016stability}, F. Gesztesy, S. Hofmann and R. Nichols
studied the domains of square root operators using techniques
distinct from those developed in \cite{axelsson2006quadratic}. The
article \cite{gesztesy2015stability} considers potentials in the class
$L^{p} + L^{\infty}$ but is not directly
relevant since it considers bounded
domains. On the other hand, \cite{gesztesy2016stability} does not
impose a boundedness assumption on the domain and considers the
potential class $L^{\frac{n}{2}} + L^{\infty}$. There is already an immediate comparison with our
potential class since it was
shown in Proposition \ref{prop:Ln2} that $L^{\frac{n}{2}} \subset
\mathcal{W}$ in dimension $n > 4$. It is not immediately clear whether
$L^{\infty}$ is contained within our class.

Axelsson, Keith and McIntosh themselves considered non-homogeneous
operators on Lipschitz domains with mixed boundary conditions in
\cite{axelsson2006kato}. The potentials that they considered were, however,
bounded both from above and below and thus contained in $RH_{2}
\subset \mathcal{W}$. In \cite{egert2014kato} and
\cite{egert2016kato}, M. Egert, R. Haller-Dintelmann and
P. Tolksdorf generalised this to certain non-smooth
domains. 

The articles \cite{axelsson2006kato},
\cite{egert2014kato} and \cite{egert2016kato} are built upon the
original AKM framework, similar to this one.  A key step in the original proof of the
AKM framework is the proof of the estimate
\begin{equation}
  \label{eqtn:GlobalAtPt}
\int^{\infty}_{0} \norm{\br{A_{t} - P_{t}}u}^{2} \frac{dt}{t} \lesssim
\norm{u}^{2}.
\end{equation}
This estimate allows for the $A_{t}$ and $P_{t}$ operators to be
freely interchanged at several stages in the proof granting use of
some of the more enviable properties of the $A_{t}$
operator. As has been demonstrated in this article through the
diagonalisation theorem, Theorem \ref{thm:Diagonalisation},
\eqref{eqtn:GlobalAtPt} will not hold for a general potential. The
articles \cite{axelsson2006kato}, \cite{egert2014kato} and \cite{egert2016kato} circumvent this problem by imposing boundedness
of the potential from above and below. The boundedness of the
potential from above and below allows one to absorb the potential into
the multiplicative perturbation $B_{2}$ so that you are instead considering
the operators
$$
\Gamma_{1} = \br{\begin{array}{c c c} 
                   0 & 0 & 0 \\
                   1 & 0 & 0 \\
                   \nabla & 0 & 0
 \end{array}}, \quad B_{1} = I \quad and \quad B_{2} =
\br{\begin{array}{c c c} 
      I & 0 & 0 \\
      0 & V & 0 \\
      0 & 0 & A
 \end{array}}.
$$
For this set of operators, the large time-scale estimate
$$
\int^{\infty}_{1} \norm{\Theta_{t}^{B,1} P_{t}^{1} u}^{2} \frac{dt}{t}
\lesssim \int^{\infty}_{1} \norm{P_{t}^{1} u}^{2} \frac{dt}{t}
\lesssim \int^{\infty}_{1} \norm{t \Pi_{1} P_{t}^{1} u}^{2}
\frac{dt}{t} \lesssim \norm{u}^{2}
$$
for $u \in R (\Gamma_{1})$ follows almost trivially. Then one only requires a small time-scale
version of  \eqref{eqtn:GlobalAtPt} to
hold, namely
$$
\int^{1}_{0} \norm{\br{A_{t} - P_{t}^{1}}u}^{2}
\frac{dt}{t} \lesssim \norm{u}^{2}
$$
for all $u \in R \br{\Gamma_{1}}$. Such an
estimate is then proved to be true.

This is a crude explanation
as to why the techniques developed in \cite{axelsson2006kato} cannot be
directly applied for a general potential that is not bounded both from
above and below.  There are similar obstructions, for example in
the selection of test functions in the Carleson measure
proof. However, these also disappear when the potential is bounded both
from above and below.

In this paper, our method has been to instead treat the potential as
part of the unperturbed operator and then to subsequently exploit the algebraic
structure of the operators $\Gamma_{\abs{V}^{\frac{1}{2}}}$, $B_{1}$ and
$B_{2}$. This exploitation has allowed us to conclude that the estimate
 \eqref{eqtn:GlobalAtPt} will at least hold on the third
 component which, it turns out, is all that is required. Similar obstructions in the proof of the main square
 function estimate also vanish when
 considered component-wise. As a consequence of this three-by-three mindset we have been able to obtain square
function estimates for potentials that aren't necessarily bounded from
above or below and, moreover, are not contained in
$L^{p}\br{\R^{n}}$ for any $1 \leq p \leq \infty$.

\bibliographystyle{siam}
\bibliography{C:/Users/julian/Desktop/TeX/texmf/bibtex/bibmain}

\end{document}